\documentclass[a4paper, 11pt, reqno]{amsart}
\usepackage[margin=1.2in]{geometry}
\usepackage[usenames,dvipsnames]{color}
\usepackage{amssymb, amsmath, latexsym, graphics, graphicx,caption}
\usepackage{tikz}
\usepackage{caption, subcaption}
\usepackage{verbatim}
\usepackage{hyperref}
\usepackage{enumerate}
\usepackage{longtable}
\usepackage{array}

\newtheorem{theorem}{Theorem}[section]
\newtheorem{lemma}[theorem]{Lemma}
\newtheorem{corollary}[theorem]{Corollary}
\newtheorem{proposition}[theorem]{Proposition}

\theoremstyle{definition}
\newtheorem{definition}[theorem]{Definition}
\newtheorem{example}[theorem]{Example}
\newtheorem{question}[theorem]{Question}

\newtheorem{conjecture}[theorem]{Conjecture}
\newtheorem{remark}[theorem]{Remark}
\makeatother

\def\L{\,{\mathcal L}\,}
\def\R{\,{\mathcal R}\,}
\def\Lless{\,\leq_{\mathcal L}\,}
\def\Rless{\,\leq_{\mathcal R}\,}
\def\D{\,{\mathcal D}\,}
\def\H{\,{\mathcal H}\,}
\def\J{\,{\mathcal J}\,}

\def\Ls{\,{\mathcal L}^\ast\,}
\def\Rs{\,{\mathcal R}^\ast\,}
\def\Ds{\,{\mathcal D}^\ast\,}
\def\Hs{\,{\mathcal H}^\ast\,}

\def\Lt{\,\widetilde{\mathcal L}\,}
\def\Rt{\,\widetilde{\mathcal R}\,}
\def\Dt{\,\widetilde{\mathcal D}\,}
\def\Ht{\,\widetilde{\mathcal H}\,}
\def\Jt{\,\widetilde{\mathcal J}\,}

\newcommand\Ltu[1]{\,\widetilde{\mathcal L}_{#1}\,}
\newcommand\Rtu[1]{\,\widetilde{\mathcal R}_{#1}\,}
\newcommand\Dtu[1]{\,\widetilde{\mathcal D}_{#1}\,}
\newcommand\Htu[1]{\,\widetilde{\mathcal H}_{#1}\,}

\newcommand\trop{\mathbb{T}} 
\newcommand\Qpos{\mathbb{Q}_{\geq0}} 
\newcommand\bool{\mathbb{B}} 
\newcommand\linz{\mathfrak{L}} 
\newcommand\lin{\mathfrak{L}^*} 

\newcommand\mat[2]{M_{#1}(#2)}  
\newcommand\upper[2]{UT_{#1}(#2)} 
\newcommand\uni[2]{U_{#1}(#2)} 
\newcommand\gl[2]{{\rm GL}_{#1}(#2)} 
\newcommand\full[2]{\overline{UT}_{#1}(#2)} 

\newcommand\lid[1]{#1^{(+)}}
\newcommand\rid[1]{#1^{(\ast)}}
\newcommand\lidlid[1]{#1^{(+)(+)}}

\newcommand\diag[1]{{\rm{D}}_#1}
\newcommand\rnorm[1]{#1^{\diamond}}
\newcommand\lnorm[1]{#1^{\bullet}}

\newcommand{\Def}{\operatorname{Def}}
\newcommand{\Supp}{\operatorname{Supp}}

\newcommand{\rank}{\operatorname{rank}}

\newcommand{\1}{\mathbf{1}}
\newcommand{\0}{\mathbf{0}}

\begin{document}
\title{Matrix semigroups over semirings}
\author{Victoria Gould}
\address[V. Gould]{Department of Mathematics, University of York, Heslington, York YO10 5DD, UK}
\email[V. Gould]{victoria.gould@york.ac.uk}
\author{Marianne Johnson}
\address[M. Johnson]{School of Mathematics, University of Manchester, Manchester M13 9PL, UK}
\email[M. Johnson]{marianne.johnson@manchester.ac.uk}
\author{Munazza Naz}
\address[M. Naz]{Department of Mathematical Sciences,
Fatima Jinnah Women University,
The Mall, Rawalpindi,
Pakistan
}
\email[M. Naz]{munazzanaz@fjwu.edu.pk}

\date{\today}

\thanks{The third author thanks the Schlumberger Foundation for its generous support in the form of a Faculty for the Future Postdoctoral Fellowship, and the Department of Mathematics of the University of York for hosting her during the period of her Fellowship.}

\maketitle

\begin{abstract} The multiplicative semigroup $M_n(F)$ of $n\times n$ matrices over a field $F$ is well understood, in particular, it is a regular semigroup. This paper considers semigroups of the form $M_n(S)$, where $S$ is a semiring, and the subsemigroups 
$UT_n(S)$ and $U_n(S)$ of $M_n(S)$ consisting  of upper triangular and unitriangular matrices.  Our main interest is in the case where $S$ is an idempotent semifield, where  we also consider the subsemigroups
 $UT_n(S^*)$ and $U_n(S^*)$  consisting of those matrices of $UT_n(S)$ and $U_n(S)$ having   all elements on and above the leading diagonal  non-zero. Our guiding examples
 of such $S$ are the 2-element Boolean semiring $\mathbb{B}$ and the tropical semiring $\mathbb{T}$. In the first case, $M_n(\mathbb{B})$ is isomorphic to the semigroup of binary relations on an $n$-element set, and in the second, $M_n(\mathbb{T})$ 
is the semigroup of $n\times n$ tropical matrices.

It is well known that the subsemigroup of $M_n(F)$ consisting of the singular matrices is idempotent generated. We begin consideration of the analogous questions for $M_n(S)$ and its subsemigroups. In particular we show that the idempotent generated subsemigroup of $UT_n(\mathbb{T}^*)$ is precisely the semigroup $U_n(\mathbb{T}^*)$. 

 Il'in has proved that for any semiring $R$ and   $n>2$, the semigroup $M_n(R)$ is regular if and only if $R$ is a regular {\em ring}. We therefore base our investigations for   $M_n(S)$ and its subsemigroups on the  analogous but weaker concept of 
 being Fountain (formerly, weakly abundant). 
 These notions are determined by the existence and behaviour of idempotent left and right identities for elements, lying in particular equivalence classes.
  We show that certain subsemigroups of $M_n(S)$, including several generalisations of well-studied monoids of binary relations (Hall relations, reflexive relations, unitriangular Boolean matrices), are {Fountain}. 
 We give a detailed study of a  family of Fountain semigroups arising in this way that has
 particularly interesting and unusual properties. 
 \end{abstract}

\section{Introduction}
\label{sec:intro}

The ideal structure of the multiplicative semigroup $M_n(F)$ of all $n \times n$ matrices over a field   $F$ is simple to understand. The semigroup $M_n(F)$  is   regular,
that is, for all $A\in M_n(F)$ there is a $B\in M_n(F)$ such that $A=ABA$,  and  possesses precisely $n+1$ ideals. These are the principal ideals  $I_k=\{ A\in M_n(F):\rank A\leq k\}$ where $0\leq k\leq n$.  Clearly
$\{ 0\}=I_0\subset I_1\subset\hdots\subset I_n=M_n(F)$; the resulting (Rees) quotients
$I_k/I_{k-1}$ for $1\leq k\leq n$ have a particularly pleasing structure, as we now explain. To do so, we use the language of Green's relations $\L,\R$ and $\J$ on a semigroup $S$, where two elements are $\L$-related ($\R$-related, $\J$-related) if and only if they  generate the same principal left (right, two-sided) ideal; correspondingly, the 
$\L$-classes, ($\R$-classes, $\J$-classes) are partially ordered by inclusion of left
(right, two-sided) ideals. We give more details of Green's relations in Section~\ref{sec:prelim}.

The matrices of each fixed rank $k$ form a single $\mathcal{J}$-class, so that the $\mathcal{J}$-order corresponds to the natural order on ranks, and  
the $\mathcal{L}$- and $\mathcal{R}$-  orders correspond to containment of row and column spaces,  respectively.   We have that 
$\J$ is the join $\D$ of $\L$ and $\R$, so that,  since $\D=\R\circ \L$ for any semigroup,
it follows that 
the non-zero elements of the quotients $I_k/I_{k-1}$ (for $1\leq k\leq n$) may be co-ordinatised by their $\L$-class, their $\R$-class and  a maximal subgroup of
$\rank k$ matrices. Moreover, each such subgroup is isomorphic to the general linear group $\gl{k}{F}$; (see \cite{Okninski} for further details and results). 

The subset $\upper{n}{F}$ of all upper triangular matrices (that is, those with all entries below the main diagonal equal to $0$) forms a submonoid of $\mat{n}{F}$. This submonoid is not regular, but satisfies a weaker regularity property called \emph{abundance}. The
subset $\uni{n}{F}$ of  unitriangular matrices (that is, those upper triangular matrices with all diagonal entries equal to $1$) forms a subgroup of the group of units $\gl{n}{F}$ of $\mat{n}{F}$. In the case where $F$ is an algebraically closed field, the semigroups $\mat{n}{F}$, $\upper{n}{F}$ and $\uni{n}{F}$ are examples of \emph{linear algebraic monoids} \cite{Putcha, Renner} and have been extensively studied. Certain important examples of linear algebraic monoids, including $M_n(F)$, have the property that the subsemigroup of singular elements is idempotent generated \cite{erdos,Putcha}.

Motivated by the above, in this paper we consider $n \times n$ matrices with entries in a \emph{semiring} $S$, with a focus on  {\em idempotent semifields}. The operation of matrix multiplication, defined in the usual way with respect to the operations of $S$, yields  semigroups $\mat{n}{S}, \upper{n}{S}$ and $\uni{n}{S}$
analogous to those above. Our motivating examples are that of the Boolean semiring $\mathbb{B}$ and the tropical semiring
$\mathbb{T}$.   In the first case, it is well known that $M_n(\mathbb{B})$ is 
isomorphic to the monoid of all binary relations on an $n$-element set, under composition of relations. In the second case, $M_n(\mathbb{T})$ is the monoid of $n\times n$ {\em tropical matrices}, which are a source of significant  interest (see, for example, \cite{Gaubert96,IJK16,IJK18}). Many of the tools which apply in the field case, such as arguments involving rank and invertible matrices, do not immediately carry over to the more general setting of semirings, and we are required to develop largely new strategies.  However, since $\mat{n}{S}$ can be identified with the semigroup of endomorphisms of the free module $S^n$, one can phrase several of its structural properties in terms of finitely generated submodules of $S^n$. For a general (semi)ring $S$, the relationships between such modules can be much more complicated than the corresponding situation for fields (where it is easy to reason with finite dimensional vector spaces), and as a result one finds that the structure of $\mat{n}{S}$ can be highly complex. 

It is known that the semigroup $\mat{n}{R}$ over a ring $R$ is regular if and only if $R$ is von Neumann regular\footnote{That is, the multiplicative semigroup of $R$ is regular.} (see \cite{Brown, vonNeumann}). Il'in \cite{Ilin} has generalised this result to the setting of semirings, providing a necessary and sufficient condition for $\mat{n}{S}$ to be regular; for all $n\geq 3$ the criterion is simply that $S$ is a von Neumann regular \emph{ring}.  Given that our  work concerns semigroups of matrices over  a {\em semiring} $S$,  we consider two natural generalisations of regularity, namely abundance and ``Fountainicity'' (also known as weak abundance, or semi-abundance; the term Fountain having been recently introduced by Margolis and Steinberg \cite{MS17} in honour of John Fountain's work in this area). We do this first in the context of the semigroups $\mat{n}{S}$, $\upper{n}{S}$ and $\uni{n}{S}$.   Later, in the case where $S$ is an idempotent semifield,  we make a careful study of these properties for the semigroups  $UT_n(S^*)$ and $U_n(S^*)$, where these are the subsemigroups of $\upper{n}{S}$ and $\uni{n}{S}$, respectively, in which all the entries above the leading diagonal are non-zero.  
Abundance and Fountainicity are determined by properties of the relations $\Ls$ and
$\Rs$ (for abundance) and $\Lt$ and $\Rt$ (for Fountainicity), where $\Ls$ is a natural extension of $\L$, and $\Lt$ a natural extension of $\Ls$, similar statements being true for the dual relations. It is worth remarking that the property of abundance may be phrased in terms of projectivity of monogenic acts over monoids \cite{kilp}.

For an idempotent semifield $S$, we describe two functions $^{(+)}: M_n(S)\rightarrow E(M_n(S))$ and $^{(*)} :M_n(S) \rightarrow E(M_n(S))$ which map  $A\in M_n(S) $ to
left and right identities $A^{(+)}$ and $A^{(*)}$ for $A$, respectively.  These maps are used to show that certain subsemigroups of $M_n(S)$, including several generalisations of well-studied monoids of binary relations (Hall relations, reflexive relations, unitriangular Boolean matrices), are {Fountain}. Many of the semigroups under consideration turn out to have the stronger property that each $\Rt$-class and each $\Lt$-class contains a \emph{unique} idempotent. Fountain semigroups whose idempotents commute are known to have this property, however examples of such semigroups with non-commuting idempotents are more elusive. In this paper we provide examples of such semigroups which occur naturally in our setting, and give a detailed study of some particularly interesting families of these.

In Section \ref{sec:prelim} we provide the necessary background on semirings, semimodules and generalised regularity conditions, including the details of Green's relations and their extensions mentioned above.  In Section \ref{sec:reg} we show that each of the semigroups $\mat{n}{S}, \upper{n}{S}$ and $\uni{n}{S}$ admits a natural left-right symmetry and provide module theoretic characterisations of $\R,\L,\Rs,\Ls,\Rt$ and $\Lt$ and hence of regularity, abundance and Fountainicity in the case of $\mat{n}{S}$. In contrast to Il'in's result that
$M_n(S)$ is regular if and only if $S$ is a regular ring \cite{Ilin} we show that 
$UT_n(S)$ is regular if and only if $n=1$ and the multiplicative semigroup of $S$ is a regular semigroup ({\bf Proposition~\ref{notReg}}). Many of the semirings we study in this article are exact \cite{WJK}; in the case where $S$ is exact, $M_n(S)$ is abundant if and only if it is regular ({\bf Theorem~\ref{exact}}). Thus in large part in this article we concentrate on Fountainicity. 

From Section~\ref{sec:semfield} onwards we focus our attention soley on the case where $S$ is an idempotent semifield. We make use of a particular idempotent matrix construction to prove that several natural subsemigroups of $\mat{n}{S}$, including $\uni{n}{S},$ ({\bf Corollary~\ref{uniS}}) are Fountain.  Specifically, for every
$A\in M_n(S)$ we associate idempotents $A^{(+)}$ and $A^{(*)}$ in $M_n(S)$; these idempotents, in certain situations, play the role of $AB$ and $BA$ for a regular matrix $A$ with $A=ABA$.  Section~\ref{sec:semfield} ends by considering the behaviour of idempotents in $UT_n(S^*)$. Idempotent generated semigroups are of fundamental importance in algebraic semigroup theory (see, for example, \cite{howie66,nambooripad,Putcha}).  We show that the idempotent generated subsemigroup of
$UT_n(S^*)$ is precisely $U_n(S^*)$  ({\bf Corollary~\ref{idmpgennozero}}).

Every idempotent semifield can be constructed from a lattice ordered abelian group
$\linz^*$, by adjoining a zero element and taking addition to be least upper bound: the resulting semifield is denoted by $\linz$ and (with some abuse of notation) we use
$\linz^*$ to denote its non-zero elements. In Section \ref{sec:lin} we specialise to the case where the natural partial order on $\linz$ is total. For the Boolean semiring $\mathbb{B}$ we provide a complete description of the generalised regularity properties of $\mat{n}{\bool}$, $\upper{n}{\bool}$and $\uni{n}{\bool}$, making use of our idempotent construction ({\bf Theorem~\ref{thm:B}}),  and provide a partial description for 
 $\mat{n}{\linz}$, $\upper{n}{\linz}$ and $\uni{n}{\linz}$.   
 
 For the remainder of the paper we consider one family of Fountain subsemigroups of $\mat{n}{S}$ in detail, namely the semigroups $UT_n(\linz^*)$ where $\linz$ is the semiring associated with a linearly ordered abelian group. This family has some unusual, but striking, properties as Fountain semigroups. It is known that if the idempotents of a Fountain semigroup $S$ commute, then every element is
 $\Rt$-related and $\Lt$-related to a unique idempotent. Examples of Fountain semigroups in which the latter behaviour holds without the idempotents commuting are hard to find: see
 \cite{araujo} in the case where $\Lt=\Ls$ and $\Rt=\Rs$. On the other hand, by this stage we have seen that every element $A$ of
 $\upper{n}{\linz^*}$ is $\Rt$-related to a unique idempotent, the idempotent $A^{(+)}$ described above, and dually, $A$ is $\Lt$-related to a unique idempotent, $A^{(*)}$ 
 ({\bf Corollary~\ref{full}}),  but the idempotents of $UT_n(\linz^*)$ do not commute for $n\geq 3$. 
 
 Much of the literature dealing with Fountain semigroups considers only those for which $\Rt$ and $\Lt$ are, respectively, left and right compatible with multiplication, a property held by $\Rs$ and $\Ls$, and by $\R$ and $\L$. For such semigroups the 
 $\Ht:=\Rt\cap\Lt$-classes of idempotents are subsemigroups, indeed, {\em unipotent monoids} (monoids possessing a single idempotent). A substantial theory exists for Fountain semigroups in which unipotent monoids, arising from the   $\Ht$-classes of idempotents, or as quotients, play the role held by maximal subgroups in the theory of regular semigroups (see, for example, \cite{hartmann}). In  Section \ref{sec:upper} we see that not only are
 $\Rt$ and $\Lt$ not left and right compatible, respectively, but are as far from being so as possible, in a sense we will make precise ({\bf Proposition~\ref{prop:leftcong}}).  We show that $UT_n(\linz^*)$ is regular if and only if it is abundant, and  characterise the  generalised regularity properties of  $\upper{n}{\linz^*}$ and $\uni{n}{\linz^*}$ ({\bf Corollary~\ref{regut}}). In particular, $\upper{n}{\linz^*}$ and $\uni{n}{\linz^*}$ are not regular, but are Fountain, for $n\geq 3$. The maximal subgroups of $\upper{n}{\linz^*}$ are isomorphic to the 
 underlying abelian group $\linz^*$ ({\bf Theorem~\ref{thm:upperH}}). 
 
 Finally, in Section \ref{sec:tildeclasses} we continue the theme of 
 Section \ref{sec:upper} by carefully analysing the structure of 
 $\Rt$-, $\Lt$- and $\Ht$-classes in $UT_n(\linz^*)$.  We use the notion of defect, introduced in \cite{TaylorThesis} in the special case of the tropical semiring, to determine the behaviour of the  $\Rt$- and $\Lt$-classes, and hence the
 $\Ht$-classes, in $UT_n(\linz^*)$. For $n\leq 4$ the $\Ht$-classes of idempotents are
 always  subsemigroups, and in certain cases for $n\geq 5$, depending on properties related to defect, that we refer to as being tight or loose. However, we are also able to show that for  every $n\geq 5$ there is an idempotent in $UT_n(\linz^*)$ such that its $\Ht$-class is not a subsemigroup ({\bf Proposition~\ref{prop:ht}}).  These properties of tightness and looseness derive from the fact that the conditions for a matrix $E \in UT_n(\linz^*)$ to be idempotent correspond to a certain set of inequalities holding between products of the entries of $E$; we say that $E$ is tight in a particular product, if the inequality corresponding to that product is tight. In the tropical case, these inequalities are a classical linear system of inequalties in $\mathbb{R}^{n(n+1)/2}$, and hence describe a polyhedron of idempotents; tightness of an idempotent matrix $E$ in a specified product therefore corresponds to the point of  $\mathbb{R}^{n(n+1)/2}$ corresponding to $E$ lying on the hyperplane specified by that product.  
 
  Throughout this article we pose a series of Open Questions concerning abundance and Fountainicity of matrix semigroups over semirings. Indeed, many of these questions remain unanswered even in the case of rings. We hope our article provides a catalyst for other investigators.

  \section{Preliminaries}
\label{sec:prelim}
In order to keep our paper self-contained, we briefly recall some key definitions and results. For further information on semirings the reader may consult \cite{Golan}.

\subsection{Semirings}
\label{subsesec:Semi}
A semiring is a commutative monoid $(S,+,0_S)$ with an associative (but not necessarily commutative) multiplication $S \times S \rightarrow S$ that distributes over addition from both sides, where the additive identity $0_S$ is assumed to be an absorbing element for multiplication (that is, $0_S a = a 0_S = 0_S$ for all $a \in S$). Thus a ring is a semiring in which $(S, +, 0_S)$ is an abelian group. \emph{Throughout this paper we shall assume that $S$ is unital (i.e. contains a multiplicative identity element, denoted $1_S$) and commutative}. A (semi)field is a (semi)ring in which $S\smallsetminus \{0_S\}$ is an abelian group under multiplication; we denote this group by $S^*$. We say that a semiring is \emph{idempotent} if the addition is idempotent, and \emph{anti-negative} if $a+b=0_S$ implies $a=b=0_S$. It is easy to show that every idempotent semiring is anti-negative and that idempotent semirings are endowed with a natural partial order structure given by $a \leq b$ if and only if $a+b=b$.  

It may help the reader to keep the following examples in mind:
\begin{itemize}
\item[1.] The set of non-negative integers $\mathbb{N}_0$ with the usual operations of addition and multiplication.
\item[2.] The Boolean semiring $\mathbb{B} = \{0,1\}$ with idempotent addition $1+1=1$.
\item[3.] The tropical semiring $\mathbb{T}:=\mathbb{R} \cup \{-\infty\}$ with addition $\oplus$ given by taking the maximum (where $-\infty$ is the least element and hence plays the role of the additive identity) and multiplication $\otimes$ given by usual addition of real numbers, together with the rule $-\infty \otimes a = a \otimes-\infty = -\infty$.
\item[4.] More generally, if $G$ is  lattice ordered abelian group, one may construct an idempotent semifield  from $G$ by adjoining a minimal element, to be treated as zero, to $G$ and taking addition to be least upper bound. (In fact, every idempotent semifield arises in this way -- see Lemma \ref{semifields} below.)
\end{itemize}

For ease of reference we record (with no claim of originality) a number of facts about semifields, which have been observed by many authors (see for example \cite{GJK, Hutchins, Weinert}):
\begin{lemma}
\label{semifields}
Let $S$ be a semifield.
\begin{enumerate}[\rm(i)]
\item Either $S$ is a field or $S$ is anti-negative.
\item If $S$ is anti-negative, then the multiplicative group $S^*$ is torsion-free.
\item If $S$ is finite, then $S$ is a finite field or the Boolean semiring.
\item If $x+x=x$ for some $x \in S^*$, then $S$ is idempotent.
\item If $S$ is idempotent, then $S^*$ is a lattice ordered abelian group.
\end{enumerate}
\end{lemma}

Examples 2 and 3 above are both idempotent semifields, in which the natural partial order is total. The respective mutiplicative groups are the trivial group, and the group of real numbers with the usual order. In general, we shall denote an idempotent semifield in which the natural partial order is total by $\linz$, since the corresponding multiplicative group $\lin$ is a linearly ordered abelian group.

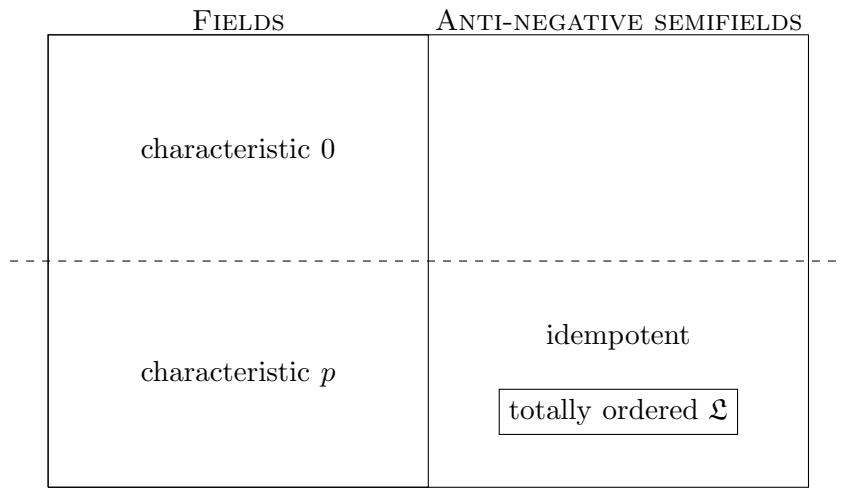
\begin{figure}
\begin{tikzpicture}
\draw (0,0) -- (10,0) -- (10,6) -- (0,6) -- (0,0);
\draw (0,0) -- (5,0) -- (5,6) -- (0,6) -- (0,0);
\draw[dashed] (-0.5,3) -- (10.5,3);
\node[align=left] at (2.5,6.2) {\sc{Fields}};
\node[align=left] at (7.5,6.2) {\sc{Anti-negative semifields}};
\node[align=left] at (2.5,4.5) {characteristic $0$};
\node[align=left] at (2.5,1.5) {characteristic $p$};
\node[align=left] at (7.5,2) {idempotent};
\node[draw,align=left] at (7.5,1) {totally ordered $\mathfrak{L}$};
\end{tikzpicture}
\caption{A rough guide to semifields.}
\label{fig:1}
\end{figure}

\subsection{Modules and matrices}
\label{subsec:ModMat}
A (left) \emph{$S$-module} is a commutative monoid $(X, + , 0_X)$ together with a left action $S \times X \rightarrow X$ satisfying for all $s,t \in S$, and all $x,y \in X$:
\begin{eqnarray*} 
1_S \cdot x = x,\; \; \; s\cdot (t \cdot x) = st\cdot x, & s\cdot 0_X = 0_X = 0_S \cdot x,\\
s \cdot (x + y) = s\cdot x + s\cdot y,& (s+t) \cdot x = s\cdot x + t \cdot x.
\end{eqnarray*}
It is clear that $S$ itself is a left $S$-module with left action given by multiplication within $S$. Let $I$ be a non-empty index set, and consider the set $S^I$ of all functions $I\mapsto S$, together with the operation of pointwise addition $(f+g)(i) = f(i)+g(i)$, and zero map $0(i)=0_S$ for all $i \in I$. This forms an $S$-module with action given by $(s\cdot f)(i) = s f(i)$, for all $i \in I$. For $i \in I$ we write $\delta_i$ for the element of $S^I$ defined by $\delta_i(j) =1_S$ if $i=j$ and $0_S$ otherwise.  The set $S^{(I)}$ consisting of all functions with finite support is a free (left) $S$-module with basis $\{\delta_i :i \in I\}$. Thus every finitely generated free $S$-module is isomorphic to a module of the form $S^n$. We shall also write $S^{m \times n}$ to denote the set of all $m\times n$ matrices over $S$, which forms an $S$-module in the obvious way.

Given $A \in \mat{n}{S}$ we write $A_{i, \star} \in S^{1 \times n}$  to denote the $i$th row of $A$ and $A_{\star, i} \in S^{n \times 1}$ to denote the $i$th column. The column space
$${\rm Col}(A)=\left\{\sum_{i=1}^n  A_{\star, i}\lambda_i : \lambda_i \in S\right\}$$
is the (right) $S$-submodule of $S^{n \times 1}$ generated by the columns of $A$. Dually, we define ${\rm Row}(A)$ to be the (left) $S$-submodule of $S^{n \times 1}$ generated by the rows of $A$.

Since an $S$-module over an arbitrary semiring has no underlying group structure, we must form quotients by considering congruences per se (rather than the congruence class of a particular element). Here by a congruence we mean an equivalence relation that is compatible with addition and the left action of $S$. The kernel of a set $X \subseteq S^{n\times 1}$ of column vectors is the left congruence on $S^{1 \times n}$ defined by 
$${\rm Ker}(X) = \{(v, v') \in S^{1 \times n} \times S^{1 \times n}: vx=v'x \mbox{ for all } x \in X\}.$$
Given a matrix $A \in \mat{n}{S}$, it is easy to see that the kernel of the  set of columns of $\{A_{\star,1},\ldots,A_{\star,n}\}$ is equal to the kernel of the column space  ${\rm Col}(A)$; this is the set-theoretic kernel of the surjective left linear function $S^{1 \times n} \rightarrow {\rm Row}(A)$ given by $v \mapsto vA$. Thus $S^{1 \times n} /{\rm Ker}({\rm Col}(A))\cong {\rm Row}(A)$ as (left) $S$-modules.

For $A \in \mat{n}{S}$ we also define 
$${\rm ColStab}(A)=\{E \in \mat{n}{S}: E^2=E, \,EA=A\} \subseteq \mat{n}{S},$$
that is, the set of idempotents which act as left identities for $A$ (hence stabilising the column space of $A$) and
$${\rm ColFix}(A) = \bigcap_{F \in \rm{ColStab}(A) }{\rm Col}(F) \subseteq S^{n \times 1}$$ 
to denote the intersection of all the column spaces of elements of ${\rm ColStab}(A)$.

\subsection{Left/right relations, regularity and generalisations}
\label{subsec:Green}
We briefly outline the basic ideas required from semigroup theory; for further reference, the reader is referred to \cite{Howie}.
Let $T$ be a semigroup.  Several equivalence relations on $T$ are defined using properties of the left and right actions of $T$ upon itself. We write $T^1$ to denote the monoid obtained by adjoining, if necessary, an identity element to $T$; $E(T)$ for the set of idempotent elements of $T$ (that is, those $e \in T$ for which $e^2=e$); and $U$ for an arbitrary (but fixed) subset of $E(T)$. For $a, b \in T$, we say that:
\begin{center}
\begin{tabular}{ l  }
$a \R b$ if $aT^1=bT^1$; $\qquad a\Rless b$ if $aT^1\subseteq bT^1$;\\
$a \L  b$ if $T^1a=T^1b$; $\qquad a\Lless b$ if $T^1a\subseteq T^1b$;\\
\\
$a\Rs b$ if for all  $x,y \in T^1$:  $xa=ya$ if and only if $xb=yb$;\\
$a\Ls b$ if  for all  $x,y \in T^1$:  $ax=ay$ if and only if $bx=by$;\\
\\
$a\Rtu{U} b$ if for all  $e \in U$:  $ea=a$ if and only if $eb=b$;\\
$a\Ltu{U} b$ if for all  $e \in U$:  $ae=a$ if and only if $be=b$.\\

\end{tabular}
\end{center}
When $U=E(T)$, we write simply $\Rt$ and $\Lt$ in place of $\Rtu{E(T)}$ and $\Ltu{E(T)}$. It is easily verified that the relations $\R, \L, \Rs, \Ls, \Rt$ and $\Lt$ are equivalence relations on $T$. The two relations $\R$ and $\L$ are the familiar Green's relations (see for example \cite{Howie}), whilst the remaining relations are much-studied \cite{ElQallaliThesis, Fountain:1982, Lawson, Pastijn}  extensions of these, in the sense that:
$$\R \subseteq \Rs\subseteq \Rt \subseteq \Rtu{U} \;\;\mbox{ and }\;\;
\L \subseteq \Ls\subseteq \Lt \subseteq \Ltu{U}.$$
The relations $\R$ and $\L$ clearly encapsulate notions of (one-sided) divisibility. The relation corresponding to the obvious two-sided version is denoted by $\J$. The meet of $\L$ and $\R$ is denoted by $\H$, whilst their join is denoted by $\D$ and the latter turns out to be equal to both $\L \circ\R$ and $\R \circ \L$. Starting with $\Ls$ and $\Rs$ (respectively, $\Lt$ and $\Rt$) one may define analogous relations $\Hs$ and $\Ds$ (respectively, $\Ht$ and $\Dt$), although the relations $\Ls$ and $\Rs$ (respectively, $\Lt$ and $\Rt$) need not commute in general.

It follows from \cite{Pastijn} that $a \Rs b$  in $T$ may be alternatively characterised as $a \R b$ in an oversemigroup $T' \supseteq T$, and dually for the relation $\Ls$. Whilst the relations $\R$ and $\Rs$ (respectively, $\L$ and $\Ls$) are well-known to be left (respectively, right) congruences on $S$, in general $\Rt$ and $\Rtu{U}$ (respectively, $\Lt$ and $\Ltu{U}$) need not be.

We say that a semigroup $T$ is \emph{regular} if for each element $a \in T$ there exists $x \in T$ such that $axa=a$. Equivalently, $T$ is regular if every $\R$-class (or every $\L$-class) of $T$ contains an idempotent. The latter characterisation in terms of the `abundance' of idempotents in $T$, has given rise to the following generalisations:

\begin{itemize}
\item $T$ is \emph{abundant} if and only if  every $\Rs$-class and every $\Ls$-class of $T$ contains an idempotent;
\item $T$ is \emph{Fountain} (or weakly abundant, or semi-abundant) if and only if  every $\Rt$-class and every $\Lt$-class of $T$ contains an idempotent;
\item $T$ is \emph{$U$-Fountain} (or weakly $U$-abundant, or $U$ semi-abundant) if and only if  every $\Rtu{U}$-class and every $\Ltu{U}$-class of $T$ contains an idempotent.
\end{itemize}
The term ``Fountain'' was coined in \cite{MS17} to highlight the contribution of John Fountain to the study of such semigroups. 

The following lemma (which will be used repeatedly throughout) is well known to specialists (see for example \cite{Gould:2010}), however we record a brief proof for completeness.

\begin{lemma}
Let $a,e,f \in T$ with $e,f \in U \subseteq E(T)$. Then $a \Rtu{U} e$ if and only if $ea=a$ and $e \Rless f$ for all idempotents $f\in U$ with $fa=a$. In particular, $e \Rtu{U} f$ if and only if $e \R f$.  
\end{lemma}

\begin{proof}
If $a \Rtu{U} e$, then by definition for all $f \in U, fa=a$ if and only if $fe=e$. The forward direction yields that any idempotent of $U$ fixing $a$ is above $e$ in the $\R$-order on $T$, and since $ee=e \in U$ the converse direction gives in particular that $ea=a$. On the other hand, suppose that $ea=a$ and $e \Rless f$ for all $f \in U$ with $fa=a$. Let $g \in U$. If $ga=a$, then by assumption $e=gx$ for some $x \in T$, giving $ge=g(gx)=(gg)x=gx=e$.  On the other hand, if $ge=e$, then $ga=g(ea)=(ge)a=ea=a$.  
\end{proof}

\subsection{Involutary anti-automorphisms}
\label{subsec:Anti}
Let $T$ be a semigroup. We say that a bijective map $\varphi:T \rightarrow T$ is an \emph{involutary anti-automorphism} if $\varphi^{-1}=\varphi$ and $\varphi(ab)= \varphi(b) \varphi(a)$ for all $a, b\in T$. For example, in an inverse semigroup the map sending an element to its inverse is such a map. The existence of such a map ensures left-right symmetry in the structure of $T$.

\begin{lemma}
\label{l-r}
Let $T$ be a semigroup with involutary anti-automorphism $\varphi$.
\begin{enumerate}[\rm(i)]
\item $T$ is abundant if and only if each $\Rs$-class contains an idempotent.
\item $T$ is Fountain if and only if each $\Rt$-class contains an idempotent.
\end{enumerate}
Further, if $U \subseteq E(T)$ with $\varphi(U) =U$, then:
\begin{enumerate}[\rm(i)] \setcounter{enumi}{2}\item$T$ is $U$-Fountain if and only if each $\Rtu{U}$-class contains an idempotent.\end{enumerate}
\end{lemma}
\begin{proof}
(i) Suppose that each $\Rs$-class contains at least one idempotent. Choose one idempotent element of each $\Rs$ class as a representative and define a function $\varepsilon: T \rightarrow E(T)$ mapping each element to the idempotent representative of its $\Rs$-class. We claim that for all $a \in T$, the element $\varphi(\varepsilon(\varphi(a)))$ is an idempotent which is $\Ls$-related to $a$. First, since $\varphi$ is an anti-automorphism and $\varepsilon(\varphi(a))$ is an idempotent we clearly have
$$\varphi(\varepsilon(\varphi(a))) \cdot \varphi(\varepsilon(\varphi(a))) = \varphi (\varepsilon(\varphi(a)) \cdot \varepsilon(\varphi(a))) = \varphi(\varepsilon(\varphi(a))).$$
Let $a,x,y \in T$. Then, by applying $\varphi$, we see that $ax=ay$ if and only if $\varphi(x)\varphi(a) = \varphi(y)\varphi(a)$. Since $\varphi(a) \Rs \varepsilon(\varphi(a))$ the latter is equivalent to $\varphi(x)\varepsilon(\varphi(a)) = \varphi(y)\varepsilon(\varphi(a))$. Applying $\varphi$ once more then yields $ax=ay$ if and only if $\varphi(\varepsilon(\varphi(a))) x = \varphi(\varepsilon(\varphi(a))) y$. Thus $a \Ls  \varphi(\varepsilon(\varphi(a)))$ as required. This shows that if each $\Rs$-class contains an idempotent, then each $\Ls$-class also contains an idempotent, and hence $T$ is abundant.

(ii) Suppose that each $\Rt$-class contains at least one idempotent. Choose one idempotent element of each $\Rt$ class as a representative and define a function $\gamma: T \rightarrow E(T)$ mapping each element to the idempotent representative of its $\Rt$-class. As above, it is easy to see that $a \in T$ is $\Lt$-related to the idempotent $\varphi(\gamma(\varphi(a)))$, and so $T$ is Fountain.

(iii) Since $\varphi(U)=U$, arguing as above gives the desired result.
\end{proof}

Following the convention in \cite{GJK}, we say that a map $\varphi: T \rightarrow T$ \emph{exchanges} two binary relations $\tau$ and $\rho$ on $T$ if:
$$x \, \tau\, y \Rightarrow  \varphi(x) \, \rho \, \varphi(y) \mbox{ and }x \, \rho \, y \Rightarrow \varphi(x)\, \tau \, \varphi(y).$$
We say that $\varphi$ \emph{strongly exchanges} $\tau$ with $\rho$ if 
$$x \, \tau \, y \Leftrightarrow \varphi(x) \, \rho \, \varphi(y) \mbox{ and } x \, \rho \, y \Leftrightarrow \varphi(x) \, \tau \, \varphi(y).$$

\begin{lemma}
\label{symmetry}
Let $T$ be a semigroup and $\varphi:T \rightarrow T$  an involutary anti-automorphism.  Then:
\begin{enumerate}[\rm(i)]
\item $\varphi$ strongly exchanges $\L$ and $\R$.  
\item $\varphi$ strongly exchanges $\Ls$ and $\Rs$.  
\item $\varphi$ strongly exchanges $\Lt$ and $\Rt$.  
\item $\varphi$ acts on the $\,\mathcal{K}\,$, $\,\mathcal{K}^*\,$, and $\,\widetilde{\mathcal{K}}\,$-classes for $\mathcal{K} \in \{\H, \D\}$.
\end{enumerate} 
Further, if $U \subseteq E(T)$ with $\varphi(U) =U$, then:
\begin{enumerate}[\rm(i)] \setcounter{enumi}{4}\item $\varphi$ strongly exchanges $\Ltu{U}$ and $\Rtu{U}$, and $\varphi$ acts on the $\Htu{U}$ and $\Dtu{U}$-classes.  
\end{enumerate}
\end{lemma}
\begin{proof}
We may extend $\varphi$ to an involutary anti-automorphism $T^1 \rightarrow T^1$ by setting $\varphi(1)=1$. By an abuse of notation, we denote this map also by $\varphi$.

(i) Let $a,b \in T$ and $x,y \in T^1$. By applying $\varphi$ it is easy to see that $a=xb$ and $b=ya$ if and only if $\varphi(a) = \varphi(b)\varphi(x)$ and $\varphi(b)=\varphi(a)\varphi(y)$. Since $\varphi$ is bijective, it is clear that $a \L b$ if and only if $\varphi(a) \R \varphi(b)$. Since $\varphi$ is an involution we must also have $a \R b$ if and only if $\varphi(a) \L\varphi(b)$.

(ii) Suppose that $a\Ls b$. Then for all $x,y \in T^1$, we have  $ax=ay$ if and only if $bx=by$. Applying $\varphi$ gives that $\varphi(x)\varphi(a)=\varphi(y)\varphi(a)$ if and only if $\varphi(x)\varphi(b)=\varphi(y)\varphi(b)$. Since $\varphi$ is bijective, this shows that $a\Ls b$ if and only if $\varphi(a) \Rs \varphi(b)$. Since $\varphi$ is an involution we must also have $a \Rs b$ if and only if $\varphi(a) \Ls \varphi(b)$.

(iii) For $x \in T$ it is easy to see that $x$ is idempotent if and only if $\varphi(x)$ is idempotent. Arguing as in parts (i) and (ii) one then finds that $a\Lt b$ if and only if $\varphi(a) \Rt \varphi(b)$ and vice versa.

(iv) This follows easily from parts (i)-(iii) together with the innate left-right symmetry of the definitions.

(v) Repeating the argumentation of part (iii) for $e \in U$, noting that $\varphi(e) \in U$, gives the desired result. \end{proof}

\section{Generalised regularity conditions for matrix semigroups}
\label{sec:reg}

In this section we consider the generalised regularity properties of the semigroups $\mat{n}{S}$, $\upper{n}{S}$ and $\uni{n}{S}$ over a general semiring $S$. In the case where $S$ is idempotent, we shall see that for all $n \geq 4$, the semigroups $\mat{n}{S}$ and $\upper{n}{S}$ are not Fountain.

Reflecting along the diagonals illuminates 
left-right symmetry in the ideal structures of $\mat{n}{S}$, $\upper{n}{S}$ and $\uni{n}{S}$.
\begin{lemma}
\label{delta}
Let $S$ be a semiring.
\begin{enumerate}[\rm(i)]
\item
 The transpose map is an involutary anti-automorphism of $\mat{n}{S}$.
\item The map $\Delta: \upper{n}{S} \rightarrow \upper{n}{S}$ obtained by reflecting along the anti-diagonal is an involutary anti-automorphism of $\upper{n}{S}$.
\item Restricting $\Delta$ to $\uni{n}{S}$ yields an involutary anti-automorphism of $\uni{n}{S}$.
\end{enumerate}
\end{lemma}
\begin{proof}
We give the details of part (ii) only, part (i) being well known. First notice that for all $1 \leq i \leq j \leq n$ we have $\Delta(X)_{i,j} = X_{n-j+1, n-i+1}$. Then, since
$$(AB)_{i,j} = \sum_{i \leq k \leq j} A_{i,k} \cdot B_{k,j}$$
we obtain
$$[\Delta(AB)]_{i,j} = (AB)_{n-j+1, n-i+1}= \sum_{n-j+1 \leq k \leq n-i+1} A_{n-j+1,k} \cdot B_{k,n-i+1}.$$
On the other hand, 
\begin{eqnarray*}
[\Delta(B)\Delta(A)]_{i,j} &=& \sum_{i \leq p \leq j} \Delta(B)_{i,p} \cdot \Delta(A)_{p,j}=\sum_{i \leq p \leq j} B_{n-p+1,n-i+1} \cdot A_{n-j+1,n-p+1}\\
 &=& \sum_{n-j+1 \leq q \leq n-i+1}  A_{n-j+1,q}\cdot   B_{q,n-i+1}= [\Delta(AB)]_{i,j} .
\end{eqnarray*}
(iii) Clearly $\Delta(\uni{n}{S}) = \uni{n}{S}$. 
\end{proof}

Thus it follows from Lemma \ref{l-r} that in investigating generalised regularity properties of $\mat{n}{S}$, $\upper{n}{S}$ and $\uni{n}{S}$ (and many other subsemigroups of $\mat{n}{S}$) it suffices to consider the relations $\R$, $\Rs$ and $\Rt$ only.

\subsection{Green's relations}
\label{subsec:GreenMat}
The monoid $\mat{n}{S}$ is isomorphic to the monoid of endomorphisms of the free $S$-module $S^n$. The relations $\R, \Rs, \Rt$  on $\mat{n}{S}$ can be readily phrased in terms of certain submodules and congruences on $S^n$ (in the case of $\R$ this is well known-- see \cite{Okninski, Kim82, JK13} for example); for the reader's convenience we briefly outline the details. 

\begin{lemma}
\label{LR}
Let $S$ be a semiring and let $A, B \in \mat{n}{S}$. Then
\begin{enumerate}[\rm(i)]
\item $A \R B$ if and only if ${\rm Col}(A) = \rm{Col}(B)$;
\item $A \Rs B$ if and only if ${\rm Ker}({\rm Col}(A)) = {\rm Ker}({\rm Col}(B))$;
\item $A \Rt B$ if and only if ${\rm ColStab}(A) ={\rm ColStab}(B)$;
\end{enumerate}
\end{lemma}

\begin{proof}
(i) Suppose that $A=BX$ for some $X \in \mat{n}{S}$. Then we can write each column of $A$ as a right $S$-linear combination of the columns of $B$. It then follows (from distributivity of the action) that if $x \in {\rm Col}(A)$ we may write $x$ first as a right $S$-linear combination of the columns of $A$, and hence then as a right $S$-linear combination of the columns of $B$, showing that ${\rm Col}(A) \subseteq {\rm Col}(B)$. On the other hand, suppose that ${\rm Col}(A) \subseteq {\rm Col}(B)$. Using the fact that $S$ has identity elements $0_S$ and $1_S$ which act in the appropriate manner, it is clear that each column of $A$ is an element of the column space ${\rm Col}(A)$ and so can be written as a right $S$-linear combination of the columns of $B$. Taking $X$ to be the coefficient matrix of these combinations, we see that $A=BX$. It then follows from the above that $A \R B$ if and only if the two column spaces are equal.

(ii) Suppose that whenever $XA=YA$ for $X, Y \in \mat{n}{S}$ we also have that $XB=YB$. Let $(x,y) \in {\rm Ker}({\rm Col}(A)) \subseteq S^{1 \times n} \times S^{1 \times n}$, and let $X$ be the matrix with all rows equal to $x$ and $Y$ the matrix with all rows equal to $y$. Then it follows that $XA=YA$ and hence $XB=YB$, whence $xB=yB$. This shows that ${\rm Ker}({\rm Col}(A)) \subseteq {\rm Ker}({\rm Col}(B))$. It then follows that $A \Rs B$ implies that the two column kernels agree. Conversely, suppose that the column kernels of $A$ and $B$. If $XA=YA$ for some $X,Y\in \mat{n}{S}$, then for all $1\leq i,j\leq n$ we have $X_{i, \star}A_{\star, j} = Y_{i, \star}A_{\star,j}$. Thus for each $i=1, \ldots, n$ we have $(X_{i, \star}, Y_{i,\star}) \in {\rm Ker}(\rm{Col}(A)) =  {\rm Ker}(\rm{Col}(B))$, whence $XB=YB$.

(iii) This is immediate from the definition of ${\rm ColStab}(A)$.
\end{proof}

 The relations $\L$, $\Ls$ and $\Lt$ can be dually characterised in terms of row spaces. It is clear from the above proof that the pre-order $\Rless$ corresponds to containment of column spaces. 
 Lemma~\ref{LR} and its dual  enable one description of regularity, abundance and Fountainicity, since it allows us to describe the conditions under which $A\in  \mat{n}{S}$
is related to an idempotent $E\in  \mat{n}{S}$ under the relevant relations. 

We now give the first in a series of results examining the relationship between generalisations of Green's relations in  $\mat{n}{S}$ and certain natural subsemigroups, such as $\upper{n}{S}$.

\begin{lemma}
\label{Rs}
Let $S$ be a semiring and let $A, B \in \upper{n}{S}$. Then $A \Rs B$ in $\upper{n}{S}$  if and only if $A \Rs B$ in $\mat{n}{S}$.
\end{lemma}

\begin{proof}
If $A \Rs B$ in $\mat{n}{S}$ then by definition we have that for all $X, Y \in \mat{n}{S}$ the equality $XA=YA$ holds if and only if the equality $XB=YB$ holds. Specialising to $X, Y \in \upper{n}{S}$ then yields $A \Rs B$ in $\upper{n}{S}$. Conversely, suppose that $A \Rs B$ in $\upper{n}{S}$. We show that $(x, y) \in {\rm Ker}({\rm Col}(A))$ if and only if $(x,y) \in{\rm Ker}({\rm Col}(B))$. Let $X$ and $Y$ be the upper triangular matrices with first row $x$ and $y$ respectively, and all remaining rows zero. Since $A \Rs B$ in $\upper{n}{S}$, we have $XA=YA$ if and only if $XB=YB$, and hence in particular (by looking at the content of the first row of these two products), for each column $A_{\star, j}$ we have $x A_{\star, j} = y  A_{\star, j}$ if and only if $x B_{\star, j} = y  B_{\star, j}$, thus giving the desired rsult.
\end{proof}

\subsection{Regularity}
\label{subsec:RegMat}
The question of which monoids $\mat{n}{S}$ are regular has been answered by Il'in \cite[Theorem 4]{Ilin}, who has shown that for $n \geq 3$, $\mat{n}{S}$ is regular if and only if $S$ is a von-Neumann regular ring.

On the other hand, for an arbitrary semiring $S$ the upper triangular monoids $\upper{n}{S}$ are not regular for all $n\geq 2$.
\begin{proposition}
\label{notReg}
Let $S$ be a semiring. The monoid $\upper{n}{S}$ is regular if and only if $n=1$ and the multiplicative reduct $(S, \cdot)$ is regular.
\end{proposition}

\begin{proof}
The monoid $\upper{1}{S}$ is clearly isomorphic to $(S, \cdot)$. Suppose then that $n \geq 2$. Consider the upper triangular matrix $A$ with entries $A_{1,j}=1$ for $2 \leq j \leq n$, and all other entries equal to $0$. We first show that an upper triangular matrix $B$ is $\R$-related to $A$ in $\upper{n}{S}$ if and only if:
\begin{enumerate}[\rm(i)]
\item All non-zero entries of n $B$ lie in the first row;
\item $B_{1,1}=0$; and
\item $B_{1,2}$ has a right inverse $b' \in S$.
\end{enumerate}
It is straightforward to verify that each such $B$ is $\R$-related to $A$; taking $X$ to be the diagonal matrix with $X_{j,j} = B_{1,j}$ for all $j \in [n]$ yields $AX=B$, whilst taking $Y$ to be the matrix with $Y_{2,j} = b'$ for all $j\geq2$ and all other entries equal to 0 yields $BY=A$.

Suppose then that $C \R A$. Thus $C=AU$ and $A=CV$ for some $U,V \in \upper{n}{S}$. Since rows $2$ to $n$ of $A$ are all zero, it is clear that all non-zero entries of $C=AU$ must lie in the first row. Moreover, since $U$ is triangular, we see that $A_{1,1}=0$ also forces $C_{1,1} =A_{1,1}U_{1,1}=0$. Since $CV=A$, then we must also have $C_{1,2}V_{2,2}=A_{1,2}=1$. Thus we have shown that each matrix $\R$-related to $A$ must satisfy conditions (i)-(iii) above.

Notice that, since the diagonal entries of each any matrix in this $\R$-class are $0$, these elements are nilpotent. Since each such matrix is non-zero, this $\R$-class does not contain an idempotent.
\end{proof}

When $S$ is a commutative ring, it is easy to see (e.g. by performing invertible row operations of the form $r_i \mapsto r_i + \lambda r_k$ for $1 \leq i \leq k \leq n$) that the unitriangular monoid $\uni{n}{S}$ is a group, and hence in particular regular. At the other extreme, when $S$ is an anti-negative semiring it is easy to see that $\uni{n}{S}$ is not regular for all $n \geq 3$, since for example it is straightforward to show that the matrix equation
$$\left(
\begin{array}{c c c}
1 & 1 & 0 \\
0 & 1 & 1\\
0 & 0 & 1
\end{array}\right)\left(
\begin{array}{c c c}
1 & x & y \\
0 & 1 & z\\
0 & 0 & 1
\end{array}\right)\left(
\begin{array}{c c c}
1 & 1 & 0 \\
0 & 1 & 1\\
0 & 0 & 1
\end{array}\right)=\left(
\begin{array}{c c c}
1 & 1 & 0 \\
0 & 1 & 1\\
0 & 0 & 1
\end{array}\right),$$
has no solution.

Thus in the case where $S$ is a semifield, it follows from the above discussion and Lemma \ref{semifields} that $\uni{n}{S}$ is regular if and only if (i) $n=1$, or (ii) $n = 2$ and for each $a \in S$ there exists $x$ such that $a+x+a=a$, or  (iii) $n \geq 3$ and $S$ is a field.

\subsection{Abundance}
\label{subsec:AbundMat}
Likewise, one can ask under what conditions on $S$ and $n$ do we have that $\mat{n}{S}$ (respectively, $\upper{n}{S}$ or $\uni{n}{S}$) is abundant?

Recall from \cite[Theorem 3.2]{WJK} that a semiring $S$ is said to be \emph{exact} (or FP-injective) if for all $A \in S^{m \times n}$:

\begin{itemize}
\item[(F1)] For each $B \in S^{p \times n}$ we have ${\rm Ker}({\rm Row}(A)) \subseteq {\rm Ker}({\rm Row}(B))$ if and only if ${\rm Row}(B) \subseteq {\rm Row}(A)$.
\item [(F2)] For each $B \in S^{m \times q}$ we have ${\rm Ker}({\rm Col}(A)) \subseteq {\rm Ker}({\rm Col}(B))$ if and only if ${\rm Col}(B) \subseteq {\rm Col}(A)$.
\end{itemize}

When $S$ is exact, abundance coincides with regularity for $\mat{n}{S}$:

\begin{theorem}
\label{exact}
Let $S$ be an exact semiring, $n \in \mathbb{N}$. Then $\R=\Rs$ in the semigroup $\mat{n}{S}$. Thus $\mat{n}{S}$ is abundant if and only if it is regular.
\end{theorem}

\begin{proof}
It is clear from the definitions that $\R \subseteq \Rs$. Suppose that $A, B \in \mat{n}{S}$ are $\Rs$-related. By Lemma \ref{LR} we see that ${\rm Ker}(\rm{Col}(A)) =  {\rm Ker}(\rm{Col}(B))$. By the exactness of $S$, this yields ${\rm Col}(A) = {\rm Col}(B)$, and hence $A \R B$, by Lemma \ref{LR} again.\end{proof}

Together with \cite[Theorem 4]{Ilin}, this gives the following:
\begin{corollary}
\label{regular}
Let $S$ be an exact semiring and let $n\geq3$. The semigroup $\mat{n}{S}$ is abundant if and only if the semiring $S$ is a von-Neumann regular ring.
\end{corollary}

Theorem \ref{exact} also provides a simple characterisation of the relation $\Rs$ on $\upper{n}{S}$ in the case that $S$ is exact.
\begin{corollary}
\label{exactRs}
Let $S$ be an exact semiring. Then $A \Rs B$ in $\upper{n}{S}$ if and only if ${\rm Col}(A) = {\rm Col}(B)$.
\end{corollary}
\begin{proof}
By Lemma \ref{Rs} we have that $A \Rs B$ in $\upper{n}{S}$ if and only if $A \Rs B$ in $\mat{n}{S}$. Theorem \ref{exact} says that $A \Rs B$ in $\mat{n}{S}$ if and only if $A \R B$ in $\mat{n}{S}$, which by Theorem \ref{LR} happens precisely when ${\rm Col}(A) = {\rm Col}(B)$.
\end{proof}

As mentioned earlier, by \cite{Pastijn} one has that $a \Rs b$ in a semigroup $T$ if and only if $a \R b$ in \emph{some} oversemigroup $T'$ of $T$; Corollary \ref{exactRs} says that for $T=\upper{n}{S}$ over an exact semiring $S$ it suffices to consider the oversemigroup $\mat{n}{S}$.

Notice however that the $\R$-relation on $\upper{n}{S}$ is not characterised by equality of column spaces; for example $ \left(\begin{array}{cc} 1 &0\\0&0\end{array}\right)$ and $\left(\begin{array}{cc} 0 &1\\0&0\end{array}\right)$ possess the same column space but are not $\R$-related in $\upper{2}{S}$. One can formulate an alternative characterisation in terms of upper triangular column operations (see for example \cite{KambitesSteinberg}, noting that the notation used there conflicts with our own).

The class of exact semirings is known to contain all fields, \emph{proper} quotients of principal ideal domains, matrix rings and finite group rings over the above, the Boolean semiring $\mathbb{B}$, the tropical semiring $\mathbb{T}$, and some generalisations of these (see \cite{WJK, WildingThesis} for details). If $S$ is a semifield, Shitov has shown that $S$ is exact if and only if either $S$ is a field or $S$ is an idempotent semifield \cite{Shitov17}. In light of Figure \ref{fig:1} one may wonder if abundance and regularity coincide in $\mat{n}{S}$ for all semifields $S$, however, this is not the case as the following example illustrates.

 \begin{example}
Let $\Qpos$ denote the semifield of non-negative rational numers, and for $n \geq 2$ consider the block diagonal matrix
 $$A= \left( \begin{array}{c c | c} 2& 1 &  \\
 1 & 3&\\
 \hline
 & & I_{n-2}\end{array}\right) \in \mat{n}{\Qpos},$$
 where $I_{n-2}$ denotes the identity matrix of $\mat{n-2}{\Qpos}$ and omitted entries are zero. It is straightforward to verify that $A$  is not regular in $\mat{n}{\Qpos}$. However, since $A$ is invertible in $\mat{n}{\mathbb{Q}}$ one has that $A \Rs I_n$ within $\mat{n}{\Qpos}$. Thus $\R \neq \Rs$ in $\mat{n}{\Qpos}$.
 \end{example} 

\begin{example}
The previous results imply that if $S$ is a field, then $\upper{n}{S}$ is abundant, as mentioned in the introduction. Indeed, for each $A \in \upper{n}{S}$, let $\overline{A}$ denote the upper triangular matrix obtained from the reduced column echelon form of $A$ by permuting the columns of $A$ to put all leading ones on the diagonal. Thus $\overline{A}=AX$ for some $X \in {\rm GL}_n(S)$ and hence $A \R \overline{A}$ in $\mat{n}{S}$. It is straightforward to check that $\overline{A}$ is an idempotent, and Lemma \ref{Rs} yields that $A \Rs \overline{A}$ in $\upper{n}{S}$. 
\end{example}

Over idempotent semifields the monoids $\upper{n}{S}$ and $\uni{n}{S}$ are typically not abundant.

\begin{proposition}
\label{notAbund}
Let $S$ be an idempotent semiring in which the only multiplicative idempotents are $0$ and $1$ (for example, an idempotent semifield), and let $n \geq 3$. Then the monoids $\upper{n}{S}$ and $\uni{n}{S}$ are not abundant.
\end{proposition}

\begin{proof}
Let $A$ be the matrix with top left corner:
$$\left(
\begin{array}{c c c}
1 & 1 & 0 \\
0 & 1 & 1\\
0 & 0 & 1
\end{array}\right),$$
and zeros elsewhere. We shall show that $A$ is not $\Rs$-related to any idempotent in $\upper{n}{S}$.

For $M \in \upper{n}{S}$ define $K(M)=\{(X,Y): XM=YM\} \subseteq \upper{n}{S} \times \upper{n}{S}$, and suppose for contradiction that $E$ is an idempotent $\Rs$-related to $A$ in $\upper{n}{S}$, so that $K(A)=K(E)$.

From the definition of $A$, it is easy to see that $(X,Y) \in K(A)$ if and only if:
\begin{itemize}
\item[(i)] $X_{i,i}=Y_{i,i}$ for $i=1,2,3$;
\item[(ii)] $X_{1,1}+X_{1,2}=Y_{1,1}+Y_{1,2}$;
\item[(iii)] $X_{2,2}+X_{2,3}=Y_{2,2}+Y_{2,3}$;
\item[(iv)] $X_{1,2}+X_{1,3}=Y_{1,2}+Y_{1,3}$.
\end{itemize}
It follows easily from condition (i) that we must have $E_{i,i}\neq 0$ for $i=1,2,3$. (Otherwise, $(E,I_n) \in K(E) \smallsetminus K(A)$.) Since $E$ is assumed to be idempotent and upper triangular, the condition that $E_{i,i}^2=E_{i,i}$ then yields $E_{i,i}=1$.

By considering the pair $(P,Q) \in K(A)$ with $P_{1,1}=P_{1,2}=Q_{1,1}=Q_{1,3}=1$ and all other entries equal to zero, we note that in position $(1,2)$ the equality $PE=QE$ yields:
\begin{eqnarray*}
(PE)_{1,2} = P_{1,1}E_{1,2} + P_{1,2}E_{2,2} &=& Q_{1,1}E_{1,2} + Q_{1,2}E_{2,2} = (QE)_{1,2}\\
E_{1,2} + 1 &=& E_{1,2} + 0\\
1 &\leq & E_{1,2}.
\end{eqnarray*}

Similarly, by considering the pair $(P,Q) \in K(A)$ with $P_{1,3}=P_{2,2}=P_{2,3}=Q_{2,2}=Q_{1,3}=1$ and all other entries equal to zero, we note that in position $(2,3)$ the equality $PE=QE$ yields:
\begin{eqnarray*}
(PE)_{2,3} = P_{2,2}E_{2,3} + P_{2,3}E_{3,3} &=& Q_{2,2}E_{2,3} + Q_{2,3}E_{3,3} = (QE)_{2,3}\\
E_{2,3} + 1 &=& E_{2,3} + 0\\
1 &\leq & E_{2,3}.
\end{eqnarray*}
Since $E$ was assumed to be idempotent we must have
\begin{eqnarray*}
E_{1,3} = (E^2)_{1,3} &=& E_{1,1}E_{1,3} + E_{1,2}E_{2,3} + E_{1,3}E_{3,3} \geq E_{1,3}+1.
\end{eqnarray*}
 But now, one can check that taking $V$ to be the matrix with $V_{1,3}=V_{i,i}=1$ for all $i$ and all remaining entries equal to zero we have $(I_n, V) \in K(E) \smallsetminus K(A)$, hence giving the desired contradiction.

Finally, if $A \in \uni{n}{S}$ is $\Rs$ related to an idempotent in $\uni{n}{S}$, then arguing exactly as above, observing that in this case $X_{i,i}=Y_{i,i}=A_{i,i}=E_{i,i}=1$ for all $i \in [n]$, we arrive at the same contradiction.
\end{proof}

For general semirings, a characterisation of abundance remains open.

\begin{question}
What conditions on $S$ and $n$ are necessary and sufficient for $\mat{n}{S}$ (respectively, $\upper{n}{S}$, $\uni{n}{S}$) to be abundant?
\end{question}

\subsection{Fountainicity}
\label{subsec:FountainMat}
\begin{proposition}
\label{Fountainicity}
The monoid $\mat{n}{S}$ is Fountain if and only if for each $A \in \mat{n}{S}$ there exists an idempotent $E$ such that ${\rm ColFix}(A) = {\rm Col}(E)$.
\end{proposition}

\begin{proof}
Let $A, E \in \mat{n}{S}$ with $E^2=E$. Then $A \Rt E$ if and only if $EA=A$ and $E \Rless F$ for any idempotent $F \in \mat{n}{S}$ satisfying $FA=A$. From earlier remarks this is equivalent to $E \in {\rm ColStab}(A)$ and ${\rm Col}(E) \subseteq {\rm Col}(F)$ for all $F \in {\rm ColStab}(A)$. Thus if $A \Rt E$ we have ${\rm Col}(E) = {\rm ColFix}(A)$ from the definition. Conversely, if ${\rm Col}(E) = {\rm ColFix}(A)$ for some $E=E^2$, then since ${\rm Col}(A) \subseteq {\rm Col}(F)$ for any $F \in {\rm ColStab}(A)$, we have ${\rm Col}(A) \subseteq {\rm Col}(E)$, whence $A \Rless E$ so that certainly $EA=A$, and by definition of ${\rm ColFix}(A)$, $E \Rless F$ for any $F=F^2$ with $FA=A$.
\end{proof}

We note that the column space of any idempotent (or indeed regular) matrix $E$ is a retract of $S^n$, and hence is a finitely (at most $n$) generated projective $S$-submodule of $S^n$ (see for example \cite[Example 17.15 and Proposition 17.16]{Golan}). Thus the previous proposition indicates that understanding the conditions under which  $\mat{n}{S}$ is Fountain boils down to understanding properties of intersections of certain projective $S$-submodules of $S^n$.

The monoids $\mat{n}{S}$ and $\upper{n}{S}$ over any idempotent semiring are seldom Fountain, as the following proposition illustrates.

\begin{proposition}
\label{notFountain}
Let $S$ be an idempotent semiring and let $n \geq 4$. Then $\mat{n}{S}$ (respectively $\upper{n}{S}$) is not Fountain.
\end{proposition}
\begin{proof}
Consider first the case $n=4$. We shall show that the matrix
$$A:=\left(
\begin{array}{c c c c }
0 & 1 & 1 & 1 \\
0 & 1 & 1 & 0 \\
0 & 0 & 1 & 1 \\
0 & 0 & 0 & 1
\end{array}\right),$$
is not $\Rt$-related to any idempotent in $\mat{4}{S}$.

If $X \in \mat{4}{S}$, then by using the anti-negativity of $S$ it is easy to see that  satisfies $XA=A$ if and only if $X$ has the form
\begin{equation}
\label{matrix}
\left(
\begin{array}{c c c c c}
a & b & c & d \\
0 & 1 & 0 & 0 \\
0 & 0 & 1 & e \\
0 & 0 & 0 & 1
\end{array}\right),
\end{equation}
where $a,b,c,d,e\in S$ satisfy
\begin{equation}
\label{cond}
a+b=a+b+c=a+c+d=1+e=1.
\end{equation}

Consider the matrices
$$F_1:=\left(
\begin{array}{c c c c c}
0 & 1 & 1 & 0 \\
0 & 1 & 0 & 0 \\
0 & 0 & 1 & 0 \\
0 & 0 & 0 & 1
\end{array}\right), F_2:=\left(
\begin{array}{c c c c c}
0 & 1 & 0 & 1 \\
0 & 1 & 0 & 0 \\
0 & 0 & 1 & 0 \\
0 & 0 & 0 & 1
\end{array}\right),$$
which by above satisfy $F_1A=F_2A=A$. It is straightforward to check that $F_1^2=F_1$ and $F_2^2=F_2$.

Now suppose for contradiction that $E^2=E\Rt A$ in $\mat{4}{S}$. Since $EA=A$, $E$ must have the form (\ref{matrix}) for some $a,b,c,d,e,f \in S$ satisfying (\ref{cond}). Since $F_1$ and $F_2$ are idempotents fixing $A$, we must also have $F_1E=E$ (which implies that $a=0$, $d=e$ and $b=c=1$) and $F_2E=E$ (which implies that $a=c=0$ and $b=d=1$), hence giving a contradiction. We conclude that $\mat{4}{S}$ is not Fountain. Noting that $\mat{4}{S}$ naturally embeds into $\mat{n}{S}$ for all $n \geq 4$, it is easy to see that the matrix with top right hand corner $A$ is not $\Rt$-related to any idempotent of $\mat{n}{S}$. Moreover, since all matrices involved in the above reasoning were in fact upper triangular, the same argument applies to show that $\upper{n}{S}$ is not Fountain.
\end{proof}

\begin{question}What conditions on $S$ and $n$ are necessary and sufficient for $\mat{n}{S}$ to be Fountain?
\end{question}

In the following sections, we shall show that for an idempotent semifield $S$ each of the monoids $\uni{n}{S}$ is Fountain. We also investigate the $\Rt$-classes of $\mat{n}{S}$, $\upper{n}{S}$ and $\uni{n}{S}$ in this case. We begin by considering certain idempotent matrix constructions, which in turn allow us to show that several interesting subsemigroups of $\mat{n}{S}$, including $\uni{n}{S}$, are Fountain.

\section{Matrix semigroups over idempotent semifields}
\label{sec:semfield}

Throughout this section let $S$ be an idempotent semifield. We recall from Lemma \ref{semifields} that every idempotent semifield $S$ arises by adjoining a minimal element $\0$ to a lattice ordered abelian group $(S^*, \cdot, \1)$, with addition corresponding to least upper bound in $S$. We write $\wedge$ and $\vee$ to denote the operations of greatest lower bound and least upper bound in $S$. (Note that taking the $S^*$ to be the trivial group yields the Boolean semiring $\bool$, whilst taking $S^*$ to be the real numbers under addition with respect to the usual total order yields the tropical semifield $\mathbb{T}$. Since the \emph{multiplicative} identity here is the real number 0, to avoid confusion the bottom element is usually denoted by $-\infty$. It would do no harm to first think of this example in what follows, keeping in mind these conventions.)

In the case where $S=\bool$, it is well known that the monoid $\mat{n}{\bool}$ is isomorphic to the monoid $\mathcal{B}_n$ of all binary relations on $[n]$ under relational composition, via the map sending a relation $\alpha \subseteq [n] \times [n]$ to the matrix $A \in \mat{n}{\bool}$ whose $(i,j)$th entry is $1$ if and only if $(i,j) \in \alpha$. With this in mind, for $A \in \mat{n}{S}$ we shall write ${\rm dom}(A)$ to denote the subset of $[n]:=\{1,\ldots, n\}$ indexing the non-zero rows of $A$, and ${\rm im}(A)$ to denote the subset of $[n]$ indexing the non-zero columns of $A$.  For $A, B \in \mat{n}{S}$ we also write $A \preceq B$ if $A_{i,j} \leq B_{i,j}$ for all $i, j \in [n]$, with respect to the partial order on $S$, noting that in the Boolean case the order $\preceq$ corresponds to containment of relations. We make use of the following simple observations several times in our arguments.

\begin{lemma}
\label{ordered}
Let $S$ be an idempotent semiring and let $A, B , C \in \mat{n}{S}$.
\begin{enumerate}[\rm (i)]
\item If $A \preceq B$, then $CA \preceq CB$ and $AC \preceq BC$.
\item If $C_{i,i} = \1$ for all $i \in [n]$, then $A \preceq CA$ and $A \preceq AC$.
\end{enumerate}
\end{lemma}

\begin{proof}
(i) For all $i,j \in [n]$ we have
\begin{eqnarray*}
(CA)_{i,j} &=& \bigvee_{k=1}^n C_{i,k}A_{k,j} \leq \bigvee_{k=1}^n C_{i,k}B_{k,j} = (CB)_{i,j},\\
(AC)_{i,j} &=& \bigvee_{k=1}^n A_{i,k}C_{k,j} \leq \bigvee_{k=1}^n B_{i,k}C_{k,j} = (BC)_{i,j}.
\end{eqnarray*}

(ii) If $C_{i,i}=\1$ for all $i \in [n]$, then for all $i,j \in [n]$ we have
$$(CA)_{i,j} = \bigvee_{k=1}^n C_{i,k}A_{k,j} \geq C_{i,i}A_{i,j} = A_{i,j}, \;\;(AC)_{i,j} = \bigvee_{k=1}^n A_{i,k}C_{k,j} \geq A_{i,j}C_{j,j} = A_{i,j}.
$$
\end{proof}

\subsection{An idempotent construction}
\label{subsec:Plus}
Let $\overline{S}$ denote the idempotent semiring obtained from $S$ by adjoining a top element satisfying $a\vee \pmb{\top} =  \pmb{\top} = \pmb{\top} \vee a$ for all $a \in S$, $\pmb{\top} \cdot a = \pmb{\top} = a \cdot \pmb{\top}$ for all $a \in \overline{S} \setminus \0$ and $\pmb{\top} \cdot \0=\0 \cdot \pmb{\top}$. In the language of \cite{CGQ97}, the commutative semiring $\overline{S}$ is \emph{residuated}, meaning that for every pair $a,b \in \overline{S}$, the set $\{ x \in \overline{S}: ax \leq b\}=\{ x \in \overline{S}: xa \leq b\}$ admits a maximal element $a\setminus b$, given by
$$a \setminus b = \begin{cases} 
\pmb{\top} &\mbox{ if } a=\0 \mbox{ or } a=b=\top,\\
ba^{-1} &\mbox{ if }
\0 < a< \pmb{\top},\\
\0 &\mbox{ if }a=\pmb{\top} \mbox{ and } b \neq \top.\end{cases}
$$
This idea can be extended to matrices, as follows.
\begin{lemma}\cite[Proposition 2 and Theorem 14]{CGQ97}.
Let $S$ be an idempotent semifield, $A,X,Y \in \mat{n}{\overline{S}}$ and define
$$(A \setminus X)_{i,j} = \bigwedge_k A_{k,i} \setminus X_{k,j},\;\;\; (X/A)_{i,j} = \bigwedge_l A_{j,l} \setminus X_{i,l},\;\;\;
(X \setminus A / Y )_{i,j} = \bigwedge_{k,l}Y_{j,l} \setminus (X_{k,i} \setminus A_{k,l}).
$$
Then
 \begin{enumerate}[\rm (i)]
\item$X \Rless A$ in $\mat{n}{\overline{S}}$ if and only if $A(A \setminus X)=X$.
\item $X \Lless A$ in $\mat{n}{\overline{S}}$ if and only if $(X/A)A=X$
\item  $A(A \setminus A / A)A \preceq A$.
\item If $A$ is regular, then $A(A\setminus A / A)A = A$, and $AXA=A \Rightarrow X \preceq A \setminus A /A$.
\end{enumerate}
\end{lemma}
We require similar constructions lying \emph{within} $\mat{n}{S}$. For $x, y \in S^n$ let $\Supp(x)$ denote the subset of $[n]$ indexing the non-zero positions of $x$ and define
$$\langle x | y \rangle:= \begin{cases}
\bigwedge_{i \in \Supp(x)}y_ix_i^{-1} &\mbox{ if } \emptyset \neq \Supp(x) \subseteq \Supp(y)\\
\0&\mbox{otherwise}.
\end{cases}.$$

The scalar product $\langle \cdot | \cdot \rangle$ is a modification of similar constructions present in the tropical literature (see for example \cite{CGQ97, JK13, JK14}), modified to our purpose, where we allow for the semiring to contain a bottom element, but do not assume that it contains a top element.

\begin{remark}
 We record some properties of $\langle \cdot | \cdot \rangle$:
\begin{enumerate}[\rm(i)]
\item It is immediate from the definition that $\langle x|y \rangle \in S$ and this is non-zero if and only if $\emptyset \neq \Supp(x) \subseteq \Supp(y)$. In particular, whenever $\langle x|y \rangle$ is non-zero, both $\Supp(x)$ and $\Supp(y)$ are non-empty, i.e. $x$ and $y$ must be non-zero vectors.
\item For each non-zero vector $x$, it is easy to see that $\langle x| x\rangle =\1$.
\item  Let $x,y,z \in S^n$. From the above observations $\langle x | y \rangle \langle y | z \rangle \neq \0$ implies that $\emptyset \neq \Supp(x) \subseteq \Supp(y) \subseteq \Supp(z)$. In this case, for each $i \in \Supp(x)$ we have
$\langle x | y \rangle \langle y | z \rangle \leq y_ix_i^{-1} z_iy_i^{-1} = z_ix_i^{-1}$, and hence
$$\langle x | y \rangle \langle y | z \rangle \leq \langle x | z \rangle.$$
\item Notice that for all $i \in [n]$ we have $\langle x |y \rangle x_i \leq y_i$. That is, $\langle x |y \rangle x \preceq y.$
\end{enumerate}
\end{remark}

For each $A \in \mat{n}{S}$ define $\lid{A} \in \mat{n}{S}$ via:
\begin{eqnarray}
\label{formula:Aplus}
(\lid{A})_{i,j} &=& \langle A_{j,\star}|A_{i,\star}\rangle\\
 &=&\nonumber \begin{cases}
\bigwedge\{A_{i,k}A_{j,k}^{-1}: k \in \Supp(A_{j,\star})\} &\mbox{ if } \emptyset \neq \Supp(A_{j,\star}) \subseteq \Supp(A_{i,\star})\\
\0&\mbox{otherwise}.
\end{cases}
\end{eqnarray}
By the previous remark, if $(\lid{A})_{i,j}$ is non-zero, then $i,j \in {\rm dom}(A)$.
\begin{proposition}
\label{prop:idmpt}
Let $S$ be an idempotent semifield and $A \in \mat{n}{S}$. Then
\begin{enumerate}[\rm(i)]
\item $\lid{A}\lid{A}=\lid{A}$.
\item $\lid{A} A = A$.
\item If ${\rm dom}(A)=[n]$ and $BA=A$, then  $B \preceq \lid{A}$ and $B\lid{A} \preceq \lid{A}$.
\end{enumerate}
\end{proposition}

\begin{proof}
(i) First note that for all $i,j,k \in [n]$ we have:
$$(\lid{A})_{i,j}   = \langle A_{j, \star} | A_{i,\star} \rangle \geq \langle A_{j, \star} |A_{k, \star} \rangle \langle A_{k, \star} | A_{i, \star}\rangle = (\lid{A})_{k,j} (\lid{A})_{i,k},$$
and so
$$(\lid{A})_{i,j}   \geq \bigvee_{k=1}^n(\lid{A})_{i,k}(\lid{A})_{k,j}  = (\lid{A}\lid{A})_{i,j}.$$

If $i \not\in {\rm dom}(A)$, then row $i$ of $\lid{A}$ is zero, and so too is row $i$ of the product $\lid{A}\lid{A}$. If $i \in {\rm dom}(A)$, then $(\lid{A})_{i,i} = \langle A_{i, \star} | A_{i, \star}\rangle =\1$, and it follows that for all $j \in [n]$ we have
$$(\lid{A} \lid{A})_{i,j} = \bigvee_{k=1}^{n} (\lid{A})_{i,k} (\lid{A})_{k,j} \geq (\lid{A})_{i,i}(\lid{A})_{i,j} = (\lid{A})_{i,j}.$$

(ii) Noting that $\langle A_{k,\star} | A_{i,\star} \rangle A_{k, \star} \preceq A_{i, \star}$ yields
$$(\lid{A}A)_{i,j}  = \bigvee_{k=1}^n (\lid{A})_{i,k}A_{k,j}  = \bigvee_{k=1}^n \langle A_{k,\star} | A_{i,\star} \rangle  A_{k,j} \leq  A_{i,j}.$$
If $i \not\in {\rm dom}(A)$, then for all $j$ we have $A_{i,j}=(\lid{A})_{i,j}=\0$. If $i \in {\rm dom}(A)$, then $(\lid{A}A)_{i,j} \geq  (\lid{A})_{i,i} A_{i,j}= A_{i,j}$ for all $j \in [n]$.

(iii) Suppose that $BA=A$. Then for all $i,j,k$ we have $B_{i,k}A_{k,j} \leq A_{i,j}$, and hence for all $i,k$ we have
$$B_{i,k} \leq \bigwedge_{j: A_{k,j}\neq 0} A_{i,j}A_{k,j}^{-1} = (\lid{A})_{k,j},$$
where the equality follows from the fact that ${\rm dom}(A)=[n]$. This shows that $B \preceq \lid{A}$, and hence $B\lid{A} \preceq \lid{A}\lid{A} =\lid{A}$.
\end{proof}

As we shall see shortly, the map $A \mapsto \lid{A}$ (along with a corresponding left-right dual map $A \mapsto \rid{A}$) allows us to show that certain subsemigroups of $\mat{n}{S}$ are Fountain. The notation $\lid{A}$ and $\rid{A}$ stems from the theory of regular semigroups and their generalisations; we use brackets in the superscript to distinguish these from the \emph{Kleene plus/star operations} as applied to matrices, which arise naturally in applications of Boolean and tropical matrices, for example.

\begin{proposition}
\label{partial}
Let $U$ denote the set of idempotents of $\mat{n}{S}$ with all diagonal entries equal to $\1$. If $A,B \in \mat{n}{S}$ with ${\rm dom}(A)={\rm dom}(B)=[n]$, then
$$A \Rt B  \Rightarrow \lid{A} = \lid{B}, \mbox{ and } A \Rtu{U} B  \Leftrightarrow \lid{A} = \lid{B}.$$
\end{proposition}
\begin{proof}
 By Proposition \ref{prop:idmpt}, we see that $\lid{A}$ is an idempotent acting as a left identity for $A$, and $\lid{B}$ is an idempotent acting as a left identity for $B$. It then follows from the definition of $\Rt$ that if $A \Rt B$, then $\lid{A}  B = B$ and $\lid{B}  A = A$. Applying Part (iii) of Proposition \ref{prop:idmpt} twice then gives $\lid{A}=\lid{B}$. Noting that ${\rm dom}(A)={\rm dom}(B)=[n]$ ensures that $\lid{A}, \lid{B} \in U$, the same argument shows that $A \Rtu{U} B$ implies $\lid{A}=\lid{B}$.

Suppose that $\lid{A} = \lid{B}$. If $X \in U$ with $XA = A$, then by Proposition \ref{prop:idmpt}(iii), $X \preceq \lid{A} = \lid{B}$. It then follows from Lemma \ref{ordered} (i) and Proposition \ref{prop:idmpt}(ii) that $XB \preceq \lid{B}B = B$, whilst $B \preceq XB$ by Lemma \ref{ordered} (ii).

\end{proof}

The matrices $A$ for which ${\rm dom}(A) =[n]$ are precisely those for which the matrix $(A/A)$ of \cite{CGQ97} is an idempotent of $M_n(S)$ acting as a left identity for $A$:
\begin{lemma}
Let $S$ be an idempotent semifield, and $A \in \mat{n}{S}$. Then the following are equivalent:
\begin{enumerate}[\rm (i)]
\item  $(A / A) \in \mat{n}{S}$;
\item  ${\rm dom}(A)=[n]$;
\item  $(A / A)  = \lid{A}$.
 \end{enumerate}\end{lemma}

\begin{proof}
 By definition, if $(A / A) \in \mat{n}{S}$, then for all $i,j$ there exists $k$ such that $A_{i,k} \setminus A_{j,k} \neq
\pmb{\top}$, or in other words, every row of $A$ is non-zero. In this case, it is easy to see that the $(i,j)$th entry of $A/A$ is given by
\begin{eqnarray*}
(A/A)_{i,j} &=& \bigwedge_k(A_{j,k} \setminus A_{i,k}) = \bigwedge_{k: A_{j,k} \neq 0} (A_{j,k} \setminus A_{i,k}) =\bigwedge_{k: A_{j,k} \neq 0} ( A_{i,k}A_{j,k}^{-1})= (\lid{A})_{i,j}
\end{eqnarray*}
This shows that (i) implies (ii) and (ii) implies (iii). Since by definition $\lid{A} \in \mat{n}{S}$, that (iii) implies (i) is trivial.
\end{proof}

\begin{theorem}
\label{subsemigroup}
Let $S$ be an idempotent semifield, $T$ a subsemigroup of $\mat{n}{S}$, $U$ the set of all idempotents of $\mat{n}{S}$ whose diagonal entries are all equal to $\1$, and write $V=T \cap U$. Assume that ${\rm dom}(A)=[n]$ and $\lid{A} \in T$ for all $A \in T$. Then for all $A, B \in T$ we have
\begin{enumerate}[\rm (i)]
\item $A \Rtu{V} \lid{A}$ in $T$;
\item if $A \in V$, then $\lid{A}=A$; and
\item $A \Rtu{V} B$ in $T$ if and only if $\lid{A}=\lid{B}$. 
\end{enumerate}
In particular, if $T$ admits an involutary anti-isomorphism $\varphi$ with $\varphi(V)=V$, then $T$ is $V$-Fountain, with each $\Rtu{V}$-class and each $\Ltu{V}$-class containing a unique idempotent of $V$.
\end{theorem}

\begin{proof}
(i) Let $A \in T$. By Proposition \ref{prop:idmpt} we know that $\lid{A}$ is an idempotent which acts as a left identity on $A$. By assumption, $\lid{A} \in T$, and since ${\rm dom}(A)=[n]$, it follows from the definition that all diagonal entries are equal to $\1$. Thus $\lid{A} \in V$ and it suffices to show that for all idempotents $F \in V$ with $FA=A$ we have $F\lid{A}=\lid{A}$. Since ${\rm dom}(A)=[n]$, we have $F \lid{A}\preceq \lid{A}$ by Proposition \ref{prop:idmpt}(iii). Since $F \in V$, we have $F_{i,i}=\1$ for all $i \in [n]$, and hence $\lid{A} \preceq F\lid{A}$ by Lemma \ref{ordered}. This shows that $A \Rtu{V} \lid{A}$ in $T$.

(ii) Since $A$ is idempotent, Proposition \ref{prop:idmpt} (iii) gives $A \preceq \lid{A}$. But since all diagonal entries of $A$ are equal to $\1$, Proposition \ref{prop:idmpt} and Lemma \ref{ordered}, give $\lid{A} \preceq \lid{A}A =A$. Thus $\lid{A}=A$ for all $A \in V$.

(iii) If $\lid{A}=\lid{B}$, then by Part (i) we have $A \Rtu{V} \lid{A} =\lid{B}\Rtu{V} B$. Suppose that $A \Rtu{V} B$ in $T$. Then in particular, $\lid{A}B=B$ and $\lid{B}A=A$. Applying Proposition \ref{prop:idmpt} (iii) then gives $\lid{A}=\lid{B}$.

In the case where $T$ admits an involutary anti-isomorphism fixing $V$, that $T$ is $V$-Fountain now follows from Part (i) together with Lemma \ref{l-r}. If $E, F \in V$ are $\Rtu{V}$-related then Parts (ii) and (iii) yield $E=\lid{E}=\lid{F}=F$.  
\end{proof}

\emph{From now on let $U$ denote the set of idempotents of $\mat{n}{S}$ having all diagonal entries equal to $\1$.} In the following subsections we apply Theorem \ref{subsemigroup} to exhibit several $V$-Fountain subsemigroups of $M_n(S)$, for an appropriate set of idempotents $V \subseteq U$.

\subsection{The semigroup without zero rows or zero columns is $U$-Fountain}
\label{well}
Let $W_n(S)$ denote the set of all matrices $A \in \mat{n}{S}$ with ${\rm dom}(A)={\rm im}(A)=[n]$. It is readily verified that $W_n(S)$ is a submonoid of $\mat{n}{S}$ containing $U$. (In the case where $S = \bool$, the semigroup $W_n(\bool)$ corresponds to the monoid of binary relations that are both left- and right-total.)

\begin{corollary}
\label{well-behaved}
Let $S$ be an idempotent semifield. The monoid $W_n(S)$ is $U$-Fountain. Each $\Rtu{U}$-class and each $\Ltu{U}$-class contains a unique idempotent of $U$, given by the maps $A \mapsto \lid{A}$ and $A \mapsto (\lid{(A^T)})^T$.
\end{corollary}

\begin{proof}
The transpose map restricts to give an involutary anti-automorphism of $W_n(S)$, mapping $U$ to itself. Thus in order to show that $W_n(S)$ is $U$-Fountain, by Lemma \ref{l-r} it suffices to show that each $\Rtu{U}$-class of $W_n(S)$ contains an idempotent of $U$. By definition, ${\rm dom}(A)=[n]$ for all $A \in W_n(S)$, and hence $\lid{A} \in U \subseteq W_n(S)$.  Theorem \ref{subsemigroup} with $T=W_n(S)$, $V=U$ and $\varphi$ the transpose map now yields that each $\Rtu{U}$-class contains a unique idempotent of $U$, specified by the map $A \mapsto \lid{A}$. Dually, via the transpose map, each $\Ltu{U}$-class contains a unique idempotent of $U$, giving the desired result. 
\end{proof}

It is natural to ask whether $W_n(S)$ is $K$-Fountain for some $K \supseteq U$. Indeed, is $W_n(S)$ Fountain?

By an abuse of notation, let us denote by $\mat{n}{S^*}$ the subsemigroup of $\mat{n}{S}$ whose entries are all non-zero.
\begin{corollary}
\label{fullfinite}
Let $S$ be an idempotent semifield. The semigroup $\mat{n}{S^*}$ is $V$-Fountain, where $V$ is the set of all idempotents of $\mat{n}{S^*}$ with all diagonal entries equal to $\1$.  Each $\Rtu{V}$-class and each $\Ltu{V}$-class contains a unique idempotent of $V$, given by the maps $A \mapsto \lid{A}$ and $A \mapsto (\lid{(A^T)})^T$.
\end{corollary}

\begin{proof}
The transpose map is an involutary anti-isomorphism on $\mat{n}{S^*}$, which maps $V$ bijectively to itself. Thus by Theorem \ref{subsemigroup} we just need to show that if $A \in \mat{n}{S^*}$ then so is $\lid{A}$. This is clearly the case, by \eqref{formula:Aplus}. \end{proof}

\begin{remark}  In the case where $S^* = (\mathbb{R}, +)$, there is a geometric way to view the previous result. We first recall that each matrix $A \in \mat{n}{\trop^*}$ has unique tropical eigenvalue, and that if this value is a non-positive real number, then the max-plus Kleene star (formed by taking the component-wise maximum of all powers of $A$ together with identity matrix of $\mat{n}{\trop}$), is a well-defined\footnote{If the eigenvalue is positive, then the powers of $A$ do not stabilise. However, in many situations it suffices to consider the Kleene star of a suitable scaing of $A$.} idempotent (see for example \cite{maxplus}). The set $V$ in Corollary \ref{fullfinite} is precisely the set of all Kleene star matrices in $\mat{n}{\trop^*}$. (Whilst $\lid{A} \in  V$, in general $\lid{A}$ is \emph{not} equal to the Kleene star of $A$; the latter need not be a left identity for $A$.)

To show that $\mat{n}{\trop^*}$ is $V$-Fountain, by arguing as in Proposition \ref{Fountainicity} it suffices to show that the intersection:
$${\rm ColFix}_V(A):=\bigcap_{\substack{A \Rless F,\\ F \in V}} {\rm Col}(F)$$
is equal to the column space of an idempotent in $V$. The column space of a matrix $A \in \mat{n}{\trop^*}$ may be naturally identified with a subset of $\mathbb{R}^n$, by removing the single element at $-\infty$ which does not lie in this set. The projectivisation map $\mathcal{P}: \mathbb{R}^n \rightarrow \mathbb{R}^{n-1}$ given by $(x_1, \ldots, x_n) \mapsto (x_1-x_n, \ldots, x_{n-1}-x_n)$ identifies those elements of ${\rm Col}(A)$ which are tropical scalings of each other, and the image of ${\rm Col}(A)$ under this map is a compact subset of $\mathbb{R}^{n-1}$. It follows from \cite[Theorem A and Theorem B]{JKconvex} that for all $F \in V$, ${\rm Col}(F)$ is max-plus convex, min-plus convex and Euclidean convex in $\mathbb{R}^n$. Thus ${\rm ColFix}_V(A)$ also has these three properties. Moreover, since this intersection has compact projectivisation, \cite[Theorem A]{JKconvex} then gives that ${\rm ColFix}_V(A)$ is the max-plus column space of a max-plus Kleene star.
\end{remark}
In light of the previous result, a very natural question is whether $\mat{n}{S^*}$ is Fountain.

Let ${\rm Sym}(n)$ denote the symmetric group acting on the set $[n]$. It is easy to see that
$$H_n(S) = \{ A \in \mat{n}{S}: \exists \sigma \in {\rm Sym}(n), \forall i \in [n],  A_{i,\sigma(i)} \neq \0\}$$
is a submonoid of $W_n(S)$.

\begin{corollary}
\label{Hall}
Let $S$ be an idempotent semifield.  The monoid $H_n(S)$ is $U$-Fountain. Each $\Rtu{U}$-class and each $\Ltu{U}$-class contains a unique idempotent of $U$, given by the maps $A \mapsto \lid{A}$ and $A \mapsto (\lid{(A^T)})^T$.
\end{corollary}

\begin{proof}
The transpose map is an involutary anti-isomorphism of $H_n(S)$, since taking transpose exchanges $\sigma$ with $\sigma^{-1}$. The set $U$ is contained in $H_n(S)$, since for each $E \in U$ we can take $\sigma$ to be the identity permutation, and $U$ is clearly closed under transpose. Noting that if $A \in H_n(S) \subseteq W_n(S)$, we automatically have $\lid{A} \in U$, the result then follows immediately from Theorem \ref{subsemigroup}.\end{proof}

\begin{remark}Suppose that $E$ is an idempotent of $H_n(\bool)$. There there exists $\sigma \in {\rm Sym}(n)$ such that $E_{i, \sigma(i)}=\1$ for all $i \in [n]$. If $\sigma^k(i)=i$, we obtain
$$E_{i,i} = (E^k)_{i,i} \geq E_{i,\sigma(i)}E_{\sigma(i), \sigma^2(i)} \cdots E_{\sigma^{k-1}(i), i} = \1,$$
showing that each idempotent of $H_n(\bool)$ lies in $U$, and hence $H_n(\bool)$ is Fountain. (For $S \neq \bool$, this argument only shows that the diagonal elements of each idempotent are non-zero; the idempotents of $H_n(S)$ do not lie in $U$ in general.) Alternatively, we note that the monoid $H_n(\bool)$ may be identified with the monoid of all Hall relations (those relations containing a perfect matching) on a set of cardinality $n$. Since the latter is known to be a finite block group, it is Fountain by \cite[Corollary 3.2]{MS17}. Another very natural question is whether $H_n(S)$ is Fountain.
\end{remark}

Let $R_n([\0, \1])$ denote the set of $n \times n$ matrices with all diagonal entries equal to $\1$, with all remaining entries lying in the interval $[\0, \1]$. Noting that $[\0, \1]$ is a subsemiring of the semifield $S$, it is easy to see that $R_n([\0, \1])$ is a submonoid of $W_n(S)$.

\begin{corollary}
\label{reflex}
Let $S$ be a linearly ordered idempotent semifield. The monoid $R_n([\0, \1])$ is Fountain. Each $\Rt$-class and each $\Lt$-class contains a unique idempotent, given by the maps $A \mapsto \lid{A}$ and $A \mapsto (\lid{(A^T)})^T$.
\end{corollary}

\begin{proof}
The idempotents elements of this monoid are clearly contained in the set of idempotents $U$ from Theorem \ref{subsemigroup}. The transpose map is an involutary anti-isomorphism of $R_n([\0,\1])$. It therefore remains to show that for $A \in R_n([\0, \1])$ we have $\lid{A} \in R_n([\0, \1])$. Given $A \in R_n([\0,\1])$ it is clear that the diagonal entries of $\lid{A}$ are all equal to $\1$. Suppose for contradiction that $(\lid{A})_{i,j} > \1$ for some $i,j \in [n]$. Then
$$(\lid{A}A)_{i,j} \geq (\lid{A})_{i,j} A_{j,j}  = (\lid{A})_{i,j} >\1,$$
contradicting that $\lid{A}A = A$. Since $S$ is assumed to be linearly ordered, we must then have $(\lid{A})_{i,j} \leq 1$ for all $i,j$ and hence $\lid{A} \in R_n([\0, \1])$.
\end{proof}

\begin{remark}
In the Boolean case, notice that the monoid $R_n([\0, \1])$ is the subsemigoup of $\mat{n}{\bool}$ defined by the single condition that all diagonal entries are equal to $\1$, and hence may be identified with the monoid of all reflexive binary relations on a set of cardinality $n$. Since the latter is a finite $\J$-trivial monoid (see for example, \cite{Volkov}), the fact that it is Fountain can be seen as a consequence of \cite[Corollary 3.2]{MS17}.
\end{remark}

\subsection{Triangular matrix semigroups with full diagonal are Fountain}
\label{subsec:Uni}
Proposition \ref{notFountain} gives an indication that the failure of $\upper{n}{S}$ to be Fountain lies in the degenerate behaviour of  matrices with diagonal entries equal to $\0$. For each subset $J \subseteq [n]$ write $UT_n^J(S)$ to denote the subset of $\upper{n}{S}$ whose non-zero diagonal entries lie precisely in positions $(j,j)$ for $j \in J$. Since $S$  has no zero divisors, it is easy to see that $UT_n^J(S)$ is a subsemigroup of $\upper{n}{S}$. It is clear from the definitions that $\upper{n}{S}$  is the meet semilattice of component semigroups $UT_n^J(S)$, with partial order corresponding to inclusion of subsets. (Specifically, if $A \in UT_n^J(S)$ and $B \in UT_n^K(S)$ it is easy to see that $AB, BA \in UT_n^{J \cap K}(S)$.) These components are not usually Fountain. For example, arguing as in the proof of Proposition \ref{notFountain} the component $UT_n^J(S)$ is not Fountain for $n \geq 4$ and  $\{2,3,4\} \subseteq J \subsetneq [n]$. 

From now on, let $\full{n}{S}$ denote the component of $\upper{n}{S}$ consisting of those upper triangular matrices $A$ with $A_{i,i} \neq \0$ for all $i \in [n]$. We say that the elements of $\full{n}{S}$ have ``full diagonal''.  Suppose that $A \in \full{n}{S}$ is idempotent. Thus $A_{i,i}^2= A_{i,i}$, and since $A_{i,i} \neq \0$, we must have $A_{i,i}=\1$ for all $i \in [n]$, giving $E(\full{n}{S}) = U \cap \uni{n}{S} = E(\uni{n}{S})$.

\begin{definition}
For all $A \in \full{n}{S}$ we write $\rid{A}:=\Delta(\lid{(\Delta(A))})$, where $\Delta$ is the involutary anti-isomorphism from Lemma \ref{delta} and $\lid{A}$ is the idempotent defined by \eqref{formula:Aplus}. 
\end{definition}

\begin{corollary}
\label{full}
Let $S$ be an idempotent semifield. Then $\full{n}{S}$ is Fountain and each $\Rt$-class and each $\Lt$-class contains a unique idempotent, given by the maps $A \mapsto \lid{A}$ and $A \mapsto \rid{A}$.
\end{corollary}
\begin{proof}
Since ${\rm dom}(A)=[n]$, we have $(\lid{A})_{i,i} =\1$. Moreover, it follows easily from \eqref{formula:Aplus} that for all $i>j$, $(\lid{A})_{i,j} =\0$, since in this case ${\rm Supp}(A_{j, \star}) \not\subseteq {\rm Supp}(A_{i,\star})$. Thus $\lid{A} \in \full{n}{S}$. As observed above, $E(\full{n}{S}) \subseteq U$. The result now follows from Theorem \ref{subsemigroup} via the anti-isomorphism $\Delta$, which clearly restricts to an anti-isomorphism of $\full{n}{S}$, mapping $E(\full{n}{S})$ to itself.
\end{proof}

\begin{corollary}
\label{uniS}
Let $S$ be an idempotent semifield. Then the monoid of unitriangular matrices $\uni{n}{S}$ is Fountain and each $\Rt$-class and each $\Lt$-class contains a unique idempotent, given by the maps $A \mapsto \lid{A}$ and $A \mapsto \rid{A}$.
\end{corollary}

\begin{proof}
We consider the setup of Theorem \ref{subsemigroup} with $T = \uni{n}{S}$, $\varphi=\Delta$ and $V=E(\uni{n}{S})$. Clearly $\Delta(E(\uni{n}{S})) = E(\uni{n}{S})$, and the argument of Corollary \ref{full} shows that if $A \in \uni{n}{S}$, then $\lid{A} \in \uni{n}{S}$.
\end{proof}

\begin{remark} In the case where $S^*$ is the trivial group, we obtain the monoid of all unitriangular Boolean matrices. This is a finite $\J$-trivial monoid (see for example, \cite{Volkov}), and so the fact that it is Fountain can be seen as a consequence of \cite[Corollary 3.2]{MS17}.\\
\end{remark}

By an abuse of notation, let us write $\upper{n}{S^*}$  to denote the subsemigroup of $\upper{n}{S}$ consisting of those upper triangular matrices whose entries on and above the diagonal are non-zero. Likewise, we write $\uni{n}{S^*}$  to denote the subsemigroup of $\uni{n}{S}$ consisting of those upper triangular matrices with all diagonal entries equal to $\1$ and whose entries  above the diagonal are non-zero. (In the case where $S =\bool$, notice that $\upper{n}{S^*}=\uni{n}{S^*}$ is the trivial group.)

\begin{corollary}
\label{FountainS*}
Let $S$ be an idempotent semifield. Then the semigroups $\upper{n}{S^*}$  and $\uni{n}{S^*}$ are Fountain. Each $\Rt$-class and each $\Lt$-class $\upper{n}{S^*}$ (respectively, $\uni{n}{S^*}$) contains a unique idempotent, given by the maps $A \mapsto \lid{A}$ and $A \mapsto \rid{A}$.
\end{corollary}

\begin{proof}
Consider the setup of Theorem \ref{subsemigroup} with $T = \upper{n}{S^*}$, $\varphi=\Delta$ and $V=E(\upper{n}{S^*})=E(\uni{n}{S^*})$. Clearly $\Delta(V) = V$, so it suffices to show that if $A \in \upper{n}{S^*}$, then $\lid{A} \in \upper{n}{S^*}$. Arguing as before, we see that $(\lid{A})_{i,i}= \1$ for all $i$ and if $j<i$, then $(\lid{A})_{i,j}=\0$. On the other hand, for $j>i$ we have $\emptyset \neq {\rm Supp}(A_{j, \star}) \subseteq {\rm Supp}(A_{i, \star})$, so that $(\lid{A})_{i,j}$ is the least upper bound of a finite set of elements of $S^*$ and hence lies in $S^*$.

Finally, consider the setup of Theorem \ref{subsemigroup} with $T = \uni{n}{S^*}\subseteq \full{n}{S}$, $\varphi=\Delta$ and $V=E(\uni{n}{S^*})$. Arguing as above we find that $\lid{A} \in \uni{n}{S^*}$ for all $A \in \uni{n}{S^*}$. 
\end{proof}

We write $D_n(S^*)$ to denote the set of invertible diagonal matrices. For any $A\in \full{n}{S}$, we write $\diag{A}$ to denote the element of $D_n(S^*)$ with $(i,i)$ entry equal to $A_{i,i}$ for all $i$, and write $\rnorm{A}$ and $\lnorm{A}$ to denote the unique unitriangular matrices satisfying $A= \rnorm{A}\diag{A}$ and $A=\diag{A}\lnorm{A}$.

\begin{theorem}
\label{thm:uppertilde}
Let $A, B \in \full{n}{S}$ (respectively, $\upper{n}{S^*}$). Then the following are equivalent:
\begin{enumerate}[\rm (i)]
\item $A \Rt B$ in $\full{n}{S}$ (respectively, $\upper{n}{S^*}$);
\item $\rnorm{A} \Rt \rnorm{B}$ in $\full{n}{S}$ (respectively, $\upper{n}{S^*}$);
\item $\rnorm{A} \Rt \rnorm{B}$ in $\uni{n}{S}$ (respectively, $\uni{n}{S^*}$);
\item $\lid{A} = \lid{B}$;
\item $\lid{(\rnorm{A})}=\lid{(\rnorm{B})}$.
\end{enumerate}
\end{theorem}

\begin{proof}
The equivalence of (i) and (iv) and of (ii), (iii) and (v) follows from Corollary \ref{full}, Corollary \ref{uniS} and Corollary \ref{FountainS*}. Let $A \in \full{n}{S}$. To complete the proof, we show that $\lid{A} = \lid{(\rnorm{A})}$, hence giving that (iv) and (v) are equivalent (indeed, the same statement).  Since $A=\rnorm{A}\diag{A}$, for all $i,j \in [n]$ we have $\rnorm{A}_{i,j} = A_{i,j}A_{j,j}^{-1}$, and so by definition
\begin{eqnarray*}
(\lid{(\rnorm{A})})_{i,j} = \bigwedge_{j \leq k \leq n}\rnorm{A}_{i,k}(\rnorm{A}_{j,k})^{-1} = \bigwedge_{j \leq k \leq n}A_{i,k}A_{k,k}^{-1}(A_{j,k}A_{k,k}^{-1})^{-1} = \bigwedge_{j \leq k \leq n}A_{i,k}A_{j,k}^{-1} = (\lid{A})_{i,j}.
\end{eqnarray*}
\end{proof}

\begin{corollary}
\label{one-onecorr}
There is a one-one correspondence between:
\begin{enumerate}
  \item $\full{n}{S}/\Rt$ and $E(\uni{n}{S})$;
  \item $\uni{n}{S}/\Rt$ and $E(\uni{n}{S})$;
  \item $\upper{n}{S^*}/\Rt$ and $E(\uni{n}{S^*})$;
  \item $\uni{n}{S}/\Rt$ and $E(\uni{n}{S^*})$.
\end{enumerate}
\end{corollary}

\subsection{The idempotent generated subsemigroups of $\full{n}{S}$ and $\upper{n}{S^*}$}
\label{subsec:idmpt}
First observe that 
$$E \in \full{n}{S} \mbox{ is idempotent if and only if for all } i,j,k \in [n], E_{i,i}= \1 \mbox{ and } E_{i,k}E_{k,j} \leq E_{i,j}.$$
\begin{theorem}
\label{idmpgensmgp}
Let $S$ be an idempotent semifield. The semigroup $\uni{n}{S}$ is the idempotent generated subsemigroup of $\full{n}{S}$. Every element of $\uni{n}{S}$ can be written as a product of at most $n-1$ idempotent elements.
\end{theorem}
\begin{proof}
As we have already observed, the idempotents of $\full{n}{S}$ form a subset of $\uni{n}{S}$. Thus it suffices to prove the second statement.

Given $X\in \uni{n}{S}$, let $\mathfrak{X}$ denote the matrix with $(i,j)$ entry given by $(\lid{X})_{i,j} \wedge (\rid{X})_{i,j}$, where $\rid{X}$ denotes the right identity of $X$ dual to $\lid{X}$ under the involution $\Delta$. Since $\rid{X}, \lid{X} \in \uni{n}{S}$, we see that $\mathfrak{X} \in \uni{n}{S}$. Since all diagonal entries of $\mathfrak{X}$ are equal to $\1$, for all $i,j \in [n]$ we have $(\mathfrak{X}\mathfrak{X})_{i,j} \geq \mathfrak{X}_{i,i}\mathfrak{X}_{i,j} = \mathfrak{X}_{i,j}$. In fact, $\mathfrak{X}$ is idempotent since
\begin{eqnarray*}
(\mathfrak{X}\mathfrak{X})_{i,j} &=& \bigvee_{k=1}^n ((\lid{X})_{i,k} \wedge  (\rid{X})_{i,k})((\lid{X})_{k,j} \wedge (\rid{X})_{k,j})\\
&\leq & \bigvee_{k=1}^n ((\lid{X})_{i,k}(\lid{X})_{k,j} \wedge (\rid{X})_{i,k}(\rid{X})_{k,j}) \leq  \bigvee_{k=1}^n ((\lid{X})_{i,j} \wedge (\rid{X})_{i,j}) = \mathfrak{X}_{i,j},
\end{eqnarray*}
where the last inequality follows from the fact that $\lid{X}$ and $\rid{X}$ are idempotents. Moreover, since $\lid{X}$ (respectively, $\rid{X}$) is a left ( resp. right) identity for $X$, it follows that for all $i,j,k \in [n]$ we also have
\begin{eqnarray}
\label{barXX}
\mathfrak{X}_{i,k}X_{k,j}= ((\lid{X})_{i,k} \wedge (\rid{X})_{i,k})X_{k,j} \leq (\lid{X})_{i,k}X_{k,j} \leq X_{i,j},\\
\label{XbarX}
X_{i,k}\mathfrak{X}_{k,j} =X_{i,k}((\lid{X})_{k,j} \wedge (\rid{X})_{k,j}) \leq X_{i,k}(\rid{X})_{k,j} \leq X_{i,j}.
\end{eqnarray}
In particular (by taking $k=j$ in \eqref{barXX}) we have $\mathfrak{X} \preceq X$. Arguing as above, one also sees that $\mathfrak{X}X=X=X\mathfrak{X}$.

Now, for each $h\in [n]$ let $X(h) \in
\uni{n}{S}$ be the matrix with entries given by

\begin{equation*}
X(h)_{i,j}=
\begin{cases}
X_{i,j}\text{ if }i<h\leq j \\
\mathfrak{X}_{i,j}\text{ otherwise.}
\end{cases}
\end{equation*}
In other words, $X(1)=\mathfrak{X}$, and for $h \geq 2$ we have
\begin{equation*}
X(h)=\left(
\begin{array}{cccc|cccc}
\1 & \mathfrak{X}_{1,2} & \cdots  & \mathfrak{X}_{1,(h-1)}
& X_{1,h} & X_{1,(h+1)} & \cdots  & X_{1,n} \\
\0 & \1 & \cdots  & \mathfrak{X}_{2,(h-1)} & X_{2,h} &
X_{2,(h+1)} & \cdots  & X_{2,n} \\
\vdots  & \vdots  & \ddots  & \vdots  & \vdots  & \vdots  &
\vdots  & \vdots  \\
\0 & \0 & \cdots  & \1 & X_{h-1,h} &
X_{h-1,h+1} & \cdots  & X_{h-1,n} \\ \hline
\0 & \0 & \cdots  & \0 & \1 & \mathfrak{X}_{h,(h+1)} & \cdots  & \mathfrak{X}_{h,n} \\
\0 & \0 & \cdots  & \0 & \0 & \1 &
\cdots  & \mathfrak{X}_{h+1,n} \\
\vdots  & \vdots  & \cdots  & \vdots  & \vdots  & \vdots  &
\ddots  & \vdots  \\
\0 & \0 & \cdots  & \0 & \0 & \0 &
\cdots  & \1
\end{array}
\right)
\end{equation*}
We shall show that for each $h \in [n]$ the matrix $X(h)$ is idempotent, and moreover that $X = X(n)X(n-1)\cdots X(2)$. To this end for $j,h \in[n]$ let us consider the following elements of $S^n$:
$$x(j,h) =\left(
\begin{array}{c}
X_{1,j}\\
\vdots\\
X_{h-1,j}\\
\hline
\mathfrak{X}_{h,j}\\
\vdots\\
\mathfrak{X}_{n,j}\\
\end{array}
\right), \;\;\; y(j,h) =\left(
\begin{array}{c}
\mathfrak{X}_{1,j}\\
\vdots\\
\mathfrak{X}_{h-1,j}\\
\hline
\0\\
\vdots\\
\0\\
\end{array}
\right).
$$
It is easy to see that $x(j,n)$ is the $j$th column of $X$. We claim that for all $2 \leq h \leq n$ the following equations hold:
\begin{eqnarray}
X(h)y(j,h)&=&y(j,h), \;\;\; \mbox{ for all } 1 \leq j \leq h-1, \label{acty}\\
X(h)x(j,h)&=&x(j,h), \;\;\; \mbox{ for all } h \leq j \leq n,\label{actx}\\
X(h)x(j,h-1)&=&x(j,h), \;\;\; \mbox{ for all } j \in [n].\label{actx'}
\end{eqnarray}
Before proving this, let us show how these equations can be used to give the desired result. Notice that \eqref{acty} and \eqref{actx} yield that $X(h)$ is idempotent, since $y(1,h), \ldots, y(h-1,h), x(h,h), \ldots, x(n,h)$ are precisely the columns of  $X(h)$. Since $X_{1,1} = \mathfrak{X}_{1,1}=\1$, we note that the $j$th column of $X(2)$ is equal to $x(j,2)$. Thus repeated application of \eqref{actx'} shows that the $j$th column of the product $X(h) \cdots X(2)$ is equal to $x(j,h)$ for all $j \in [n]$. Noting that the final row of $\mathfrak{X}$ and $X$ agree then yields that $X=X(n) \cdots X(2)$.

To prove that \eqref{acty} holds, let $1 \leq j \leq h-1$ and consider the $i$th entry of the product $X(h)y(j,h)$, given by
$$\bigvee_{k=1}^n X(h)_{i,k}y(j,h)_k = \bigvee_{k=1}^{h-1} X(h)_{i,k}\mathfrak{X}_{k,j}.$$
If $i \geq h$, then since the first $h-1$ entries in row $i$ of $X(h)$ are zero, the supremum above is clearly equal to $\0$. If $i \leq h-1$, then the supremum above becomes $\bigvee_{k=1}^{h-1} \mathfrak{X}_{i,k}\mathfrak{X}_{k,j} = \bigvee_{i \leq k \leq j} \mathfrak{X}_{i,k}\mathfrak{X}_{k,j} =\mathfrak{X}_{i,j},$  since $\mathfrak{X}$ is upper triangular and idempotent. Thus in both cases we obtain the $i$th entry of $y(j,h)$.

To prove that \eqref{actx} holds, let $h \leq j \leq n$ and consider the $i$th entry of the product $X(h)x(j,h)$ given by
$$\bigvee_{k=1}^n X(h)_{i,k}x(j,h)_k = \bigvee_{k=1}^{h-1} X(h)_{i,k}X_{k,j} \vee \bigvee_{k=h}^{j} X(h)_{i,k}\mathfrak{X}_{k,j}.$$
If $i \leq h-1$ then this becomes $\bigvee_{k=1}^{h-1} \mathfrak{X}_{i,k}X_{k,j} \vee \bigvee_{k=h}^{j} X_{i,k}\mathfrak{X}_{k,j} \leq X_{i,j},$ where the inequality follows from \eqref{barXX} and \eqref{XbarX}; in fact the term corresponding to $k=i$ yields equality. On the other hand, if $i \geq h$, then the supremum above becomes $\bigvee_{k=h}^{n} \mathfrak{X}_{i,k}\mathfrak{X}_{k,j} = \bigvee_{i \leq k \leq j} \mathfrak{X}_{i,k}\mathfrak{X}_{k,j} =\mathfrak{X}_{i,j}$, as before. Thus in both cases we obtain the $i$th entry of $x(j,h)$.

Finally, to prove that \eqref{actx'} holds, consider the $i$th entry of the product $X(h)x(j,h-1)$ given by
$$\bigvee_{k=1}^n X(h)_{i,k}x(j,h-1)_k =  \bigvee_{k=1}^{h-2} X(h)_{i,k}X_{k,j} \vee X(h)_{i,h-1} \mathfrak{X}_{h-1,j}\vee \bigvee_{k=h}^{j} X(h)_{i,k}\mathfrak{X}_{k,j}.$$
This time there are three cases to consider. If $i \leq h-2$ then the supremum above becomes
$$\bigvee_{k=1}^{h-2} \mathfrak{X}_{i,k}X_{k,j} \vee  \mathfrak{X}_{i,h-1} \mathfrak{X}_{h-1,j} \vee \bigvee_{k=h}^{j} X_{i,k}\mathfrak{X}_{k,j} \leq X_{i,j},$$
where the inequality follows from \eqref{XbarX} and \eqref{barXX}, together with the fact that $\mathfrak{X}$ is idempotent and $\mathfrak{X} \preceq X$. The term corresponding to $k=i$ then yields equality.

If $i=h-1$, then the supremum above becomes
$$\mathfrak{X}_{h-1,j}\vee  \bigvee_{k=h}^{j} X_{h-1,k}\mathfrak{X}_{k,j} \leq X_{i,j},$$
where the inequality follows from \eqref{XbarX} and the fact that $\mathfrak{X} \preceq X$.  The term corresponding to $k=j$ then yields equality.

If $i \geq h$, then the supremum above becomes $\bigvee_{i \leq k \leq j} \mathfrak{X}_{i,k}\mathfrak{X}_{k,j} =\mathfrak{X}_{i,j}$  as before. Thus in all three cases the $i$th entry agrees with the $i$th entry of $x(j,h)$. This completes the proof.
\end{proof}

\begin{corollary}
\label{idmpgennozero}
Let $S$ be an idempotent semifield. The semigroup $\uni{n}{S^*}$ is the idempotent generated subsemigroup of $\upper{n}{S^*}$. Every element of $\uni{n}{S^*}$ can be written as a product of at most $n-1$ idempotent elements.
\end{corollary}
\begin{proof}
The idempotents of $\upper{n}{S^*}$ form a subset of $\uni{n}{S^*}$. For $X \in \uni{n}{S^*}$, notice that each of the idempotents $X(h)$ constructed in proof of the previous proposition lie in $\uni{n}{S^*}$.
\end{proof}

\begin{corollary}\label{semidirect}Every element of $\full{n}{S}$ (respectively, $\upper{n}{S^*}$) is a product of a diagonal matrix in $D_n(S^*)$ and $n-1$ idempotents of $\full{n}{S}$ (respectively, $\upper{n}{S^*}$. Moreover, we have the following decompostions as semidirect products of semigroups:
\begin{eqnarray*}
\full{n}{S}&\simeq& \uni{n}{S} \ast D_{n}(S^*) \simeq (E(\full{n}{S}))^{n-1} \ast D_{n}(S^*),\\
\upper{n}{S^*}&\simeq& \uni{n}{S^*} \ast D_{n}(S^*) \simeq (E(\upper{n}{S^*}))^{n-1}\ast D_{n}(S^*).
\end{eqnarray*}
\end{corollary}
\begin{proof}
Since each $A \in \full{n}{S}$ may be written as $A=\rnorm{A}\diag{A}$ where $\rnorm{A}\in \uni{n}{S}$ and $\diag{A} \in D_{n}(S^*)$, the first statement follows from Theorem \ref{idmpgensmgp} (respectively, Corollary \ref{idmpgennozero}). Recall that the semidirect product $M \ast N$ of (not necessarily commutative)semigroups $(M, +)$ and $(N, \boxplus )$ with respect to a left action $\cdot: N\times M \rightarrow M$ is the set $M \times N$ with multiplication given by $(m_1,n_2)(m_2, n_2) = (m_1 + n_1\cdot m_2, n_1\boxplus n_2)$. The elements of $\full{n}{S}$ are easily seen to be in one to one correspondence with the elements of $\uni{n}{S} \times D_{n}(S^*)$ via the map identifying $A$ with the pair $(\rnorm{A}, \diag{A})$. For $A, B \in \full{n}{S}$ one then has that $AB= \rnorm{A}\diag{A}\rnorm{B}\diag{B}=(\rnorm{A}\diag{A}\rnorm{B}\diag{A}^{-1})(\diag{A}\diag{B})$, showing that $\full{n}{S}$ is isomorphic to the semidirect product defined by the conjugation action of $D_{n}(S^*)$ on $\uni{n}{S}$. Similarly, $\upper{n}{S^*}$ is isomorphic to the semidirect product defined by the conjugation action of $D_{n}(S^*)$ on $\uni{n}{S^*}$.
\end{proof}

\begin{question} What is the idempotent generated subsemigroup of $M_n(S)$,$ UT_n(S)$,
$U_n(S)$, $UT_n(S^*)$ and $U_n(S^*)$ for an arbitrary (semi)ring?
\end{question} 

Both $\uni{n}{S}$ and $\uni{n}{S^*}$ are $\J$-trivial (see for example, \cite[Lemma 4.1]{JF}), and so each of the semigroups $\full{n}{S}$ and $\upper{n}{S^*}$ is a semidirect product of a $\J$-trivial semigroup and a group acting by conjugation. We compare Corollary \ref{semidirect} with \cite[Corollary 3.2]{MS17}, which states that each finite monoid whose idempotents generate an $\R$-trivial monoid is right-Fountain; the canonical example of such a monoid being the semidirect product of a finite $\R$-trivial monoid with a finite group acting by automorphisms.
\begin{remark}
If $X \in \uni{n}{S}$, then it is easy to see that for all $s \in \mathbb{N}$ and all $i,j \in [n]$ we have
\begin{eqnarray*}
(X^s)_{i,j} &=& \bigvee_{i \leq  r_1 \leq \cdots \leq r_{s-1} \leq j} X_{i,r_1} X_{r_1,r_2}\cdots X_{r_{s-1},j}= X_{i,j} \vee  \bigvee_{t=1}^{m}\bigvee_{i < p_1 < \cdots < p_{t} < j} X_{i,p_1} X_{p_1,p_2}\cdots X_{p_{t},j},
\end{eqnarray*}
where $m = {\rm min}(|i-j|-1, s-1)$. It follows from this that $X \preceq X^2\preceq \cdots \preceq X^{n-1} = X^n$, and hence $\uni{n}{S}$ is aperiodic.  
\end{remark}

\begin{proposition}
Let $E, F \in \full{n}{S}$ be idempotents. Then for all $m \in \mathbb{N}$ such that $2m \geq n+1$ one has:
$$(EF)^m = (EF)^m E= E(FE)^m = (FE)^m F= F(EF)^m = (FE)^m.$$
\end{proposition}

\begin{proof}
Let $m \in \mathbb{N}$, $1 \leq i \leq j \leq n$ and set
\begin{eqnarray*}
\Sigma_1 &=& ((EF)^m)_{i,j} = \bigvee_{i \leq  r_1 \leq \cdots \leq r_{2m-1} \leq j} E_{i,r_1} F_{r_1,r_2}\cdots F_{r_{2m-1},j},\\
\Sigma_2 &=& ((EF)^mE)_{i,j} = \bigvee_{i \leq  r_1 \leq \cdots \leq r_{2m} \leq j} E_{i,r_1} F_{r_1,r_2}\cdots F_{r_{2m-1},r_{2m}}E_{r_{2m},j},\\
\Sigma_3 &=& ((FE)^m)_{i,j} = \bigvee_{i \leq  r_1 \leq \cdots \leq r_{2m-1} \leq j} F_{i,r_1} E_{r_1,r_2}\cdots E_{r_{2m-1},j}, \;\;\;
\mbox{ and}\\
\Sigma_4 &=& ((FE)^mF)_{i,j} = \bigvee_{i \leq  r_1 \leq \cdots \leq r_{2m} \leq j} F_{i,r_1} E_{r_1,r_2}\cdots E_{r_{2m-1},r_{2m}}F_{r_{2m},j}.\end{eqnarray*}

Now let $t$ denote a term in the supremum $\Sigma_1$. Noting that $t=tE_{j,j}$, by setting $r_{2m}=j$ in the expression for $\Sigma_2$ above we see that $t$ is also a term in the supremum $\Sigma_2$. This shows that $\Sigma_1 \leq \Sigma_2$. Since $(FE)^mE = F(EF)^m$  and $(EF)^mE = E(FE)^m$, similar arguments show that $\Sigma_1 \leq \Sigma_4$ and $\Sigma_3 \leq \Sigma_2, \Sigma_4$.

Conversely, suppose that $s$ is a term of the supremum $\Sigma_2$. If $2m \geq n$, then for $1 \leq i \leq r_1 \leq \cdots \leq r_{2m} \leq j \leq n$, we must have that $r_a=r_{a+1}$ for at least two distinct $a$ in the range $0 \leq a \leq 2m$ (with $r_0=i$ and $r_{2m+1}=j$) so as not to contradict $2m+2 \geq n+2 $. At least one of these values of $a$ must satisfy $a>0$. If $a=2m$, then clearly $s$ is a term of $\Sigma_1$, and we obtain that $\Sigma_1=\Sigma_2$. Suppose then that $0 <a < 2m$, then one of $E_{r_a, r_{a+1}}$ or $F_{r_a, r_{a+1}}$ is a factor of $s$. By idempotency of $E$ and $F$ together with the fact that $E_{r_a, r_{a+1}}=F_{r_a, r_{a+1}}=\1$ one has
\begin{eqnarray*}
F_{r_{a-1}, r_{a}}E_{r_a, r_{a+1}} F_{r_{a+1}, r_{a+2}} &=& F_{r_{a-1}, r_{a}}F_{r_a, r_{a+1}} F_{r_{a+1}, r_{a+2}} \leq F_{r_{a-1},r_{a+2}}, \mbox{ and }\\
E_{r_{a-1}, r_{a}}F_{r_a, r_{a+1}} E_{r_{a+1}, r_{a+2}} &=& E_{r_{a-1}, r_{a}}E_{r_a, r_{a+1} }E_{r_{a+1}, r_{a+2}} \leq E_{r_{a-1}r_{a+2}},
\end{eqnarray*}
from which it is then easy to see that $s \leq \Sigma_1$ and hence $\Sigma_1=\Sigma_2$. Switching the roles of $E$ and $F$ in the above yields $\Sigma_3=\Sigma_4$, and hence $\Sigma_1 \leq \Sigma_4 = \Sigma_3 \leq \Sigma_2 = \Sigma_1.$
\end{proof}

In the following section we consider in more detail the generalised regularity properties of the semigroups $\mat{n}{S}$, $\upper{n}{S}$ and $\uni{n}{S}$ in the case where the idempotent semifield $S$ has a linear order.

\section{Matrices over linearly ordered idempotent semifields}
\label{sec:lin}
Throughout this section let $\linz$ be a linearly ordered idempotent semifield.

\subsection{Monoids of binary relations}
\label{bin}
Consider first the case where the underlying group is trivial, that is, where $\linz$ is the Boolean semifield $\mathbb{B} = \{0,1\}$. As noted in Lemma \ref{semifields}, $\bool$ is the unique idempotent semifield containing finitely many elements.  We recall from \cite[Theorem 1.1.1]{Kim82} that each submodule $M$ of $\bool^n$ has a unique minimal generating set (basis), consisting of those non-zero elements $x \in M$ which cannot be expressed as a Boolean sum of elements of $M$ occurring beneath $x$ (with respect to the obvious partial order on Boolean vectors).

\begin{proposition}
\label{basis}
If every non-zero row of $A \in \mat{n}{\bool}$ is contained in the unique basis of ${\rm Row}(A)$, then $A \Rt \lid{A}$.
\end{proposition}

\begin{proof}
Suppose that $XA=A$. We first show that whenever an entry of $\lid{A}$ is zero, the corresponding entry of the product $X\lid{A}$ must also be zero (that is, $X\lid{A} \preceq \lid{A}$). For all $i,j \in [n]$ we have
\begin{equation}
\label{XA+}
(X\lid{A})_{j,i} = \bigvee_{k=1}^n X_{j,k} (\lid{A})_{k,i} = \bigvee_{k: X_{j,k}=1} (\lid{A})_{k,i}.
\end{equation}
If $j \not\in {\rm dom}(A)$, then it follows from the definition that both column and row $j$ of $\lid{A}$ are zero, and so we require that column and row $j$ of the product $X \lid{A}$ should be zero. This is immediate for columns.  Since row $j$ of $A$ is zero, it follows from $XA=A$ that whenever $X_{j,k}=1$ we must have that row $k$ of $A$ is zero, which in turn yields that row $k$ of $\lid{A}$ is zero, and hence the supremum in \eqref{XA+} is equal to zero.

All remaining zeros in $\lid{A}$ lie in positions $(i,j)$ such that $0 \neq A_{j, \star} \not\preceq A_{i,\star}$. Let $(i,j)$ be such an index. If $X_{i,k}=1$, we note that it follows from the fact that $XA=A$ that $A_{i,t} \geq X_{i,k} A_{k,t} = A_{k,t}$  for all $t$. Or in other words, $A_{k, \star} \preceq A_{i, \star}$. Thus if $(\lid{A})_{k,j}=1$ for some $k$ with $X_{i,k}=1$ we obtain $A_{j, \star} \preceq A_{k, \star} \preceq A_{i,\star}$, contradicting that $A_{j, \star} \not\preceq A_{i,\star}$. So we must have that  the supremum in \eqref{XA+} is equal to zero. This completes the proof that $X\lid{A} \preceq \lid{A}$.

From the above argument, if $j \not\in {\rm dom}(A)$, then $\lid{A}$ and $X\lid{A}$ clearly agree in row $j$. Suppose then that $j \in {\rm dom}(A)$ and $A_{j, \star}$ is contained in the basis for ${\rm Row}(A)$. Since $XA=A$ we require in particular that the $j$th row of $XA$ is equal to the $j$th row of $A$. The $j$th row of $XA$ is a linear combination of the rows of $A$, namely,
$(XA)_{j, \star} = \bigvee_{k=1}^n X_{j,k} A_{k, \star}$. Since this linear combination is equal to a basis element $A_{j, \star}$ of ${\rm Row}(A)$, it follows in particular that there exists $s \in [n]$ with $X_{j,s}=1$ and $A_{s, \star}=A_{j, \star}$. Thus $X_{j,s}=1$ and $(\lid{A})_{j,s}=(\lid{A})_{s,j}=1$. But then for all $i$ we have:
$$(X\lid{A})_{j,i} = \bigvee_{k: X_{j,k}=1} (\lid{A})_{k,i} \geq  (\lid{A})_{s,i}  \geq (\lid{A})_{s,j}(\lid{A})_{j,i} = (\lid{A})_{j,i},$$
where the final inequality is due to the fact that $\lid{A}$ is idempotent. Since we have already proved that $X\lid{A} \preceq \lid{A}$, it now follows that the two matrices must agree in row $j$.
\end{proof}

\begin{theorem}\label{thm:B}
The generalised regularity properties of the monoids $\mat{n}{\bool}$, $\upper{n}{\bool}$ and $\uni{n}{\bool}$ can be summarised as follows:

\begin{center}
\begin{tabular}{ l |l| l| l| l }
& $n=1$ & $n=2$ & $n=3$ & $n \geq 4$\\
\hline
$\mat{n}{\bool}$& Regular (band) & Regular & Fountain & Not Fountain\\
$\upper{n}{\bool}$& Regular(band) & Abundant & Fountain & Not Fountain\\
$\uni{n}{\bool}$& Regular (group) & Regular (band) & Fountain &  Fountain\\
\end{tabular}
\end{center}
\end{theorem}
\begin{proof}
(i) The monoid $\mat{n}{\bool}$, or equivalently the monoid $\mathcal{B}_n$ of binary relations on an $n$ element set, is well known to be regular if and only if $n \leq 2$; for $n=1$ this is isomorphic to the multiplicative monoid of $\bool$ consisting of two idempotents. Since $\mathbb{B}$ is exact, it follows from Theorem \ref{exact} above, that $\mat{n}{\bool}$ is also not abundant for $n \geq3$. Moreover, Proposition \ref{notFountain} shows that $\mathcal{B}_n$ is not Fountain for all $n \geq4$. By means of Lemma \ref{l-r}, Lemma \ref{delta} and Proposition \ref{basis}, it is straightforward to verify that $\mat{3}{\bool}$ is Fountain; one just has to check that any $A \in \mat{3}{\bool}$ with ${\rm dom}(A)=3$ and all rows distinct is $\Rt$-related to an idempotent (since a submodule of $\bool^n$  generated by fewer than three non-zero elements must contain each of these elements in its basis). Thus it suffices to consider the case where $A_{\sigma(1), \star} = A_{\sigma(2),\star} \vee A_{\sigma(3),\star}$ for some permutation $\sigma$ of $[3]$. In thhis case one finds that $A$ is $\Rt$-related to the matrix with rows $\sigma(2)$ and $\sigma(3)$ equal to the corresponding rows of the identity matrix, and row $\sigma(1)$ equal to the join of these two.

(ii) Since $\upper{1}{\bool}=\mat{1}{\bool}$, it is clearly regular. Proposition \ref{notReg} shows that $\upper{n}{\bool}$ is not regular for $n \geq 2$, Proposition \ref{notAbund} shows that $\upper{n}{\bool}$ is not abundant for $n\geq 3$ and Proposition \ref{notFountain} shows that $\upper{n}{\bool}$ is not Fountain for $n \geq 4$. This leaves only the cases $n=2$ and $n=3$ to consider in more detail.

The monoid $\upper{2}{\bool}$ is easily seen to be abundant; seven of the eight elements are actually idempotents. The remaining element is the non-regular element considered in the proof of Proposition \ref{notReg}, and it is easily verified that:
$$\left(
\begin{array}{c c}
1 & 1 \\
0 & 0\end{array}\right)^2=\left(
\begin{array}{c c}
1 & 1 \\
0 & 0\end{array}\right) \Rs\left(
\begin{array}{c c}
0 & 1 \\
0 & 0\end{array}\right) \Ls \left(
\begin{array}{c c}
0 & 1 \\
0 & 1\end{array}\right) = \left(
\begin{array}{c c}
0 & 1 \\
0 & 1\end{array}\right)^2.$$

For $n=3$, first notice that we cannot hope to use the idempotent construction $A \mapsto \lid{A}$ to show that $\upper{3}{\bool}$ is Fountain, since in general this map does not restrict to a map $\upper{n}{S} \rightarrow \upper{n}{S}$. For example,
$$A=\left(
\begin{array}{c c c}
1 & 1 & 1 \\
0 & 0 & 1\\
0 & 0 & 1
\end{array}\right), \lid{A}=\left(
\begin{array}{c c c}
1 & 1 & 1 \\
0 & 1 & 1\\
0 & 1 & 1
\end{array}\right).$$

By direct computation we find that forty-one of the sixty-four elements of $\upper{3}{\bool}$ are idempotents. They have the form
\begin{eqnarray*}
&&\left(
\begin{array}{c c c}
1 & x & y \\
0 & 1 & z\\
0 & 0 & 1
\end{array}\right), \left(
\begin{array}{c c c}
1 & x & y \\
0 & 1 & z\\
0 & 0 & 0
\end{array}\right), \left(
\begin{array}{c c c}
0 & x & y \\
0 & 1 & z\\
0 & 0 & 1
\end{array}\right), \left(
\begin{array}{c c c}
1 & x & y \\
0 & 0 & z\\
0 & 0 & 1
\end{array}\right),\\
&&\left(
\begin{array}{c c c}
1 & u & w \\
0 & 0 & 0\\
0 & 0 & 0
\end{array}\right), \left(
\begin{array}{c c c}
0 & 0 & w \\
0 & 0 & v\\
0 & 0 & 1
\end{array}\right), \left(
\begin{array}{c c c}
0 & u & uv \\
0 & 1 & v\\
0 & 0 & 0
\end{array}\right), \left(
\begin{array}{c c c}
0 & 0 & 0 \\
0 & 0 & 0\\
0 & 0 & 0
\end{array}\right),\\
\end{eqnarray*}
where $u,v,w,x,y,z \in \bool$ with $xz \leq y$.

By Lemma \ref{Rs}, Theorem \ref{exact} and Theorem \ref{LR}, any element of $\upper{3}{\bool}$ with column space equal to the column space of one of the above matrices is $\Rs$-related to the same, and in this way one finds that all but the following four matrices are $\Rs$-related to an idempotent: 
$$X_1=\left(
\begin{array}{c c c}
1 & 1 & 0 \\
0 & 1 & 1\\
0 & 0 & 1
\end{array}\right), X_2=\left(
\begin{array}{c c c}
0 & 1 & 0 \\
0 & 1 & 1\\
0 & 0 & 1
\end{array}\right), X_3=\left(
\begin{array}{c c c}
0 & 0 & 1 \\
0 & 1 & 1\\
0 & 0 & 0
\end{array}\right), X_4=\left(
\begin{array}{c c c}
0 & 1 & 0 \\
0 & 1 & 1\\
0 & 0 & 0
\end{array}\right).$$
It is then straightforward to verify that:
$$X_1 \Rt \left(
\begin{array}{c c c}
1 & 0 & 0 \\
0 & 1 & 1\\
0 & 0 & 1
\end{array}\right) \Rt X_2\qquad \mbox{and} \qquad X_3 \Rt \left(
\begin{array}{c c c}
1 & 1 & 1 \\
0 & 1 & 0\\
0 & 0 & 0
\end{array}\right) \Rt X_4.$$

(iii) Since $\uni{1}{\bool}$ is the trivial group, it is obviously regular. The monoid $\uni{2}{\bool}$ contains two idempotent elements, and hence is also regular. Proposition \ref{notAbund} shows that for all $n \geq 3$, $\uni{n}{\bool}$ is not abundant, whilst Theorem \ref{uniS} shows that $\uni{n}{\bool}$ is Fountain for all $n$.
\end{proof}

\begin{remark}
The representation theory of monoid algebras $KM$ of certain\footnote{Namely, those for which all sandwich matrices indexed by idempotents are right invertible over the group algebra of the corresponding group $\H$-class} finite (right) Fountain monoids $M$ over a field $K$  is elucidated in \cite{MS17}, where descriptions of the projective indecomposables and quiver of such a monoid algebra are given. Thus the results of \cite{MS17} cannot be applied to the monoid of binary relations for $n>3$, although there is scope for these results to apply for $n\leq 3$. A preliminary study of the (rational) representation theory in the case $n=3$ can already be found in \cite{Bremner}.
\end{remark}

\subsection{Infinite linearly ordered semifields}
 We recall from \cite{Wagneur} that $V =\{v_1, \ldots, v_k\} \subset \linz^n$ is said to be \emph{linearly independent} if
$$v_i \not\in \left\{ \bigvee_{j \neq i} \lambda_j v_j: \lambda_j \in S\right\} \mbox{ for } 1 \leq i \leq k,$$
and that a \emph{basis} for an $\linz$-semimodule $M$ is a linearly independent generating set for $M$. By \cite[Corollary 4.7 and Theorem 5]{Wagneur}, every finitely generated $\linz$-submodule $M \subseteq \linz^n$ admits a basis, and moreover bases are unique up to scaling (that is, if $\{v_1, \ldots, v_k\}$ and $\{w_1, \ldots, w_j\}$ are two bases for $M$, then $k=j$ and for $i=1, \dots, j$ there exists $\lambda_i \in \lin$ such that $w_i=\lambda_iv_i$ ). In particular, the column space (respectively, row space)  of $A \in \mat{n}{\linz}$ admits a unique up to scaling basis; by the proof of \cite[Theorem 6]{Wagneur} this can be taken to be a subset of the set of columns (respectively, rows) of $A$.

\begin{proposition}
\label{prop:lin}
Let $A,B \in \mat{n}{\linz}$. Then
\begin{enumerate}[\rm(i)]

\item $\lidlid{A}=\lid{A}$.
\item $(\lid{A})_{i,i}  = \1$ for all $i \in {\rm dom}(A)$.
\item Suppose that ${\rm dom}(A)=[n]$ and $BA=A$. If $A_{j, \star}$ is contained in a basis of ${\rm Row}(A)$, then $(B\lid{A})_{j, \star}=(\lid{A})_{j, \star}$.
\end{enumerate}
\end{proposition}
\begin{proof}
(i) Suppose that $(\lid{A})_{i,j}\neq \0$. Then by definition $j \in {\rm dom}(A)$ and $\Supp(A_{j, \star}) \subseteq \Supp(A_{i, \star})$. It follows easily from the definition of $\lid{A}$ that $\Supp((\lid{A})_{j, \star}) \subseteq \Supp((\lid{A})_{i, \star})$. Also, since $j \in {\rm dom}(A)$ we note that $(\lid{A})_{j,j}=1$ and hence $j\in {\rm dom}(\lid{A})$. This shows that  $(\lidlid{A})_{i,j}\neq \0$. Thus there exists $s \in [n]$ such that $(\lidlid{A})_{i,j}(\lid{A})_{js} = (\lid{A})_{i,s}$. Since the inequality $(\lid{A})_{i,j} >(\lidlid{A})_{i,j}$ then contradicts that $\lid{A}$ is idempotent, we must have $(\lid{A})_{i,j} \leq (\lidlid{A})_{i,j}$. This shows that $\lid{A} \preceq \lidlid{A}$. On the other hand suppose that $(\lid{A})_{i,j} < (\lidlid{A})_{i,j}$ for some $i,j$. From the definition, we must have that $j \in {\rm dom}(\lidlid{A})$ and hence $j \in {\rm dom}(A)$. But then $(\lid{A})_{j,j}=1$ and
$$(\lidlid{A}\lid{A})_{i,j} \geq (\lidlid{A})_{i,j} (\lid{A})_{j,j} > (\lid{A})_{i,j},$$
contradicting Part (ii).

(ii) This is immediate from the definition.

(iii) Suppose that $A_{j, \star}$ is contained in a basis for ${\rm Row}(A)$. Since $BA=A$ we require in particular that the $j$th row of $BA$ is equal to the $j$th row of $A$. The $j$th row of $BA$ is a linear combination of the rows of $A$, namely,
$(BA)_{j, \star} = \bigvee_{k=1}^n B_{j,k} A_{k, \star}$. Since this linear combination is equal to a basis element $A_{j, \star}$ of ${\rm Row}(A)$, it follows in particular that there exists $s \in [n]$ with $B_{j,s}=1$ and $A_{s, \star}=A_{j, \star}$. Thus $B_{j,s}=1$ and $(\lid{A})_{j,s}=(\lid{A})_{s,j}=1$. But then for all $i$ we have:
$$(B\lid{A})_{j,i} = \bigvee_{k=1}^n B_{j,k}(\lid{A})_{k,i} \geq  (\lid{A})_{s,i}  \geq (\lid{A})_{s,j}(\lid{A})_{j,i} = (\lid{A})_{j,i},$$
where the final inequality is due to the fact that $\lid{A}$ is idempotent. By Propositon \ref{prop:idmpt} we have $B\lid{A} \preceq \lid{A}$, and so it follows that the two matrices 
 in row $j$.
\end{proof}

\begin{conjecture}
The generalised regularity properties of the monoids $\mat{n}{\linz}$, $\upper{n}{\linz}$ and $\uni{n}{\linz}$ can be summarised as follows:
\begin{center}
\begin{tabular}{ l |l| l| l| l }
& $n=1$ & $n=2$ & $n=3$ & $n \geq 4$\\
\hline
$\mat{n}{\linz}$& Regular & Regular & Fountain (?) & Not Fountain\\
$\upper{n}{\linz}$& Regular & Abundant & Fountain (?) & Not Fountain\\
$\uni{n}{\linz}$& Regular (group) & Regular (band) & Fountain &  Fountain\\
\end{tabular}
\end{center}
\end{conjecture}

\begin{remark}
(i) For $n=1$, $\mat{n}{\linz}$ is isomorphic to the multiplicative monoid of $\linz$;  a group with zero adjoined, which is plainly regular. For $a \in \linz$ define $\bar{a} = a^{-1}$ if $a \neq \0$ and $\bar{a}=\0$ otherwise. Thus $a \bar{a} a = a$ for all $a \in \linz$ and $a \bar{a} \leq \1$ with equality if $a \neq \0$. From this it is clear that
$$\left(\begin{array}{c c} \0 & b\\c & \0 \end{array}\right)\left(\begin{array}{c c} \0 & \bar{c}\\\bar{b} & \0 \end{array}\right)\left(\begin{array}{c c} \0 & b\\c & \0 \end{array}\right) = \left(\begin{array}{c c} \0 & b\\c & \0 \end{array}\right).$$
Suppose now that $a,b,c,d \in \linz$ with at least one of $a$ or $d$ non-zero. If $ad \geq bc$ we have $\1 \geq bc \bar{a} \bar{d}$, $a \geq bc\bar{d}$ and $d\geq bc\bar{a}$. From this one readily verifies that
$$\left(\begin{array}{c c} a & b\\c & d \end{array}\right)\left(\begin{array}{c c} \bar{a} & b\bar{a}\bar{d}\\c\bar{a}\bar{d} & \bar{d} \end{array}\right)\left(\begin{array}{c c} a & b\\c & d \end{array}\right) = \left(\begin{array}{c c} a & b\\c & d \end{array}\right).$$
Likewise, if $bc \geq ad$ one finds that
$$\left(\begin{array}{c c} a & b\\c & d \end{array}\right)\left(\begin{array}{c c}\bar{c} & d\bar{c}\bar{b}\\a\bar{c}\bar{b} & \bar{b} \end{array}\right)\left(\begin{array}{c c} a & b\\c & d \end{array}\right) = \left(\begin{array}{c c} a & b\\c & d \end{array}\right).$$
 Thus we conclude that $\mat{2}{\linz}$ is regular. Since $\linz$ is exact, it follows from Corollary \ref{exact} above, that $\mat{n}{\linz}$ is not abundant for $n \geq3$. Moreover, Proposition \ref{notFountain} shows that $\mat{n}{\linz}$ is not Fountain for all $n \geq4$.

(ii) Since $\upper{1}{\linz}=\mat{1}{\linz}$, it is clearly regular. Propositions \ref{notReg}, \ref{notAbund} and \ref{notFountain} show that $\upper{n}{\linz}$ is not regular for $n \geq 2$, not abundant for $n\geq 3$ and not Fountain for $n \geq 4$. Let us show that the monoid $\upper{2}{\linz}$ is abundant. The elements of $\uni{2}{S}$ are easily seen to be idempotent, and so it is clear that each element of $\full{2}{S}$ is $\R$-related to an idempotent in $\upper{2}{S}$. For the remaining three components of $\upper{2}{S}$, notice that one can apply Lemma \ref{Rs} and Theorem \ref{exact} to deduce that every element of $\upper{2}{S}$ is $\Rs$-related to an idempotent.

(iii) By Theorem \ref{uniS}, $\uni{n}{\linz}$ is Fountain for all $n$. For $n=1$ this is clearly the trivial group. For $n=2$, it is a band. For $n \geq 3$, we have that $\uni{n}{\linz}$ is not abundant.

Thus we leave open only the question of whether $\mat{3}{\linz}$  and $\upper{3}{\linz}$ are Fountain.
\end{remark}

\section{The semigroups $\upper{n}{\lin}$}
\label{sec:upper}

For the remainder of the paper we focus on the Fountain semigroups $\upper{n}{\lin}$, for $\linz \neq \bool$.  This family of semigroups turns out to have a particularly interesting structure. In the case where $\linz =\trop$, these semigroups have been studied by \cite{TaylorThesis}. (For $\linz=\bool$ these are, of course, trivial groups.)

Recall that $D_n(\lin)$ denotes the set of invertible diagonal matrices and that for any $A\in \upper{n}{\lin}$, we write $A= \rnorm{A}\diag{A}$ and $A=\diag{A}\lnorm{A}$ where $\diag{A}$ is the element of $D_n(\lin)$ with diagonal entries equal to those of $A$, and $\rnorm{A}, \lnorm{A} \in \uni{n}{\lin}$.

\begin{theorem}
\label{thm:upper*}
Let $A, B \in \upper{n}{\lin}$. Then the following are equivalent:

\begin{enumerate}[\rm (i)]
\item $A \Rs B$ in $\upper{n}{\lin}$;
\item $A \Rs B$ in $\upper{n}{\linz}$;
\item $A \Rs B$ in $\mat{n}{\linz}$;
\item $A \R B$ in $\mat{n}{\linz}$;
\item $A \R B$ in $\upper{n}{\linz}$;
\item $A \R B$ in $\upper{n}{\lin}$;
\item $A=BX$ for some $X \in D_n(\lin)$;
\item $\rnorm{A}=\rnorm{B}$.
\end{enumerate}
\end{theorem}

\begin{proof}
Parts (ii) and (iii) are equivalent by Lemma \ref{Rs}. Since $\linz$ is exact, parts (iii) and (iv) are equivalent by Theorem \ref{exact}. Anti-negativity of $\linz$ means that the $i$th column of $A$ (respectively, $B$) cannot be expressed as a linear combination of columns $i+1,\ldots, n$ of $B$ (respectively, $A$). Thus if $A=BX$ and $B=AY$, then we must have $X, Y \in \upper{n}{\linz}$. This shows that (iv) and (v) are equivalent. The equivalence of (vii) and (viii), that (vi) implies (i), and that (vii) implies (iv) are obvious. To complete the proof we shall show that (iv) implies (vii),  that (v) implies (vi) and that  (i) implies (ii).

To see that (iv) implies (vii), first note that each column of $A$ cannot be expressed as a linear combination of the remaining columns. This shows that every column of $A$ is contained in a basis of ${\rm Col}(A)$. Since ${\rm Col}(A)={\rm Col}(B)$ and bases are unique up to permutation and scaling, we must have $A=BX$ where $X \in D_n(\lin)$.

Suppose that (v) holds, thus $A=BX$ for some $X \in \upper{n}{\linz}$. For each $i \leq j$ let $m(i,j) = \{ s: i \leq s \leq j, A_{i,j} = B_{i,s}X_{s,j}\}$. Since $A_{i,j} \neq \pmb{0}$ for all $i \leq j$, we must have $X_{s,j} \neq \pmb{0}$ for all $s \in m(i,j)$.
Let $X'$ be the matrix obtained by replacing all of the zero entries of $X$ lying on or above the diagonal by some element $g \in \lin$ such that for all $i \leq j $ and all $t$ one has $A_{i,j} > B_{i,t} g$. Then it is straightforward to check that
$$(BX')_{i,j} = \bigvee_{i \leq s \leq j} B_{i,s} X'_{s,j} =\bigvee_{s \in m(i,j)} B_{i,s} X_{s,j} =A_{i,j}.$$

To prove that (i) implies (ii), suppose for contradiction that $A \Rs B$ in $\upper{n}{\lin}$ but not in $\upper{n}{\linz}$. By symmetry of the relation, it suffices to consider the case where for some $X, Y \in \upper{n}{\linz}$ we have $XA=YA$, but $XB \neq YB$. Without loss of generality let us further suppose that  $p, q,r,s$ are such that
\begin{equation}
\label{assumption}
(XB)_{r,s} = X_{r,p}B_{p,s} > Y_{r,q}B_{q,s} = (YB)_{r,s}.
\end{equation}
Construct a new pair of matrices $X', Y'$ by replacing all of the $\pmb{0}$ entries of $X$ and $Y$ by an element $g \in \lin$ satisfying:
\begin{itemize}
\item[(1)] $g \leqq A_{i,j}A_{k,j}^{-1}, B_{i,j}B_{k,j}^{-1}$, for all $i \leq k \leq j$ and ,
\item[(2)]$X_{r,p}B_{p,s}> g B_{t,s}$ for all $t=1, \ldots, n$.
\end{itemize}
Notice that in doing so, terms corresponding to those zero terms that did not contribute to the maximum in any non-zero entry of the products ($XA, YA, XB, YB$) also do not contribute to the corresponding maximum in any of the new products ($X'A, Y'A, X'B, Y'B$). In particular, if $(XA)_{i,j}=(YA)_{i,j}$ is non-zero, then we find $(X'A)_{i,j} = (XA)_{i,j}=(YA)_{i,j} = (Y'A)_{i,j}$.

On the other hand, if $(XA)_{i,j} = (YA)_{i,j}=0$, then it follows that $X_{i,k}=Y_{i,k} = 0$ for all $i \leq k \leq j$. Replacing each of these $0$'s by $g$ then gives $(X'A)_{i,j} = g\bigvee_{i \leq k \leq j} A_{k,j}= (Y'A)_{i,j}$. This shows that we have $X'A = Y'A$.

Finally,
$$(X'B)_{r,s}  \geq X_{r,p} B_{p,s} > \bigvee_{r \leq t \leq s} Y'_{r,t} B_{t,s} = (Y'B)_{r,s}.$$
To see that the first inequality holds, note that \eqref{assumption} implies that $X_{r,p}$ is non-zero, and hence $X'_{r,p}=X_{r,p}$. The second inequality holds by \eqref{assumption} and our choice of $g$ satisfying (2). This gives $X'B \neq Y'B$. Since $X', Y' \in \upper{n}{\lin}$ this gives the desired contradiction.
\end{proof}

Together with its obvious left-right dual, the previous result yields:
\begin{corollary}
\label{thm:upperD}
Let $A, B \in \upper{n}{\lin}$. Then $A \D B$ in $\upper{n}{\lin}$if and only if  there exists $C \in \upper{n}{\lin}$ such that $\rnorm{A}=\rnorm{C}$ and $\lnorm{C}=\lnorm{B}$.
\end{corollary}

\begin{theorem}
\label{thm:uni*}
Let $A, B \in \uni{n}{\lin}$. Then the following are equivalent:
\begin{enumerate}[\rm (i)]
\item $A \Rs B$ in $\uni{n}{\lin}$;
\item $A \Rs B$ in $\uni{n}{\linz}$;
\item $A \J B$ in $\uni{n}{\lin}$;
\item $A \J B$ in $\uni{n}{\linz}$;
\item any of the equivalent conditions of Theorem 1;
\item $A=B$.
\end{enumerate}
\end{theorem}

\begin{proof}
Since $A = \rnorm{A}$ for all $A \in \uni{n}{\lin}$, it is clear that (v) implies (vi). Hence it suffices to show that each of (i)-(iv) implies (v). It is clear from the definitions that (iii) implies (iv), which by \cite[Lemma 4.1]{JF} implies (v). Since (ii) implies condition (i), it remains to show that (i) implies (v).

 Suppose that $A \Rs B$ in $\uni{n}{\lin}$. Let $X, Y \in \upper{n}{\lin}$ with $XA=YA$. Then $\diag{X}\lnorm{X}A=\diag{Y}\lnorm{Y}A$. By comparing the diagonal entries of $XA$ and $YA$ one finds that $\diag{X}=\diag{Y}$, and hence by cancellation $\lnorm{X} A = \lnorm{Y} A$. Since $A \Rs B$ in $\uni{n}{\lin}$ this gives $\lnorm{X} B = \lnorm{Y} B$. Left multipication by $\diag{X}$ then yields $XB=YB$. By symmetry of assumption, this shows that $A \Rs B$ in $\upper{n}{\lin}$, as required.
\end{proof}

\begin{corollary}
\label{one-onecorr}
There is a one-one correspondence between $\upper{n}{\lin}/\R$ and $\uni{n}{\lin}$.
\end{corollary}

\begin{proof}
Consider the mapping $^{\diamond}:\upper{n}{\lin}\rightarrow \uni{n}{\lin}$ where $A\mapsto \rnorm{A}$. By the definition of $\diamond$, $\rnorm {(\rnorm{A})}=\rnorm{A}$ and if $A \in \uni{n}{\lin}$ then $A=\rnorm{A}$.
        For any $A,B \in \upper{n}{\lin}$, $\rnorm{A}=\rnorm{B}$ implies that $$A=\rnorm{A}\diag{A}=\rnorm{B}\diag{A}=B (\diag{B})^{-1}\diag{A}.$$ From Theorem \ref{thm:upper*} we see that  $A$ $\R$ $B$ in $\upper{n}{\lin}$. Conversely, if $A$ $\R$ $B$ in $\upper{n}{\lin}$ then $\rnorm{A}$ $\R$ $A$ $\R$ $B$ $\R$ $\rnorm{B}$, which implies that $\rnorm{A}$ $\R$ $\rnorm{B}$. From Theorem \ref{thm:uni*}, it follows that $\rnorm{A}=\rnorm{B}$. Hence $\R$ is the kernel of the map $^{\diamond}$,  which assures the one-one correspondence.
\end{proof}

\begin{corollary}\label{regut}
For $\linz \neq \bool$, the generalised regularity properties of $\upper{n}{\lin}$ and $\uni{n}{\lin}$ can be summarised as follows:
\begin{center}
\begin{tabular}{ l |l| l| l}
& $n=1$ & $n=2$ & $n \geq 3$\\
\hline
$\upper{n}{\lin}$& Regular(group) & Regular (inverse) & Fountain\\
$\uni{n}{\lin}$& Regular (group) & Regular (semilattice) & Fountain\\
\end{tabular}
\end{center}
\end{corollary}

\begin{proof}
That $\upper{n}{\lin}$ and $\uni{n}{\lin}$ are Fountain for all $n$ follows from Theorem \ref{thm:uppertilde} together with the fact that $\lidlid{A}=\lid{A}$. It is clear that $\upper{1}{\lin}$ is isomorphic to the linearly ordered abelian group $(\lin,\cdot)$, whilst $\uni{1}{\lin}$ is isomorphic to the trivial group. The semigroup $\upper{2}{\lin}$ is inverse; if $A\in \upper{2}{\lin}$ then $A^{-1}$ is given by:\\
\begin{equation*}
A^{-1} = \left(
\begin{array}{c c}
(A_{1,1})^{-1} & A_{1,2}(A_{1,1})^{-1}(A_{2,2})^{-1} \\
\pmb{0} & (A_{2,2})^{-1}
\end{array}\right).
\end{equation*}
The semilattice of idempotents of $\upper{2}{\lin}$ is $\uni{2}{\lin}$, which as a semilattice is order isomorphic to $\lin$.  Finally, let $n \geq 3$, and consider the unitriangular matrix $G$ with all entries above the diagonal equal to $g>\pmb{1}$. A straightforward calculation reveals that this matrix is not idempotent, and hence by Theorem \ref{thm:uni*} it is neither regular nor abundant. Noting that $\{AX \in \uni{n}{\lin}: X \in D_n(\lin)\} = \{A\}$ we conclude from Theorem \ref{thm:upper*} that $\upper{n}{\lin}$ is neither regular nor abundant.
\end{proof}

\begin{theorem}\label{thm:upperH}
Each maximal subgroup in $\upper{n}{\lin}$ is isomorphic to $(\lin,.)$.
\end{theorem}

\begin{proof}
Let $A,E \in \upper{n}{\lin}$ be such that $E^2=E$ and $A \, \H \, E$. Then by Theorem \ref{thm:upper*}, $\rnorm{A}=E=\lnorm{A}$. Therefore
$$A=\rnorm{A} \diag{A}=E\diag{A} \text{ and   }A=\diag{A} \lnorm{A}=\diag{A}E.$$
For each $i\in[n]$, $A_{1,i}=E_{1,i}A_{i,i}=A_{1,1}E_{1,i}$ gives $A_{i,i}=A_{1,1}=\lambda$. Then we have $A=\lambda E$.
Conversely, every $A\in \upper{n}{\lin}$ such that $A=\lambda E$ is $\R$ and $\L$-related with $E$. Hence
$$H_E=\{\lambda E: \lambda \in \lin \}.$$
Since $\H$-classes of idempotents are precisely the maximal subgroups of $\upper{n}{\lin}$, every maximal subgroup is isomorphic to $(\lin,.)$.
\end{proof}

Theorem \ref{thm:upper*} does not generalise to give similar characterisations of $\R$ and $\Rs$ for $\upper{n}{\linz}$, as the following example illustrates.

\begin{example}
\label{notRestrict}
Let $S$ be an arbitrary semiring.

(i) For $n\geq 2$, let $A,B \in \upper{n}{S}$ be the matrices with $A_{1,n}=A_{1,n-1}=B_{1,n}=\pmb{1}$ and all other entries equal to $\pmb{0}$, we see that $A \R B$ in $\mat{n}{S}$. However, in order for $A=BX$ to hold, we require
$$\pmb{1}=A_{1,n-1} = \sum_{k=1}^nB_{1,k}X_{k,n-1} = X_{n,n-1},$$
so that $X$ is not upper triangular. Thus the $\R$-relation on $\upper{n}{S}$ is not the restriction of the $\R$-relation on $\mat{n}{S}$.

(i) For $n\geq 3$, let $A,B \in \upper{n}{S}$ be the matrices with $A_{1,n-1}=A_{2,n}=B_{1,n} = B_{2,n-1}=1$ and all other entries equal to $\pmb{0}$. It is clear that these matrices have identical row spaces and identical column spaces, and hence are $\H$-related in $\mat{n}{S}$. However, in order for $A=XB$ to hold, we require
$$\pmb{1}=A_{2,n} = \sum_{k=1}^nX_{2,k}B_{k,n} = X_{2,1},$$
so that $X$ is not upper triangular. Thus the $\H$-relation on $\upper{n}{S}$ is not the restriction of the $\R$-relation on $\mat{n}{S}$.
\end{example}

\begin{example}
\label{notRestrictD}
Let $S$ be an anti-negative semiring, and $n\geq 2$. Take $A, B \in \upper{n}{S}$ to be the matrices with $A_{2,n}=B_{1,n-1}=\pmb{1}$, and all other entries equal to $\pmb{0}$. Then it is clear that $A \L C \R B$ in $\mat{n}{S}$, where $C$ is the matrix with a single non-zero entry in position $(1,n)$. We show that there does not exist an upper triangular matrix $X$ simultaneously satisfying $A \L X$ and $X \R B$ in $\upper{n}{S}$.

If $A \L X$, then we would have $PX=A$ for some $P \in \upper{n}{S}$, giving
$$A_{i,j} = \sum_{k=i}^{j} P_{i,k} X_{k,j} \mbox{ for all } i \leq j.$$
Since $A_{2,n}=\pmb{1}$ we note that there must exist $l>1$ such that $P_{2,l} \neq \pmb{0} \neq X_{l,n}$. Since $A_{2,j}=\pmb{0}$ for each $j<n$, and $P_{2,l} \neq \pmb{0}$, we then deduce (using the anti-negativity of $S$) that $X_{l,j}=\pmb{0}$ for all $l \leq j <n$.

Now, if $X \R B$, then we would have $XQ=B$ for some $Q \in \upper{n}{S}$, giving
$$B_{i,j} = \sum_{k=i}^{j} X_{i,k} Q_{k,j} \mbox{ for all } i \leq j.$$
Since $B_{1,n-1}=\pmb{1}$ we note that there must exist $h>n$ such that $X_{1,h} \neq \pmb{0} \neq Q_{h,n-1}$. But now, the first row of $X$ contains a non-zero entry in position $h<n$, whilst each element of ${\rm Row}(A)$ does not. This tells us that $X$ cannot be written in the form $RA$ for any $R \in \mat{n}{S}$, thus contradicting that $X \L A$ in $\upper{n}{S}$. Hence the $\D$ relation in $\upper{n}{S}$ is \emph{not} the restriction of the $\D$ relation in $\mat{n}{S}$.
\end{example}

At this point it is useful to have a simplified formulation for $\lid A$ and $\rid B$, where
$A, B\in UT_n(\lin)$. Direct calculation from the formula given in \eqref{formula:Aplus}, together with the definition of the anti-isomorphism $\Delta$ yields that for $i\leq j$ we have
\begin{eqnarray}
\lid{A}_{i,j}&:=&\bigwedge \{ A_{i,k} (A_{j,k})^{-1}: j\leq k \leq n \}.
\label{Aplus}\\
\rid{B}_{i,j}&:=&\bigwedge \{ B_{k,j} (B_{k,i})^{-1}: 1\leq k\leq i \}
 \label{Astar}
\end{eqnarray}
Now suppose that $E$ is an idempotent in $\upper{n}{\lin}$ such that $\lid{A}=E=\rid{B}$. For all $1\leq j \leq n$, $E_{i,j}A_{j,j}\leq A_{i,j}$ and $B_{i,i}E_{i,j}\leq B_{i,j}$. Thus without any loss of generality, we can formulate
\begin{eqnarray}
\label{Ralpha}
A_{i,j}&:=&E_{i,j}A_{j,j}\alpha _{i,j}, \mbox{ where } \alpha_{i,j} \geq \1 \mbox{ is uniquely determined},\\
\label{Lbeta}
B_{i,j}&:=&B_{i,i}E_{i,j}\beta _{i,j}, \mbox{ where } \beta_{i,j} \geq \1 \mbox{ is uniquely determined}.
\end{eqnarray}

For each $i \in [n]$, $A_{i,i}=A_{i,i} \alpha _{i,i}$, and $E_{i,n} = A_{i,n}(A_{n,n})^{-1}$. From \eqref{Ralpha}, it follows that $\alpha_{i,i} = \alpha_{i,n} = \pmb{1}$. Similarly, it is easy to see that $\beta _{i,i}=\pmb{1}$ and $\beta _{1,i}=\pmb{1}$ for all $1\leq i\leq n$.

If $A,B \in \uni{n}{\lin}$ with $\lid{A}=E=\rid{B}$, then since $A_{i,i} = B_{i,i}=\pmb{1}$ \eqref{Ralpha} and \eqref{Lbeta}  become:
\begin{eqnarray}
A_{i,j}:=E_{i,j}\alpha _{i,j}\text{ where }\alpha _{i,j}\geq \pmb{1}
\label{Ralpha*}\\
B_{i,j}:=E_{i,j}\beta _{i,j}\text{ where }\beta _{i,j}\geq \pmb{1}.
\label{Lbeta*}
\end{eqnarray}

Since it is crucial to our later work, we record the following observation as  a lemma.

\begin{lemma}\label{lem:thetwoalphabetas} Let $A,B\in UT_n(\lin)$ and express
$A_{i,j}$ and $B_{i,j}$ as in \eqref{Ralpha} and \eqref{Lbeta}. Then the same elements
$\alpha_{i,j}$ and $\beta_{i,j}$ appear in the corresponding descriptions of
$A^{\diamond}$ and $B^{\bullet}$.
\end{lemma}
\begin{proof} This follows from the definitions and the facts that $(A^{\diamond})^{(+)}=\lid A$
and $(B^{\bullet})^{(*)}=\rid B$, and  $A= \rnorm{A}\diag{A}$ and $B=\diag{B}\lnorm{B}$.
\end{proof}

In many of the contexts in which Fountain semigroups have arisen, such as that of
Ehresmann and adequate semigroups \cite{Kambites:2011,Branco:2015}, restriction and ample semigroups \cite{Fountain:2009}, and wider classes \cite{Gould:2012}
$\Rt$ and $\Lt$ are, respectively, left and right congruences. We now show that not only is this not true of the relations $\Rt$ and $\Lt$ in  $UT_n(\lin)$ but that it fails in the most extreme manner possible, as we now state and prove.

\begin{proposition}\label{prop:leftcong}
Let $A,B\in \upper{n}{\lin}$. Then $CA$ $\Rt$ $CB$ for all $C\in \upper{n}{\lin}$ exactly if $A$ $\R$ $B$.
\end{proposition}

\begin{proof}
Let $A,B\in \upper{n}{\lin}$. It is clear that if $A$ $\R$ $B$ then $CA$ $\R$ $CB$ since $\R$ is a left congruence and hence $CA$ $\Rt$ $CB$. We prove the converse. Suppose $CA$ $\Rt$ $CB$ for all $C \in \upper{n}{\lin}$. Let $C\in \upper{n}{\lin}$ be given by:
\begin{equation*}
C_{i,j}=(\lid{A})_{i,j}\wedge (\lid{B})_{i,j}.
\end{equation*}
Then, $C \preceq \lid{A}$ and $C \preceq \lid{B}$. So $C_{i,k}A_{k,j}\leq \lid{A}_{i,k}A_{k,j}\leq A_{i,j}$ and $C_{i,k}B_{k,j}\leq \lid{B}_{i,k}B_{k,j}\leq B_{i,j}$ for all $1\leq i\leq k\leq
j\leq n$. Since $C_{i,i}= \pmb{1}$ so $(CA)_{i,j}=A_{i,j}$ and $(CB)_{i,j}=B_{i,j}$. By our assumption, we have $A$\ $\Rt$\ $B$ hence $\lid{A}=\lid{B}$.
Let us write $E=\lid{A}=\lid{B}$. Using Formulae \eqref{Ralpha}, we can write
\begin{equation}
A_{i,j}=A_{j,j}E_{i,j}\alpha _{i,j} \text{ and  } B_{i,j}=B_{j,j}E_{i,j}\gamma_{i,j}
\label{alphagamma}
\end{equation}
where $\alpha _{i,j}$, $\gamma_{i,j}\geq \pmb{1}$. For each $i\in [n]$, $\alpha _{i,i}= \alpha _{i,n}=\pmb{1}$ and $\gamma _{i,i}=\gamma _{i,n}=\pmb{1}$.

\textbf{Claim:} We claim that $\alpha _{i,j}=\gamma_{i,j}$ for each $1<i<j<n$.

\textbf{Proof of Claim}: Fix $i,\ell\in [n]$ where $i< \ell <n$. Suppose contrary that $\alpha _{i,\ell}\neq \gamma _{i,\ell}$ given $\alpha _{i,s}=\gamma_{i,s}$ for all $s$ such that $i<\ell<s\leq n$. Then formulae in \eqref{alphagamma} tell us that $$A_{i,s}=E_{i,s}A_{s,s}\alpha_{i,s}=E_{i,s}\gamma_{i,s}A_{s,s}=B_{i,s}(B_{s,s})^{-1}A_{s,s}$$ so each column vector $A_{*,s}$ is a scalar multiple of the corresponding vector $B_{*,s}$ where $\ell <s\leq n$.

Now $\alpha _{i,\ell}\neq \gamma_{i,\ell}$ implies that either $\gamma _{i,\ell}>\alpha _{i,\ell}$ or $\alpha _{i,\ell}>\gamma_{i,\ell}$. Without loss of generality, assume that $\gamma_{i,\ell}>\alpha _{i,\ell}$ and let $C\in \upper{n}{\lin}$ with its $i$th row as follows:
\begin{equation*}
C_{i,j}=\left\{
\begin{array}{lll}
E_{i,j}(\gamma _{i,\ell})^{-1} &  & \text{if }i\leq j<\ell \\
E_{i,\ell}(\gamma _{i,\ell})^{-1}\alpha _{i,\ell} &  & \text{if }j=\ell \\
E_{i,j}\gamma _{i,\ell} \hdots \gamma _{i,j} &  & \text{if }
\ell<j\leq n
\end{array}
\right.
\end{equation*}
while all other rows are the same as those of $E$.

Let $X \in M_n(\lin)$ and suppose that $EX=X$. Since all rows of $C$ except row $i$ agree with the rows of $E$, we see that all rows of $CX$ except the $i$th row will be equal to the corresponding row of $X$.

For $1 \leq j \leq n$, the $(i,j)$th entry of $CX$ is given by $(CX)_{i,j} = \bigvee_{i \leq k \leq n} C_{i,k}X_{k,j}$. Using the definition of the $i$th row of $C$ and the fact that $E_{i,i}=1$, we see that $(CX)_{i,j}$ is given by
\begin{eqnarray*}
 &&C_{i,i}X_{i,j} \vee \left(\bigvee_{i < k <l} C_{i,k}X_{k,j}\right) \vee C_{i,l}X_{l,j} \vee \left(\bigvee_{l < k \leq n} C_{i,k}X_{k,j}\right)\\
&=& \gamma_{i,l}^{-1}X_{i,j} \vee \left(\bigvee_{i < k <l} E_{i,k}\gamma_{i,l}^{-1}X_{k,j}\right) \vee E_{i,l}\gamma_{i,l}^{-1}\alpha_{i,l}X_{l,j} \vee \left(\bigvee_{l < k \leq n} E_{i,k}(\gamma_{i,l} \cdots \gamma_{i,k})X_{k,j}\right)\\
&=& \gamma_{i,l}^{-1}\left[X_{i,j} \vee \left(\bigvee_{i < k <l} E_{i,k}X_{k,j}\right) \vee E_{i,l}X_{l,j}\alpha_{i,l}\right] \vee \left(\bigvee_{l < k \leq n} E_{i,k}(\gamma_{i,l}\gamma_{i,l+1} \cdots \gamma_{i,k})X_{k,j}\right).
\end{eqnarray*}
Since $EX=X$, we have $\bigvee_{i <k <l} E_{i,k}X_{k,j} \leq X_{i,j}$, and hence the expression in the square brackets rearranges and simplifies to give:
\begin{eqnarray*}
(CX)_{i,j} &=& \gamma_{i,l}^{-1}\left[E_{i,l}X_{l,j}\alpha_{i,l} \vee X_{i,j}\right] \vee \left(\bigvee_{l < k \leq n} E_{i,k}(\gamma_{i,l}\gamma_{i,l+1} \cdots \gamma_{i,k})X_{k,j}\right).
\end{eqnarray*}

Consider the effect of setting $X=A$ and $X=B$ in the above. We shall show that, in each case, the right-most non-zero term of this summation dominates. There are therefore three cases to consider:

First, if $i \leq j <l$ we have $A_{k,j}=B_{k,j}=0$ for all  $k \geq l$. Thus
$$(CA)_{i,j} = \gamma_{i,l}^{-1}A_{i,j} \mbox{ and }(CB)_{i,j} = \gamma_{i,l}^{-1}B_{i,j}.$$

Second, if $j=l$ we have $A_{k,l}=0$ and $B_{k,l}=0$ for all  $k > l$. Thus
\begin{eqnarray*}
(CA)_{i,l} = \gamma_{i,l}^{-1}\left[E_{i,l}A_{l,l}\alpha_{i,l}\vee A_{i,l}\right] = \gamma_{i,l}^{-1}A_{i,l},\\
(CB)_{i,l} = \gamma_{i,l}^{-1}\left[E_{i,l}B_{l,l}\alpha_{i,l}\vee B_{i,l}\right] = \gamma_{i,l}^{-1}B_{i,l},
\end{eqnarray*}
where right-hand equalites follow from the fact that $A_{l,l}E_{i,l}\alpha_{i,l}=A_{i,l}$ and $B_{l,l}E_{i,l}\alpha_{i,l}\leq B_{l,l}E_{i,l}\gamma_{i,l} =B_{i,l}$.

Third, if $l<j \leq n$ we have $A_{k,j}=0=B_{k,j}$ for all  $k > j$. Thus for $X=A$ or $X=B$ we have
\begin{eqnarray*}
(CX)_{i,j} &=& \gamma_{i,l}^{-1}\left[E_{i,l}X_{l,j}\alpha_{i,l}\vee X_{i,j}\right] \vee \left(\bigvee_{l < k \leq j} E_{i,k}(\gamma_{i,l}\gamma_{i,l+1} \cdots \gamma_{i,k})X_{k,j}\right)
\end{eqnarray*}
We show that the term $E_{i,j}(\gamma_{i,l}\gamma_{i,l+1} \cdots \gamma_{i,j})X_{j,j}$ dominates all others in the supremum in round brackets. First note that by assumption $\alpha_{i,j}=\gamma_{i,j}$,; let us call this common value $\delta_{i,j}$. Then, by commutativity, this final term can be rewritten as $(X_{j,j}E_{i,j}\delta_{i,j})(\gamma_{i,l}\gamma_{i,l+1} \cdots \gamma_{i,j-1})$. Since $X=A$ or $X=B$, by (17) we see that the latter is equal to $X_{i,j}(\gamma_{i,l}\gamma_{i,l+1} \cdots \gamma_{i,j-1})$. Thus we seek to prove that each term in the supremum is dominated by  $X_{i,j}\gamma_{i,l}\gamma_{i,l+1} \cdots \gamma_{i,j-1}$.

To see this, we recall that:
\begin{itemize}
\item[(i)] the $\gamma$'s were chosen so that $\gamma_{i,q} \geq 1 \geq \gamma_{i,q}^{-1}$ for all $q$;
\item[(ii)] since $EX=X$ we have $E_{i,k}X_{k,j} \leq X_{i,j}$ for all $k$;
\item[(iii)] for the fixed values $i$ and $l$ we have assumed that $\alpha_{i,l}< \gamma_{i,l}$.
\end{itemize}
From this one sees that for all $l<k<j$ we have
\begin{eqnarray*}
X_{i,j}\gamma_{i,l}^{-1} &&\leq \qquad X_{i,j}\qquad \leq\qquad X_{i,j}\gamma_{i,l}\gamma_{i,l+1} \cdots \gamma_{i,j-1},\\
\gamma_{i,l}^{-1}E_{i,l}X_{l,j}\alpha_{i,l}&&  \leq \qquad X_{i,j}\gamma_{i,l} \quad \leq \qquad X_{i,j}\gamma_{i,l}\gamma_{i,l+1} \cdots \gamma_{i,j-1},\mbox{and} \\
E_{i,k}\gamma_{i,l}\gamma_{i,l+1} \cdots \gamma_{i,k}X_{k,j} &&\leq \qquad X_{i,j}\gamma_{i,l}\gamma_{i,l+1} \cdots \gamma_{i,j-1}.
\end{eqnarray*}
Now, by our assumption $CA\, \widetilde{\mathcal{R}}\, CB$,  by Corollary \ref{FountainS*}  we must have
$(\lid{CA})_{i,\ell}=(\lid{CB})_{i,\ell}$. For $X=A, B$ we therefore have
\begin{eqnarray*}
(CX)_{i,k}((CX)_{\ell,k})^{-1} &=&\left\{
\begin{array}{lll}
X_{i,\ell}(\gamma _{i,\ell})^{-1}(X_{\ell,\ell})^{-1} &  & \text{if }k=\ell \\
X_{i,k}\gamma _{i,\ell}\hdots \gamma _{i,\ell-1}(X_{\ell,k})^{-1} &  & \text{if }
\ell<k\leq n
\end{array}
\right.  \\
&=&\left\{
\begin{array}{lll}
X_{i,\ell}(X_{\ell,\ell})^{-1}(\gamma _{i,\ell})^{-1} &  & \text{if }k=\ell \\
X_{i,k}(X_{\ell,k})^{-1}\gamma _{i,\ell}\hdots \gamma _{i,\ell-1} &  & \text{if }
\ell<k\leq n
\end{array}
\right.
\end{eqnarray*}
Since $E_{i,\ell}\leq A_{i,k}(A_{\ell,k})^{-1}$ for all $k$, where $\ell\leq k\leq n$ and $\alpha _{i,\ell}(\gamma _{i,\ell})^{-1}< \pmb{1}$, we have $A_{i,l}(A_{l,l})^{-1}(\gamma_{i,l})^{-1}= E_{i,\ell}\alpha _{i,\ell}(\gamma_{i,\ell})^{-1}<A_{i,k}(A_{\ell,k})^{-1}\gamma _{i,\ell}\hdots \gamma _{i,\ell-1}$. Thus
$$(\lid{CA})_{i,\ell}=\bigwedge \{(CA)_{i,k}((CA)_{\ell,k})^{-1}:\ell\leq k\leq n\} =E_{i,\ell}\alpha _{i,\ell}(\gamma _{i,\ell})^{-1}.$$
On the other hand, since $B_{i,l}(B_{l,l})^{-1}(\gamma_{i,l})^{-1}=E_{i,\ell}\leq B_{i,k}(B_{\ell,k})^{-1}$, we find
$$(\lid{CB})_{i,\ell}=\bigwedge \{(CB)_{i,k}((CB)_{\ell,k})^{-1}:\ell\leq k\leq n\} =E_{i,\ell}.$$
But now we must have $E_{i,\ell}\alpha _{il}(\gamma _{il})^{-1}=E_{i,\ell}$, which is true if and only if $\alpha
_{il}=\gamma _{il}$, and hence providing the desired contradiction.

Dually, if we assume $\alpha _{il}>\gamma _{il}\geq \pmb{1}$ and reverse
their roles in the formation of matrix $C$, we also arrive at a contradiction. So
our supposition is wrong and $\alpha _{il}=\gamma _{il}$ for all $1\leq
i<\ell\leq n$. Thus we conclude that
\begin{equation*}
A_{i,j}=A_{j,j}(E_{i,j}\alpha _{i,j})=A_{j,j}(E_{i,j}\gamma
_{i,j})=A_{j,j}(B_{i,j}(B_{j,j})^{-1})=\mu _{j}B_{i,j}
\end{equation*}
for every $i$ and fixed $j$, where $1\leq i\leq j\leq n$, and hence $A \R B$.
\end{proof}

\section{Structure of $\Rt$- and $\Lt$-classes in $UT_n(\lin)$}\label{sec:tildeclasses}

The aim of this section is to make a careful analysis of the relations $\Rt,\Lt$ and $\Ht$  on the semigroups $UT_n(\lin)$, and establish some facts concerning 
 $\Dt$, building on the work of Section~\ref{sec:upper}. The reader may wonder why we do not also consider an analogue  $\Jt$ of $\J$; such a relation certainly exists (see, for example \cite{jtilde}). However, since  $\J\subseteq \Jt$, the following result tells us that  $UT_n(\lin)$ is a single $\Jt$-class. 
\begin{theorem}\label{thm:j} 
 \cite[Theorem 2.3.11]{TaylorThesis}The semigroup $UT_n(\lin)$ is simple for all $n\in\mathbb{N}$.
\end{theorem}

We note that the proof of the above theorem in \cite{TaylorThesis} is written for the special case 
where  $\lin$ is $\mathbb{R}$ under $+$, so that the corresponding semiring $\linz$ is the tropical semiring $\mathbb{T}$, but carries over without significant adjustment to  our more general context.

\subsection{Deficiency, tightness and looseness} The notion of {\em deficiency } was introduced in \cite{TaylorThesis}, again  in the case where  $\linz$ is the tropical semiring. For $n\in\mathbb{N}$ we denote by $\Phi[n]$ the directed graph with vertices the elements of $[n]$ and an edge $(i,j)$, denoted $i\rightarrow j$, if and only if $i \leq j$. We refer to a sequence of edges $i_1 \rightarrow i_2 \rightarrow \cdots \rightarrow i_k$ in $\Phi[n]$ as a {\em path of length $k$ in $\Phi[n]$}. If $i_1<i_2 < \cdots <i_k$ we shall say that the path is {\em simple}.

\begin{definition}\label{defn:deficiency}
Let $2 \leq k \leq n$, $A\in \upper{n}{\lin}$ and let $\gamma $ denote a path $i_{1}\rightarrow
i_{2}\rightarrow \cdots \ \rightarrow i_{k}$ in $\Phi [n]$ of
length $k$. The \emph{deficiency of $\gamma $ in $A$}, denoted $
\Def_{A}(\gamma )$, is defined to be
\[
\Def_{A}(\gamma )=A_{i_{1},i_{k}}\left( \prod\limits_{2\leq t\leq
k}A_{i_{t-1},i_{t}}\right)^{-1}.
\]
\end{definition}

In the case where $M,N\in U_n(\lin)$ and $\linz=\mathbb{T}$,  our next result forms part of   \cite[Theorem 2.3.6]{TaylorThesis}. We will complete the generalisation of the latter in Theorem~\ref{thm:uniDdeficiency}.

\begin{theorem} \label{thm:deficienylength2} Let $M$, $N\in \upper{n}{\lin}$. Then the following are equivalent:

\begin{enumerate}[\rm(i)]
\item for all paths $\gamma $ in $\Phi [n]$, we  have $\Def_{M}(\gamma )= \Def_{N}(\gamma )$;

\item for all paths $\gamma $ of length $2$ in $\Phi [n]$,  we  have $\Def_{M}(\gamma)=\Def_{N}(\gamma )$;

\item for all paths $\gamma $  in $\Phi [n]$ of the form $1\rightarrow i\rightarrow j$, we have
$\Def_{M}(\gamma )=\Def_{N}(\gamma )$.
\end{enumerate}
\end{theorem}

\begin{proof}
Clearly $(i)\Rightarrow(ii)\Rightarrow (iii)$. It remains to show that $(iii)\Rightarrow(i)$. Let $\gamma$ denote the path $i_1 \rightarrow i_2 \rightarrow \cdots \rightarrow i_k$ in $\Phi [n]$, and let $M \in \upper{n}{\lin}$. Then

\begin{eqnarray*}
\Def_{M}(\gamma )&=& \left[M_{i_1,i_k}\right]\left[(M_{i_1,i_2})^{-1}\right] \cdots \left[(M_{i_{k-1},i_k})^{-1}\right] \\
&=& \left[M_{i_1,i_k}M_{1,i_k}(M_{1,i_k})^{-1}\right]\left[M_{1,i_1}(M_{1,i_1}M_{i_1,i_2})^{-1}\right] \cdots \left[M_{1,i_{k-1}}(M_{1,i_{k-1}}M_{i_{k-1},i_k})^{-1}\right]\\
&=& \left(M_{i_1,i_k}M_{1,i_1}(M_{i_1,i_k})^{-1}\right)\left(M_{1,i_2}(M_{1,i_1}M_{i_1,i_2})^{-1}\right) \cdots \left(M_{1,i_{k}}(M_{1,i_{k-1}}M_{i_{k-1},i_k})^{-1}\right)\\
&=&\big(\Def_{M}(1\rightarrow i_1 \rightarrow i_k)\big)^{-1} \Def_{M}(1\rightarrow i_1 \rightarrow i_2)\hdots \Def_{M}(1\rightarrow i_{k-1} \rightarrow i_k),
\end{eqnarray*}
from which it follows that if (iii) holds, then so will (i).
\end{proof}

We now recall from \cite[Theorem 2.3.6]{TaylorThesis} that deficiency characterises $\D$-class of unitriangular matrices in $\upper{n}{\trop^*}$. This is true in general for $\upper{n}{\lin}$ and follows from Theorem \ref{thm:deficienylength2} and following characterisation of the $\D$-related unitriangular matrices in $\upper{n}{\lin}$:

\begin{theorem}\label{thm:uniDdeficiency}(c.f. \cite[Theorem 2.3.6]{TaylorThesis})
Let $A$, $B\in \uni{n}{\lin}$. Then $A$ $\D$ $B$ in $\upper{n}{\lin}$ exactly if $\Def_{A}(\gamma )=\Def_{B}(\gamma )$
for all paths $\gamma $ of length $2$ in $\Phi [n]$.
\end{theorem}

\begin{proof} Suppose that $A, B \in \uni{n}{\lin}$ and $A$ $\D$ $B$ in $\upper{n}{\lin}$. By Corollary \ref{thm:upperD}, there exists $C\in \upper{n}{\lin}$ such that $\rnorm{C}=\rnorm{A}$ and $\lnorm{C}=\lnorm{B}$. Since $A, B \in \uni{n}{\lin}$ we have $A=\rnorm{A}=\rnorm{C}$ and $B=\lnorm{B}=\lnorm{C}$ giving $A=C(\diag{C})^{-1}=\diag{C}\lnorm{C}(\diag{C})^{-1}=\diag{C}B(\diag{C})^{-1}.$ Thus for all $i \leq k \leq j$ one has
\[
\begin{array}{rcl}
\Def_{A}(i \rightarrow k \rightarrow j)&=&A_{i,j}(A_{i,k})^{-1}(A_{k,j})^{-1}\\
&=& C_{i,i}B_{i,j}(C_{j,j})^{-1}(C_{i,i})^{-1}(B_{i,k})^{-1}C_{k,k}(C_{k,k})^{-1}(B_{k,j})^{-1}C_{j,j}\\
&=& B_{i,j}(B_{i,k}B_{k,j})^{-1} =\Def_{B}(\gamma )
\end{array}\]
Conversely, suppose that $\Def_{A}(\gamma )=\Def_{B}(\gamma )$ for all paths $\gamma $ of length $2$ in $\Phi [n]$. Let $G \in D_n(\lin)$  be the diagonal matrix with $G_{i,i}=A_{1,i}(B_{1,i})^{-1}$ for all $i \in [n]$. Then for all $i,j \in [n]$ where $i\leq j$, we have
\begin{eqnarray*}
(GAG^{-1})_{i,j}&=&G_{i,i}A_{i,j}(G_{j,j})^{-1}=(A_{1,i}B_{1,i}^{-1})A_{i,j}A_{1,j}^{-1}B_{1,j}= (A_{1,i}A_{i,j}A_{1,j}^{-1})(B_{1,j}B_{1,i}^{-1})\\
&=&(\Def_{A}(1 \rightarrow i  \rightarrow j))^{-1}(B_{1,j}B_{1,i}^{-1})=(\Def_{B}(1 \rightarrow i  \rightarrow j))^{-1}(B_{1,j}B_{1,i}^{-1})\\
&=&(B_{1,j}^{-1}B_{1,i}B_{i,j})(B_{1,j}B_{1,i}^{-1})=B_{i,j}.
\end{eqnarray*}
and hence $GAG^{-1}=B$. Now by Theorem \ref{thm:upper*}
we have $A\, \L \, GA\, \R \, GAG^{-1}=B$ in $\upper{n}{\lin}$. Hence $A\, \D \, B$.
\end{proof}

In order to examine  the relations $\Rt,\Lt,\Ht$ and $\Dt$ in $\upper{n}{\lin}$ we  build on the notion of deficiency for an idempotent
$E\in\upper{n}{\lin}$. Recall that, for such an element, and any path $i\rightarrow j\rightarrow k$, we have that $\1\leq \Def_E(i\rightarrow j\rightarrow k)$; indeed, this condition is necessary and sufficient for $E=E^2$ in $\upper{n}{\lin}$.

\begin{definition}\label{defn:tight} Let $E\in \upper{n}{\lin}$ (equivalently,
$ U_n(\lin)$) be idempotent.
For any path $\gamma$ in $\Phi [n]$, we say that, $E$ is \emph{tight in $\gamma$} if $\Def_E(\gamma)=\1$ and {\em loose in $\gamma$} if $\Def_E(\gamma)>\1$.
\end{definition}

It is clear that all paths of the form $i \rightarrow i \rightarrow j$ or $i \rightarrow j \rightarrow j$ are tight in every idempotent. In what follows we consider how the different configurations of tightness and looseness of the remaining (simple) paths of length $2$ impact upon the tilde classes of an idempotent. To this end, we introduce some notation to allow us to express our results compactly. For $\alpha \in \lin$ and $1 \leq i \leq j \leq n$, let  $[\alpha]_{i,j}$ denote the element of $\upper{n}{\linz}$ with $(i,j)$th entry equal to $\alpha$ and all other entries equal to $\1$. For $A,B \in \upper{n}{\linz}$ we write $A \circ B$ to denote the Hadamard product of $A$ and $B$; this is the matrix whose $(i,j)$th entry is equal to $A_{i,j}B_{i,j}$ for all $i,j \in [n]$. Consider the set of upper triangular matrices defined by
$$\mathcal{O}_n([\0,\1]) = \{A \in \upper{n}{\lin}: \forall i \in[n],  A_{i,i} \leq A_{i,i+1} \leq \cdots \leq A_{i,n}\leq \1 \geq  A_{1,i} \geq A_{2,i} \geq \cdots  \geq A_{i,i}\}.$$
Notice that if $n=1$ then  \[\mathcal{O}_n([\0,\1])=\{ \lambda\in\lin: \lambda\leq \pmb 1\}:=\lin_{\leq \pmb 1}.\]
Clearly, $\lin_{\leq \pmb 1}$ is a commutative, cancellative submonoid of $\lin$, but (except in the trivial case) is not a subgroup.

\begin{lemma}\label{lem:ot} The set $\mathcal{O}_n([\0,\1])$ forms  a subsemigroup of $\upper{n}{\lin}$. 
\end{lemma} 
\begin{proof} If $A,B \in \mathcal{O}_n([\0,\1])$ then 
for $1\leq i\leq j\leq n$ we have $(AB)_{i,j}=\bigvee_{i \leq k \leq j}  A_{i,k}B_{k,j}\leq \1$. Moreover, since for each fixed $i,k$ with $i<n$  we have that  $A_{i,k} \geq A_{i+1,k}$, it follows that $(AB)_{i,j} \geq (AB)_{i+1,j}$. Similarly, if $j<n$, then $(AB)_{i,j+1} \geq (AB)_{i,j}$. Thus $AB \in\mathcal{O}_n([\0,\1]))$.
\end{proof}

\begin{proposition}\label{prop:ntightall}
Let $n>2$ and let $E$ be an idempotent of $\upper{n}{\lin}$. If $E$ is tight in all paths in  $ \Phi [n]$ then $\widetilde{H}_{E}$ is a subsemigroup isomorphic to $\lin\times \mathcal{O}_{n-2}([\0,\1])$.
\end{proposition}

\begin{proof}
For $G \in \mathcal{O}_{n-2}([\0,\1])$ let $\overline{G}$ denote the upper triangular matrix constructed by extending $G$ as follows
$$\overline{G}= \left( \begin{array}{c c c} 
\1 &\cdots& \1\\
 & G & \vdots\\
 &  & \1\end{array}\right) \in \mathcal{O}_{n}([\0,\1]).$$
Consider $C=\overline{G} \circ E$.
Then, since $E$ is tight in all paths, for any $1\leq i\leq j\leq n$ we have
\[
(\lid{C})_{i,j}=\bigwedge_{i\leq k\leq j}C_{i,k}C_{j,k}^{-1}=
\bigwedge_{j \leq k \leq n}E_{i,k}\overline{G}_{i,k}(E_{j,k}\overline{G}_{j,k})^{-1}=\bigwedge_{j \leq k \leq n}E_{i,j}\overline{G}_{i,k}\overline{G}_{j,k}^{-1}=E_{i,j},\]
where the fnal equality follows from the fact that $\overline{G} \in \mathcal{O}_{n}([\0,\1])$ with all entries in the final column equal to $\1$. This shows that $\lid{C}=E$; a dual argument yields $\rid{C}=E$ and so $(\lambda C)^+=(\lambda C)^*=E$, for any $\lambda\in \lin$.

Conversely, we must show that any matrix in $\widetilde{H}_{E}$ is a scalar multiple of a matrix  of the form $\overline{G} \circ E$, where $G \in \mathcal{O}_{n-2}([\0,\1])$. Let $A \in \widetilde{H}_{E}$. Then $(\rnorm {A})^{(+)}=\lid {A}= E$ and $(\lnorm {A})^{(*)}=\rid A = E$. By using equations \eqref{Ralpha*}, \eqref{Lbeta*} we have $\rnorm{A}_{i,j}= E_{i,j} \alpha_{i,j}$ and $\lnorm{A}_{i,j}= E_{i,j} \beta_{i,j}$ where $\alpha_{i,j},\beta_{i,j}\geq \mathbf{1}$ for $i<j$ and $\alpha_{i,i}=\alpha_{i,n}=\beta_{i,i}=\beta_{1,i}=\mathbf{1}$.
 Calculating $\lid{(\rnorm{A})}$  using Formula \eqref{formula:Aplus} and the fact that $E$ is tight for all paths (giving $E_{i,j}E_{j,k}=E_{i,k}$ for all
$i\leq j \leq k$), we obtain:
$$E_{i,j}=\bigwedge_{j \leq k \leq n}E_{i,k}\alpha_{i,k}(E_{j,k}\alpha_{j,k})^{-1}=\bigwedge_{j \leq k \leq n} E_{i,j}\alpha_{i,k}\alpha_{j,k}^{-1},$$
from which it follows that $ \alpha_{j,k} \leq \alpha_{i,k}$ for all $i \leq j \leq k $. Dually, $ \beta_{k,i} \leq \beta_{k,j}$ for all $1 \leq k \leq i$.

Since $\rnorm A \diag A = A= \diag A \lnorm A$, we then have  
\begin{equation}\label{eqna} E_{i,j}\alpha_{i,j}A_{j,j}=A_{i,i}E_{i,j}\beta_{i,j}\end{equation}
 so that $A_{j,j}\alpha_{i,j}=A_{i,i} \beta_{i,j}$ 
 for all $i\leq j$. Since  $\beta_{1,j}=\1$, we see that $A_{j,j}\alpha_{1,j}=A_{1,1}$ and so  \begin{equation}\label{eqnb}A_{j,j}=A_{1,1}\alpha_{1,j}^{-1}.\end{equation} Now, since  
$\alpha_{1,n}=\1$, we then obtain  $A_{1,1}=A_{n,n}$. Using \eqref{eqna} and \eqref{eqnb}
we calculate
\begin{equation}\label{eqnc}\beta_{i,j}=A_{j,j}\alpha_{i,j}A_{i,i}^{-1}=A_{1,1}\alpha_{1,j}^{-1}\alpha_{i,j}(A_{1,1}\alpha_{1,i}^{-1})^{-1}=\alpha_{1,i}\alpha_{i,j}\alpha_{1,j}^{-1}.\end{equation}
For $1\leq i\leq j\leq n$ we now let $\gamma_{i,j}=\alpha_{i,j}\alpha_{1,j}^{-1}$ so that
$\gamma_{i,j}\leq \1$. Using the facts that $ \alpha_{j,k} \leq \alpha_{i,k}$ and  $ \beta_{k,i} \leq \beta_{k,j}$ for all allowable subscripts, together with
\eqref{eqnc}, we see that
$(\gamma_{i,j}) = \overline{G(A)}$ for some $G(A)\in  \mathcal{O}_{n-2}([\0,\1])$. Moreover, $A_{i,j}=E_{i,j}\gamma_{i,j}A_{1,1}$, and so $A=A_{1,1}C(A)$, where $C(A)= (\overline{G(A)} \circ E)$, as required. 

Given that the decomposition of $A$ as $A=\lambda D$ where $D_{1,1}=\1$ is perforce unique, the above establishes that $\theta:\widetilde{H}_{E} 
\rightarrow \lin\times \mathcal{O}_{n-2}([\0,\1])$ given by
$\theta(A)=(A_{11}, G(A))$, is a bijection. It remains to show that $\theta$ is an isomorphism. 
 
To this end, suppose now that $A,B\in \widetilde{H}_{E}$, so that $A=A_{1,1} C(A)$ and 
$B=B_{1,1} C(B)$, where $C(A)=\overline{G}\circ E$ and $B(A)=\overline{H}\circ E$, respectively. Clearly, $AB=A_{1,1}B_{1,1} C(A)C(B)$. Let $i\leq j$; using the fact that $E$ is tight in all paths, we calculate
\[(C(A)C(B))_{i,j}=\bigvee_{i\leq k\leq j}C(A)_{i,k}C(B)_{k,j}=\bigvee_{i\leq k\leq j}
E_{i,k}\overline{G}_{i,k}E_{k,j}\overline{H}_{k,j}=E_{i,j}\bigvee_{i\leq k\leq j}\overline{G}_{i,k}\overline{H}_{k,j}.\]
Let $\kappa_{i,j}=\bigvee_{i\leq k\leq j}\overline{G}_{i,k}\overline{H}_{k,j}$ and notice that
 $\kappa_{1,i}=
\kappa_{i,n} = \1$ for all $1\leq i\leq n$. It follows (by restricting the range of $i,j$) that
\[\theta(AB)=(A_{1,1}B_{1,1}, G(A)G(B))=\theta(A)\theta(B),\]
so that $\theta$ is an isomorphism as required. 
\end{proof}

At the other extreme we have:

\begin{proposition}\label{prop:looseall}
Let $E,A\in \upper{n}{\lin}$. Suppose that $E$ is an idempotent that is 
loose in all simple paths $i\rightarrow k\rightarrow j$ of  length $2\in \Phi [n]$. Then $A \Rt E$ ($A
\Lt E$) if and only if $A \R  E$ ($A\L E$). Thus $\widetilde{H}_{E}$ is a group isomorphic to $\lin$.

\end{proposition}

\begin{proof}
Suppose that $A \Rt E$. Then $E=\lid{A}$ and so for all $i \in \{1, \ldots, n\}$, $E_{i,n}=A_{i,n}(A_{n,n})^{-1}$ by formula \eqref{Aplus}. Suppose for finite induction that
 $1< \ell\leq n$ and for all $t$ with $\ell\leq t\leq n$ and for all $i\in\{ 1,\ldots, n\}$ with $i\leq t$ we have
 $E_{i,t}=A_{i,t}(A_{t,t})^{-1}$.

Let $i\in\{ 1,\hdots, n\}$ with $i\leq \ell-1$. Then by formula \eqref{Aplus} and our inductive hypothesis we see that
\begin{eqnarray*}
E_{i,\ell -1}&=& \bigwedge\limits_{\ell-1\leq k\leq n}A_{i,k}(A_{\ell-1,k})^{-1}
=A_{i,\ell-1}(A_{\ell-1,\ell-1})^{-1}\wedge \bigwedge\limits_{\ell\leq k\leq n}A_{i,k}(A_{\ell-1,k})^{-1}\\
&=&A_{i,\ell-1}(A_{\ell-1,\ell-1})^{-1}\wedge \bigwedge\limits_{\ell\leq k\leq n}E_{i,k}(E_{\ell-1,k})^{-1}.
\end{eqnarray*}
By our assumption on tightness, we have
$E_{i,\ell-1}<E_{i,k}(E_{\ell-1,k})^{-1}$ for all $i<\ell-1< k$. Since $\lin$ is linearly ordered we deduce that
$E_{i,\ell -1}=A_{i,\ell-1}(A_{\ell-1,\ell-1})^{-1}$ for all $i<\ell-1$, whilst for $i=\ell-1$ we have $\1=E_{i,\ell-1}=A_{i, \ell-1}A_{\ell-1, \ell-1}^{-1}$. Finite induction now yields that $E=A(\diag{A})^{-1}$ and hence, by Theorem \ref{thm:upper*}, 
$A\,\R\, E$. The converse is clear. 

It follows from the above that $\widetilde{H}_{E} = H_E$, which by Theorem \ref{thm:upperH} is isomorphic to $\lin$.
\end{proof}

The previous two results hint towards the properties of deficiency, tightness and looseness being important for determination of the $\Rt$- $\Lt$- and $\Ht$-classes in $UT_n(\lin)$; our subsequent investigations confirm that this is indeed the case. We begin by examining the notion of tightness in more detail.

\begin{lemma}\label{tightnesslemma}
Let $E\in \upper{n}{\lin}$ be such that $E^{2}=E$.

\begin{enumerate}[\rm(i)]
\item If $E$ is tight in the path $1\rightarrow2\rightarrow \hdots \rightarrow n$\ then $E$ is tight in all paths $\gamma \in \Phi [n]$.
\item If $E$ is tight in all paths $i\rightarrow u\rightarrow j$ of length 2, then $E$ is tight in all paths $\gamma \in \Phi [n]$.
\item If $E$ is tight in the paths $i\rightarrow u\rightarrow j$ and $u\rightarrow
v\rightarrow j$ then $E$ is tight in the paths $i\rightarrow
v\rightarrow j$ and $i\rightarrow u\rightarrow v$.

\item  For any $i<u<v<j$, if $E$ is tight in three out of the four paths
$i\rightarrow u\rightarrow v, i\rightarrow u\rightarrow j,
i\rightarrow v\rightarrow j$ and $u\rightarrow v\rightarrow j$, then
$E$ is tight in the fourth.
\end{enumerate}
\end{lemma}

\begin{proof}
Let $E\in \upper{n}{\lin}$ such that $E^{2}=E$, so that  for all $i,k,j\in \{1,\hdots
,n\}$ we have $E_{i,k}E_{k,j}\leq E_{i,j}$.  As already observed, if $i=k$ or $k=j$, then this inequality is tight.

(i) Suppose that $E$ is tight in the path $1\rightarrow 2\rightarrow \hdots
\rightarrow n$ and $1\leq i<j<k\leq n$. Then
\begin{eqnarray*}
E_{1,n}&=&E_{1,2}\cdots E_{i-1,i}E_{i,i+1}\cdots E_{j-1,j}E_{j,j+1}\cdots E_{k-1,k}E_{k,k+1}\cdots E_{n-1,n}\\
&\leq& E_{1,i}E_{i,j}E_{j,k}E_{k,n} \leq E_{1,i}E_{i,k}E_{k,n} \leq E_{1,n}.\end{eqnarray*}
Notice that if $E$ is loose in the path $i \rightarrow j \rightarrow k$, then the second inequality would be strict, giving a contradiction.

(ii) If $E$ is tight in all paths of length 2 in $\Phi [n]$, then
$$E_{1,2}E_{2,3}E_{3,4}\cdots E_{n-1,n}=E_{1,3}E_{3,4}\cdots E_{n-1,n}=\cdots= E_{1,n},$$
 and the result follows from part (i).

(iii) If $E_{i,u}E_{u,j}=E_{i,j}$ and $E_{u,v}E_{v,j}=E_{u,j}$ then
\begin{equation*}
E_{i,j}=E_{i,u}E_{u,v}E_{v,j}\leq E_{i,v}E_{v,j}\leq E_{i,j}.
\end{equation*}
This gives  $E_{i,j}=E_{i,v}E_{v,j}$ and hence
$$E_{i,u}E_{u,v}=(E_{i,j}(E_{u,j})^{-1})(E_{u,j}(E_{v,j})^{-1})=
E_{i,j}(E_{v,j})^{-1}=E_{i,v}.$$

(iv) This is a matter of checking the remaining possibilities, making use of part (iii).
\end{proof}

\subsection{Detailed calculations for $n<5$} \label{subsec:ut3}

We have already shown in Corollary \ref{regut} that $\upper{n}{\lin}$ is regular for $n=1$ and $n=2$ and so, in these cases, $\Rt=\Rs=\R$ and $\Lt=\Ls=\L$. In this subsection, we perform a full analysis for $n=3$ and $n=4$.
In view of Corollary \ref{one-onecorr} it is sufficient to characterise the equivalence classes in $\upper{n}{\lin}$ by computing the $\Rt$- and $\Lt$-classes of idempotents in $\uni{n}{\lin}$. We begin with the case where  $n=3$.

\begin{proposition} \label{3x3LRtildcls} Let $E$ be an idempotent of $\upper{3}{\lin}$.
\begin{enumerate}[\rm(i)]
\item  If $E$ is loose in the path $1\rightarrow 2 \rightarrow 3$ then
$$\widetilde{R}_{E}=E D_n(\lin) =R_{E}, \;\; \widetilde{L}_{E}=D_n(\lin)E  =L_{E}, \;\;  \mbox{ and } \;\; \widetilde{H}_{E}= H_E\cong \lin.$$
\item Otherwise
\begin{eqnarray*}
\widetilde{R}_{E}&=&\left\{ ([\alpha]_{1,2} \circ E)G :\alpha^{-1} \in \lin_{\leq 1},G \in D_n(\lin)\right\}\\
\widetilde{L}_{E}&=&\left\{ G([\beta]_{2,3} \circ E)G :\beta^{-1} \in \lin_{\leq 1},G \in D_n(\lin)\right\}\\
\widetilde{H}_{E}&\cong&\lin \times \lin_{\leq 1}.
\end{eqnarray*}
\end{enumerate}
\end{proposition}

\begin{proof} Part (i) follows immediately from Proposition \ref{prop:looseall}, Theorem \ref{thm:upper*} (and its left-right dual statement), and Theorem \ref{thm:upperH}.

Now let $E\in \upper{3}{\lin}$ be an idempotent which is tight in the path $1 \rightarrow 2 \rightarrow 3$. Thus $E_{1,3}=E_{1,2}E_{2,3}$. If $A \in \upper{3}{\lin}$ with $A \Rt E$, then (by Theorem \ref{thm:uppertilde}) we have $E=\lid{(\rnorm{A})}$. By formula \ref{Ralpha*} we have $\rnorm{A} = [\alpha]_{1,2}\circ E$ for some $\alpha \geq 1$, and hence $A=([\alpha]_{1,2}\circ E)D_A$.

Conversely, suppose that  $A=([\alpha]_{1,2}\circ E)G$ for some $\alpha \geq \1$ and $G \in D_n(\lin)$. Since $[\alpha]_{1,2}\circ E$ is unitriangular, notice that we must have $\rnorm{A} = [\alpha]_{1,2}\circ E$ and hence
\begin{equation*}
\lid{(\rnorm{A})}=\left(
\begin{array}{ccc}
\1 & E_{12}\alpha \wedge E_{1,3}E_{2,3} & E_{1,3} \\
\0 & \1 & E_{2,3} \\
\0 & \0 & \1
\end{array}
\right) = \left(
\begin{array}{ccc}
\1 & E_{12}\alpha \wedge E_{1,2} & E_{1,3} \\
\0 & \1 & E_{2,3} \\
\0 & \0 & \1
\end{array}
\right) =E,
\end{equation*}
using the fact that $E$ is tight in the path $1 \rightarrow 2 \rightarrow 3$, and $\alpha \geq \1$. This shows that the $\Rt$-class of $E$ is as required. A similar argument holds for the $\Lt$-class.

Suppose now that $A \Ht E$ in $\upper{3}{\lin}$. Then $A=([\alpha]_{1,2} \circ E)D_A = D_A([\beta]_{2,3} \circ E), \mbox{ where } \alpha, \beta \geq \1$, giving
$$\begin{pmatrix}
    A_{1,1} & \alpha E_{1,2}A_{2,2} & E_{1,3}A_{3,3} \\
    \0 & A_{2,2} & E_{2,3}A_{3,3} \\
    \0 & \0 & A_{3,3}
    \end{pmatrix} =
    \begin{pmatrix}
    A_{1,1} & A_{1,1}E_{1,2} &A_{1,1} E_{1,3} \\
    \0 & A_{2,2} &A_{2,2}\beta E_{2,3} \\
    \0 & \0 & A_{3,3}
    \end{pmatrix}.$$
Comparing entries at the position $(1,3)$ yields $A_{1,1}=A_{3,3}$, whilst comparing entries at positions $(1,2)$ and $(2,3)$ gives $A_{2,2}\alpha=A_{1,1}$ and $A_{2,2} \beta = A_{3,3}$. Thus it follows that $\alpha=\beta$ and together with our tightness assumption this gives
\begin{equation}A =A_{1,1}([\alpha^{-1}]_{2,2} \circ E), \mbox{ where } \alpha^{-1}\leq \1, A_{1,1} \in \lin. \label{eqn:n=3tightHt}
\end{equation}
In fact, for any matrix of the form given in \eqref{eqn:n=3tightHt}, it is straightforward to check that $\lid A =E= \rid A$ in $\upper{3}{\lin}$. Hence
$$\widetilde{H}_{E}=\left\{ \lambda([\mu]_{2,2} \circ E) :\lambda ,\mu \in \lin, \mu \leq \1\right\}. $$
Also, note that for $\lambda, \lambda', \mu, \mu' \in \lin$ with $\mu, \mu' \leq \1$ we have $\mu\mu'\leq 1$ and  
\begin{eqnarray*}
\lambda([\mu]_{2,2} \circ E) \cdot \lambda'([\mu']_{2,2} \circ E) &=&\lambda\lambda'
\begin{pmatrix}
\1 & E_{1,2}\vee E_{12}\alpha _{2} & E_{1,2}E_{2,3} \\
\0 & \alpha _{1}\alpha _{2} & E_{2,3}\alpha _{1}\vee E_{2,3} \\
\0 & \0 & \1
\end{pmatrix}=\lambda\lambda'([\mu\mu']_{2,2} \circ E)
\end{eqnarray*}
Thus $\widetilde{H}_{E}$ is closed under multiplication and, moreover, it follows from the form of this product that, $\widetilde{H}_{E} \cong \lin\times \lin_{\leq \pmb 1}$.
\end{proof}

We know that  $\R$ and $\L$ commute in any semigroup so that in a regular semigroup (where $\R=\Rt$ and $\L=\Lt$), certainly $\Rt$ and $\Lt$ commute. We have remarked in Corollary~\ref{regut} that $\upper{3}{\lin}$ is not regular. Nevertheless, we can show that, in this semigroup, $\Rt$ and $\Lt$ commute.

\begin{proposition}
\label{prop:commute}
\begin{enumerate}[\rm(i)]
\item For each  $A\in \upper{3}{\lin}$ we have that
$\lid A\,\D\rid A$.
\item The relations $\Rt$ and $\Lt$ commute on $\upper{3}{\lin}$, that is, we have
 $\Lt\circ \Rt=\Rt \circ \Lt$. Consequently, $\Dt=\Lt\circ \Rt=\Rt \circ \Lt.$
\item Let $E,F$ be idempotents in $\upper{3}{\lin}$. Then $E \D F$ if and only if $E \Dt F$.
\end{enumerate}
\end{proposition}

\begin{proof}
(i) Let $A\in \upper{3}{\lin}$. Direct calculation gives
$$\lid{A} =\begin{pmatrix}
\1&A_{1,2}A_{2,2}^{-1}\wedge A_{1,3}A_{2,3}^{-1}&A_{1,3}A_{3,3}^{-1}\\
\0&\1&A_{2,3}A_{3,3}^{-1}\\
\0&\0&\1\end{pmatrix}, \;\; \rid{A} = \begin{pmatrix}
\1&A_{1,2}A_{1,1}^{-1}&A_{1,3}A_{1,1}^{-1}\\
\0&\1&A_{1,3}A_{1,2}^{-1}\wedge A_{2,3}A_{2,2}^{-1}\\
\0&\0&\1\end{pmatrix}.$$
Noting that
$$A^{(+)}_{1,2}  A^{(+)}_{2,3}  (A^{(+)}_{1,3})^{-1}=A^{(\ast)}_{1,2} A^{(\ast)}_{2,3} (A^{(\ast)}_{1,3})^{-1} = \1\wedge A_{1,2}A_{2,3}A_{2,2}^{-1}A_{1,3}^{-1},$$
it follows from Theorem~\ref{thm:uniDdeficiency} that $\lid A\, \D\,  \rid A$ in $\upper{3}{\lin}$. 

(ii) Let $A,B\in \upper{3}{\lin}$ be such that  $A\,(\Rt\circ \Lt)\, B$. Then there exists
$X\in \upper{3}{\lin}$ with $A\,\Rt\, X\,\Lt\, B$.  It then follows from Theorem \ref{thm:uppertilde} and the above observation that
$$\rid A\,\D\, \lid A=\lid X\,\D\,\rid X=\rid B\,\D\, \lid B,$$
and so there is a $Y\in \upper{3}{\lin}$ with
$$A\,\Lt\, \rid A\,\L\, Y\,\R\, \lid B\,\Rt B$$
giving $A\,  \Lt \, Y\,\Rt\,  B$, as required. The dual argument now completes the proof.

(iii) Supppose that $E$ and $F$ are idempotents with $E$ $\Dt$ $F$. By part (ii) there exists $X\in \upper{3}{\lin}$ such that $E$ $\Rt$ $X$ $\Lt$ $F$. Since each $\Rt$-class and each $\Lt$-class contains a unique idempotent, we must have we must have $E=\lid X$ and $F=\rid X$. But now part (i) gives $E \D F$.
\end{proof}

We will see that the very regular behaviour of the $\sim$-relations in
$UT_n(\lin)$ for $n=1,2,3$ does not persist for larger $n$. In fact, even for $n=4$, we see that the characterisations of $\Rt$-,$\Lt$- and $\Ht$-classes become more complex, and $\Rt$ and $\Lt$ no longer commute. However, we still have the property that  $\widetilde{H}_E$ for an idempotent  $E\in UT_4(\lin)$ is a subsemigroup: its nature depends on the tightness or otherwise of $E$ in the paths of length 2 in $\Phi[4]$. 

In the rest of this subsection we carefully analyse the $\Rt$-,$\Lt$- and $\Ht$-classes of $E\in \upper{4}{\lin}$, depending on the  tightness and looseness patterns, as indicated. We proceed along the lines of that of Propositions~\ref{3x3LRtildcls}, but with inevitably more complex arguments.

We observe that there are four simple paths of length $2$ in $\Phi[4]$ namely, $1 \rightarrow 2 \rightarrow 3,\ 1 \rightarrow 2 \rightarrow 4,\ 1 \rightarrow 3 \rightarrow 4$ and $2 \rightarrow 3 \rightarrow 4$. We explicitly consider the $\Rt$, $\Lt$ and $\Ht$-classes of idempotents in $\upper{4}{\lin}$ depending on the tightness and looseness patterns of the four simple paths of lengths 2 in $\Phi[4]$. Apriori there are 16 possible tightness patterns; by Lemma~\ref{tightnesslemma} six of these do not arise, since  any idempotent which is tight both $1\rightarrow 2\rightarrow 3$ and $1\rightarrow 3\rightarrow 4$  (or dually in both $1\rightarrow 2\rightarrow 4$ and $2\rightarrow 3\rightarrow 4$) is tight in all four paths. This leaves ten cases to consider. Moreover, since the involutary anti-automorphism $\Delta$ strongly exchanges the relations $\Lt$ and $\Rt$, it is easy to see that there will be a duality between the following pairs of cases:
\begin{itemize}
\item `$E$ is tight in only $1\rightarrow 2\rightarrow 3$' is dual to `$E$ is tight in only $2\rightarrow 3\rightarrow 4$';
\item `$E$ is tight in only $1\rightarrow 2\rightarrow 4$' is dual to `$E$ is tight in only $1\rightarrow 3\rightarrow 4$';
\item `$E$ is tight in only $1\rightarrow 2\rightarrow 3$ and $1\rightarrow 2\rightarrow 4$' is dual to `$E$ is tight in only $2\rightarrow 3\rightarrow 4$ and $1\rightarrow 3\rightarrow 4$';
\end{itemize}
while remaining four cases are self-dual. This reduces the task to consideration of seven patterns of tightness and looseness of idempotents in $\upper{4}{\lin}$.   The following table gives the $\Rt$-, $\Lt$- and $\Ht$-classes corresponding to each of these 7 cases obtained by direct calculations (see below) and calling upon Proposition~\ref{prop:looseall} for the case where $E$ is loose in all paths.
\begingroup
\renewcommand\arraystretch{1.35}
\begin{longtable}{|l|}
   \hline
\textbf{1. $\mathbf E$ is loose in all four paths} (self-dual) \\ \hline
$\widetilde{R}_{E}=E D_4(\lin) =R_{E}$, $\widetilde{L}_{E}=D_4(\lin)E  =L_{E}$, $\widetilde{H}_{E}=H_{E}$ \\
 \hline \hline
\textbf{2. $\mathbf E$ is tight in only $\mathbf{1\rightarrow 2\rightarrow 3}$} (dual case: tight in only $2\rightarrow 3\rightarrow 4$)\\  \hline  
$\widetilde{R}_{E}=\left\{[\alpha]_{1,2} \circ E: \alpha \in \lin, \1 \leq \alpha\right\}D_4(\lin)$ \\
$\widetilde{L}_{E}=D_4(\lin)\left\{[\Def_E(2 \rightarrow 3 \rightarrow 4)]_{2,3}\circ [\alpha]_{3,4} \circ E : \alpha \in \lin, \1 \leq \alpha\right\}$\\
$\qquad\cup \; D_4(\lin)\left\{ [\alpha]_{2,3} \circ E : \alpha \in \lin, \1 \leq \alpha \leq \Def_E(2 \rightarrow 3 \rightarrow 4) \right\}$\\
$\widetilde{H}_{E}=H_E$\\
 \hline \hline
\textbf{3. $\mathbf E$ is tight in only $\mathbf{1\rightarrow 2\rightarrow 4}$} (dual case: tight in only $1\rightarrow 3\rightarrow 4$)\\ 
\hline
$\widetilde{R}_{E}=\left\{[\alpha]_{1,2} \circ E: \alpha \in \lin, \1 \leq \alpha\right\}D_4(\lin)$ \\
$\widetilde{L}_{E}=D_4(\lin)\left\{[\alpha]_{2,4} \circ E : \alpha \in \lin, \1 \leq \alpha \right\}  $ \\
$\widetilde{H}_{E}=H_E$
\\ \hline\hline
\textbf{4. $\mathbf E$ is tight in $\mathbf{1\rightarrow 3\rightarrow 4}$, $\mathbf{2\rightarrow 3\rightarrow 4}$} (dual case: tight in $1\rightarrow 2\rightarrow 4$, $1\rightarrow 2\rightarrow 3$)\\ 
\hline
$\widetilde{R}_{E}=\left\{[\alpha]_{1,3}\circ [\beta]_{2,3} \circ E : \alpha, \beta \in \lin,  \1 \leq \alpha, \1 \leq \beta \leq \Def_E(1 \rightarrow 2 \rightarrow 3)\alpha\right \}D_4(\lin)$\\
$\qquad \cup \; \left\{ [\alpha]_{1,2}\circ [\beta]_{1,3} \circ [\Def_E(1 \rightarrow 2 \rightarrow 3)\beta]_{2,3} \circ E : \alpha, \beta \in \lin,\1 \leq \alpha, \beta \right\}D_4(\lin)$ \\
$\widetilde{L}_{E}=D_4(\lin)\left\{[\alpha]_{3,4} \circ E: \alpha\in \lin,\1 \leq \alpha\right\}$ \\
$\widetilde{H}_{E}=\left\{\lambda([\mu]_{3,3} \circ E): \lambda, \mu \in \lin, \mu \leq \1 \right\}$\\
 \hline \hline
\textbf{5. $\mathbf E$ is tight in $\mathbf{1\rightarrow 2\rightarrow 3}$, $\mathbf{2\rightarrow 3\rightarrow 4}$} (self dual)\\ 
\hline
$\widetilde{R}_{E}=\left\{[\alpha]_{1,2} \circ E: \alpha \in \lin, \1 \leq \alpha\right\}D_4(\lin)$ \\
$\widetilde{L}_{E}=D_4(\lin)\left\{[\alpha]_{3,4} \circ E : \alpha \in \lin, \1 \leq \alpha \right\}  $ \\
$\widetilde{H}_{E}=H_E$\\
\hline
\hline
\textbf{6. $\mathbf E$ is tight in $\mathbf{1\rightarrow 2\rightarrow 4}$, $\mathbf{1\rightarrow 3\rightarrow 4}$} (self dual)\\ 
\hline
$\widetilde{R}_{E}=\left\{[\alpha]_{1,2} \circ[\beta]_{1,3} \circ E: \alpha, \beta \in \lin, \1 \leq \alpha, \beta\right\}D_4(\lin)$ \\
$\widetilde{L}_{E}=D_4(\lin)\left\{[\alpha]_{2,4} \circ [\beta]_{3,4} \circ E) : \alpha, \beta \in \lin, \1 \leq \alpha, \beta \right\}  $ \\
$\widetilde{H}_{E}=\left\{\lambda([\mu]_{2,2} \circ [\mu]_{2,3} \circ [\mu]_{3,3} \circ E): \lambda, \mu \in \lin, \mu \leq \1 \right\}$\\

\hline \hline
\textbf{7. $\mathbf E$ is tight in all four paths} (self dual)\\ 
\hline
$\widetilde{R}_{E}=\left\{[\alpha]_{1,2} \circ[\beta]_{1,3} \circ [\gamma]_{2,3} \circ E: \alpha, \beta, \gamma \in \lin,\1 \leq \alpha, \1 \leq \gamma \leq \beta\right\}D_4(\lin)$ \\
$\widetilde{L}_{E}=\left\{ [\alpha]_{2,3} \circ [\beta]_{2,4} \circ [\gamma]_{3,4} \circ  E) : \alpha, \beta, \gamma \in \lin,\1 \leq \gamma, \1\leq \alpha \leq \beta \right\} D_4(\lin)$ \\
$\widetilde{H}_{E}=\left\{\lambda([\alpha]_{2,2} \circ [\beta]_{2,3} \circ [\gamma]_{3,3} \circ E: \lambda, \alpha, \beta, \gamma \in \lin, \alpha \vee \gamma\leq \beta  \leq \1 \right\}$\\
\hline
  \caption{Table of $\Rt$-, $\Lt$- and $\Ht$-classes in $\upper{4}{\lin}$}
  \label{table:n=4RtandLt}
\end{longtable}
\endgroup

Notice that from the above table one can easily compute the dual results.  For example, if $E$ is tight only in the simple path $1\rightarrow 2\rightarrow 3$, we find that $\widetilde{R}_{E} =\left\{ [\alpha]_{1,2} \circ E : \1 \leq \alpha \right\} D_4(\lin)$, and hence by duality if $E$ is tight only in the simple path $2\rightarrow 3\rightarrow 4$, then $\widetilde{L}_{E}=  D_4(\lin)\left\{ [\beta]_{3,4} \circ E :\1 \leq \beta\right\}$.

Let $E,A \in \upper{4}{\lin}$ such that $E^{2}=E$. If $A \Rt E$, then by \eqref{Ralpha*} we may write $$\rnorm{A} = [\alpha_{1,2}]_{1,2} \circ [\alpha_{1,3}]_{1,3} \circ [\alpha_{2,3}]_{2,3} \circ E,$$ where  $\alpha_{1,2}, \alpha_{1,3}, \alpha_{2,3} \geq \1$. Since $A \Rt E$ if and only if $\lid{(\rnorm{A})}=E$, a description of the $\Rt$-class of $E$ can be found by computing the solutions of the following system of equations:
\begin{eqnarray}
E_{1,2} &=& \alpha_{1,2}E_{1,2} \wedge \alpha_{1,3}E_{1,3}\alpha_{2,3}^{-1}E_{2,3}^{-1} \wedge E_{1,4}E_{2,4}^{-1},\label{(a)}\\
E_{1,3} &=& \alpha_{1,3}E_{1,3} \wedge  E_{1,4}E_{3,4}^{-1},\label{(b)}\\
E_{2,3} &=& \alpha_{2,3}E_{2,3} \wedge E_{2,4}E_{3,4}^{-1}.\label{(c)}
\end{eqnarray}
Notice that equation \eqref{(a)} holds if and only if $\alpha_{2,3}\leq \Def_E(1 \rightarrow 2 \rightarrow 3) \alpha_{1,3}$ and either:  $E$ is tight in the path $1 \rightarrow 2 \rightarrow 4$; or $\alpha_{1,2}=\1$; or $\alpha_{2,3}= \Def_E(1 \rightarrow 2 \rightarrow 3) \alpha_{1,3}$. Equation \eqref{(b)} holds if and only if either $E$ is tight in the path $1 \rightarrow 3 \rightarrow 4$ or $\alpha_{1,3}=\1$; whilst \eqref{(c)} holds if and only if either $E$ is tight in the path $2 \rightarrow 3 \rightarrow 4$ or $\alpha_{2,3}=\1$.

Similarly, if $A \Lt E$, then by \eqref{Lbeta*} we may write $\lnorm{A} = [\beta_{2,3}]_{2,3} \circ [\beta_{2,4}]_{2,4} \circ [\beta_{3,4}]_{3,4} \circ E$, where  $\beta_{2,3}, \beta_{2,4}, \beta_{3,4} \geq \1$. Since $A \Lt E$ if and only if $\rid{(\lnorm{A})}=E$, the $\Lt$ class of $E$ can be found by computing the solutions to:
\begin{eqnarray}
E_{3,4} &=& \beta_{3,4}E_{3,4} \wedge \beta_{2,4}E_{2,4}\beta_{2,3}^{-1}E_{2,3}^{-1} \wedge E_{1,4}E_{1,3}^{-1},\label{(a)'}\\
E_{2,3} &=& \beta_{2,3}E_{2,3} \wedge  E_{1,3}E_{1,2}^{-1},\label{(b)'}\\
E_{2,4} &=& \beta_{2,4}E_{2,4} \wedge E_{1,4}E_{1,2}^{-1}.\label{(c)'}
\end{eqnarray}
Equation \eqref{(a)'} holds if and only if $\beta_{2,3}\leq \Def_E(2 \rightarrow 3 \rightarrow 4) \beta_{2,4}$ and either: $E$ is tight in the path $1 \rightarrow 3 \rightarrow 4$; or $\beta_{3,4}=\1$; or $\beta_{2,3}= \Def_E(2 \rightarrow 3 \rightarrow 4) \beta_{2,4}$. Similarly, \eqref{(b)'} holds if and only if either $E$ is tight in the path $1 \rightarrow 2 \rightarrow 3$ or $\beta_{2,3}=\1$; whilst \eqref{(c)'} holds if and only if either $E$ is tight in the path $1 \rightarrow 2 \rightarrow 4$ or  $\beta_{2,4}=\1$.

Suppose now that equations \eqref{(a)}-\eqref{(c)'} hold, and additionally that $A$ is equal to:
$$
\begin{pmatrix}
    A_{1,1} & \alpha_{1,2}E_{1,2}A_{2,2}  &  \alpha_{1,3}E_{1,3}A_{3,3}   & E_{1,4}A_{4,4} \\
    \0 & A_{2,2}  &  \alpha_{2,3}E_{2,3}A_{3,3}   & E_{2,4}A_{4,4} \\
    \0 & \0  & A_{3,3}   & E_{3,4}A_{4,4} \\
    \0 & \0  &  \0   & A_{4,4} \\
    \end{pmatrix}= \begin{pmatrix}
    A_{1,1} & A_{1,1}E_{1,2}  &  A_{1,1}E_{1,3}   & A_{1,1}E_{1,4} \\
    \0 & A_{2,2}  &  A_{2,2}E_{2,3}\beta_{2,3}   & A_{2,2}E_{2,4}\beta_{2,4} \\
    \0 & \0  & A_{3,3}   & A_{3,3}E_{3,4}\beta_{3,4} \\
    \0 & \0  &  \0   & A_{4,4} \\
    \end{pmatrix}.
$$
Notice that the latter equation can be summarised as: 
\begin{eqnarray}
A_{1,1} = \alpha_{1,2}A_{2,2} &=& \alpha_{1,3}A_{3,3} = A_{4,4} = A_{2,2} \beta_{2,4}= A_{3,3}\beta_{3,4},\label{(h1)}\\
 \alpha_{2,3}A_{3,3}&=&A_{2,2}\beta_{2,3},  \label{(h2)}
\end{eqnarray}
and hence that equations \eqref{(a)}-\eqref{(h2)} determine the $\Ht$-class of $E$.

Using the above equations it is now straightforward to analyse the required cases.

\noindent\textbf{1. If $\mathbf E$ is loose in all four paths:} By Proposition~\ref{prop:looseall}, $A \Rt E$ exactly if $A \R\ E$, and dually $A \Lt E$ if and only if $A \L E$. It then follows from Theorem~\ref{thm:upper*} that
$$\widetilde{R}_{E}=E D_4(\lin) =R_{E}; \;\; \widetilde{L}_{E}=D_4(\lin)E  =L_{E};\;\;\widetilde{H}_{E}=H_{E}.$$

\noindent\textbf{2. If $\mathbf E$ is tight in only $\mathbf{1\rightarrow 2\rightarrow 3}$:} It follows from the observations above that in order to satisfy equations \eqref{(b)} and \eqref{(c)} we must have $\alpha _{1,3}=\alpha _{2,3}=\1$. Since $\Def_E(1 \rightarrow 2 \rightarrow 3) = \1$, it then follows that \eqref{(a)} holds. Thus
$$\widetilde{R}_{E}=\left\{ [\alpha]_{1,2} \circ E:\1 \leq \alpha \right\}D_4(\lin).$$

Since the path $1\rightarrow 2 \rightarrow 3$ is tight, equation \eqref{(b)'} holds. In order to satisfy equation \eqref{(c)'} we must have $\beta _{2,4}=\1$. There are then two possibilities which lead to the satisfaction of \eqref{(a)'}; either $\beta_{3,4} =\1$ and $\beta_{2,3} \leq \Def_E(2 \rightarrow 3 \rightarrow 4)$, or $\beta_{2,3} = \Def_E(2 \rightarrow 3 \rightarrow 4)$. This gives  
\begin{eqnarray*}
\widetilde{L}_{E}=&&D_4(\lin)\left\{[\Def_E(2 \rightarrow 3 \rightarrow 4)]_{2,3}\circ [\alpha]_{3,4} \circ E : \1 \leq \alpha\right\}\\ &\cup& D_4(\lin)\left\{ [\alpha]_{2,3} \circ E : \1 \leq \alpha \leq \Def_E(2 \rightarrow 3 \rightarrow 4) \right\}.
\end{eqnarray*}

Finally, $A \in \widetilde{H}_{E}$ if and only if \eqref{(a)}-\eqref{(h2)} hold simultaneously.  This is the case if and only if $A_{1,1}=A_{2,2}=A_{3,3}=A_{4,4}$ and all other parameters are equal to $\1$. Hence  $\widetilde{H}_{E}=\{\lambda E: \lambda \in \lin\} =H_E$. Dually, if $E$ is tight in only $2\rightarrow 3\rightarrow 4$ then
\begin{eqnarray*}
\widetilde{R}_{E}=&&\left\{[\Def_E(1 \rightarrow 2 \rightarrow 3)]_{2,3}\circ [\alpha]_{1,2} \circ E : \1 \leq \alpha\right\}D_4(\lin)\\ &\cup& \left\{ [\alpha]_{2,3} \circ E : \1 \leq \alpha \leq \Def_E(1 \rightarrow 2 \rightarrow 3) \right\}D_4(\lin)\\
\widetilde{L}_{E}=&&D_4(\lin)\left\{ [\alpha]_{3,4} \circ E:\1 \leq \alpha \right\},\;\;\;\widetilde{H}_{E}=H_E.
\end{eqnarray*}

\noindent\textbf{3. If $\mathbf E$ is tight in only $\mathbf{1\rightarrow 2\rightarrow 4}$:} As in the previous case, in order to satisfy equations \eqref{(b)} and \eqref{(c)} we must have $\alpha _{1,3}=\alpha _{2,3}=\1$, whence \eqref{(a)} holds, giving
$$\widetilde{R}_{E}=\left\{ [\alpha]_{1,2} \circ E:\1 \leq \alpha \right\}D_4(\lin).$$

Since the path $1\rightarrow 2 \rightarrow 4$ is tight, equation \eqref{(c)'} holds. In order to satisfy equation \eqref{(b)'} we must have $\beta _{2,3}=\1$. Since $E$ is loose in the path $2 \rightarrow 3\rightarrow 4$, and $\beta_{2,4} \geq \1$ we note that $\1 < \Def_E(2 \rightarrow 3 \rightarrow 4)\beta_{2,4}$. Thus in order for \eqref{(a)'} to be satisfied we require $\beta_{3,4} =\1$, giving 
\begin{eqnarray*}
\widetilde{L}_{E}=&& D_4(\lin)\left\{ [\alpha]_{2,4} \circ E : \1 \leq \alpha \right\}.
\end{eqnarray*}

Now $A \in \widetilde{H}_{E}$ if and only if \eqref{(a)}-\eqref{(h2)} hold simultaneously, and this is the case if and only if $A_{1,1}=A_{2,2}=A_{3,3}=A_{4,4}$ and all other parameters are equal to $\1$. Thus  $\widetilde{H}_{E}=\{\lambda E: \lambda \in \lin\} =H_E$. Dually, if $E$ is tight in only $1\rightarrow 3\rightarrow 4$ then,
$$\widetilde{R}_{E}=\left\{ [\alpha]_{1,3} \circ E : \1 \leq \alpha\right\}D_4(\lin)\;\;\; \widetilde{L}_{E}=D_4(\lin)\left\{ [\alpha]_{3,4} \circ E:\1 \leq \alpha \right\}\;\;\; \widetilde{H}_{E}=H_E.$$

\noindent\textbf{4. If $\mathbf E$ is tight in only $\mathbf{1\rightarrow 3\rightarrow 4}$ and $\mathbf{2\rightarrow 3\rightarrow 4}$:} In this case \eqref{(b)} and \eqref{(c)} are satisfied. For \eqref{(a)} to hold we need  $\1 \leq \alpha_{2,3} \leq \Def_E(1 \rightarrow 2 \rightarrow 3)\alpha_{1,3}$ and either $\alpha_{1,2}=\1$ or $\alpha_{2,3}=\Def_E(1 \rightarrow 2 \rightarrow 3)\alpha_{1,3}$. Thus
\begin{eqnarray*}
\widetilde{R}_{E}&=&\left\{  [\alpha]_{1,3}\circ [\beta]_{2,3}\circ E:\1 \leq \alpha, \1 \leq \beta \leq \Def_E(1 \rightarrow 2 \rightarrow 3)\alpha \right\}D_4(\lin)\\
&\cup& \left\{ [\alpha]_{1,2} \circ [\beta]_{1,3}\circ [\Def_E(1 \rightarrow 2 \rightarrow 3)\beta]_{2,3}\circ E:\1 \leq \alpha, \beta \right\}D_4(\lin).
\end{eqnarray*}

For equations \eqref{(b)'} and \eqref{(c)'} to hold we must have $\beta _{2,3}=\beta_{2,4} =\1$. Since $E$ is tight in $1 \rightarrow 3 \rightarrow 4$, it then follows that \eqref{(a)'} is also satisfied, giving
\begin{eqnarray*}
\widetilde{L}_{E}=&&D_4(\lin)\left\{ [\alpha]_{3,4} \circ E : \1 \leq \alpha\right\}.\end{eqnarray*}

Combining our previous observations with \eqref{(h1)} and \eqref{(h2)} yields $A_{1,1}=A_{4,4}=A_{2,2}=\alpha_{2,3} A_{3,3}$, $\beta_{2,3}=\beta_{2,4}=\alpha_{1,2}= \1$ and $\alpha_{2,3}=\alpha_{1,3}=\beta_{3,4}$. Hence  $$\widetilde{H}_{E}=\{\lambda ([\mu]_{3,3}\circ E): \lambda, \mu \in \lin, \mu \leq \1\}.$$

Dually, if $E$ is tight only in $1\rightarrow 2\rightarrow 3$, $1\rightarrow 2\rightarrow 4$ then
\begin{eqnarray*}
\widetilde{R}_{E}=&& \left\{ [\alpha]_{3,4} \circ E : \1 \leq \alpha\right\}D_4(\lin) \\
\widetilde{L}_{E}=&& D_4(\lin)\left\{  [\alpha]_{2,4}\circ [\beta]_{2,3}\circ E:\1 \leq \alpha, \1 \leq \beta \leq \Def_E(2 \rightarrow 3 \rightarrow 4)\alpha \right\}\\
&\cup& D_4(\lin)\left\{ [\alpha]_{3,4} \circ [\beta]_{2,4}\circ [\Def_E(2 \rightarrow 3 \rightarrow 4)\beta]_{2,3}\circ E:\1 \leq \alpha, \beta \right\}\\
\widetilde{H}_{E}=&&\{\lambda ([\mu]_{2,2}\circ E): \lambda, \mu \in \lin, \mu \leq \1\}.
\end{eqnarray*}

\noindent\textbf{5. If $\mathbf E$ is tight only in $\mathbf{1\rightarrow 2\rightarrow 3}$, $\mathbf{2\rightarrow 3\rightarrow 4}$:} In this case \eqref{(c)} is satisfied, and for \eqref{(b)} to hold requires $\alpha_{1,3}=\1$. Since $E$ is tight in $1 \rightarrow 2 \rightarrow 3$, we find that the only way that \eqref{(a)} can hold is if $\alpha_{2,3}=\1$. Thus
\begin{eqnarray*}
\widetilde{R}_{E}&=&\left\{  [\alpha]_{1,2}\circ E:\1 \leq \alpha\right\}D_4(\lin).
\end{eqnarray*}

Similarly, \eqref{(b)'} is satisfied and for \eqref{(c)'} to hold we must have $\beta_{2,4} =\1$. Since $E$ is tight in $2 \rightarrow 3 \rightarrow 4$, it then follows that the only way that \eqref{(a)'} can hold is if $\beta_{2,3}=\1$. Thus
\begin{eqnarray*}
\widetilde{L}_{E}=&&D_4(\lin)\left\{ [\alpha]_{3,4} \circ E : \1 \leq \alpha\right\}.\end{eqnarray*}

Combining the previous observations with \eqref{(h1)} and \eqref{(h2)} then yields that $A_{1,1}=A_{4,4}=A_{3,3}= A_{4,4}$, and all other parameters are equal to $\1$, giving $\widetilde{H}_{E}=H_E.$

\noindent\textbf{6. If $\mathbf E$ is tight in only $\mathbf{1\rightarrow 2\rightarrow 4}$, $\mathbf{1\rightarrow 3\rightarrow 4}$:}  In this case \eqref{(b)} is satisfied, and for \eqref{(c)} to hold requires $\alpha_{2,3}=\1$. But then, since $E$ is tight in $1 \rightarrow 2 \rightarrow 4$, we see that \eqref{(a)} also holds, giving
\begin{eqnarray*}
\widetilde{R}_{E}&=&\left\{  [\alpha]_{1,2}\circ [\beta]_{1,3} E:\1 \leq \alpha, \beta \right\}D_4(\lin).
\end{eqnarray*}

Likewise, \eqref{(c)'} is satisfied and for \eqref{(b)'} to hold we must have $\beta_{2,3} =\1$. Since $E$ is tight in $2 \rightarrow 3 \rightarrow 4$, we then have that \eqref{(a)'} holds, giving
\begin{eqnarray*}
\widetilde{L}_{E}=&&D_4(\lin)\left\{ [\alpha]_{2,4} \circ [\beta]_{3,4} \circ E: \1 \leq \alpha, \beta \right\}.\end{eqnarray*}

Finally, $A \in \widetilde{H}_{E}$ if and only if \eqref{(a)}-\eqref{(h2)} hold simultaneously, which is the case if $A_{1,1}=A_{4,4}= \alpha_{12}A_{2,2}, A_{2,2}= A_{3,3}$, $\alpha_{1,2}=\alpha_{1,3}=\beta_{2,4}=\beta_{3,4}$ and $\alpha_{2,3}=\beta_{2,3} = \1$. It follows from this that
$$\widetilde{H}_{E}=\{\lambda([\mu]_{2,2}\circ [\mu]_{2,3}\circ [\mu]_{3,3}\circ E): \lambda,\mu, \in \lin, \mu \leq \1 \}.$$

\noindent\textbf{7. If $\mathbf E$ is tight in all four paths :} Equations \eqref{(b)}, \eqref{(c)}, \eqref{(b)'} and \eqref{(c)'} are satisfied. Since all paths are tight (and hence all deficiences are equal to $\1$) it is clear that \eqref{(a)} holds if and only if $\alpha_{2,3} \leq \alpha_{1,3}$, whilst \eqref{(a)'} holds if and only if $\beta_{2,3} \leq \beta_{2,4}$, giving
\begin{eqnarray*}
\widetilde{R}_{E}&=&\left\{  [\alpha]_{1,2}\circ [\beta]_{1,3} \circ  [\gamma]_{2,3} \circ E:\1 \leq \alpha, \1 \leq \gamma \leq \beta \right\}D_4(\lin)\\
\widetilde{L}_{E}&=&D_4(\lin)\left\{  [\alpha]_{2,3}\circ [\beta]_{2,4} \circ  [\gamma]_{3,4} \circ E:\1 \leq \gamma, \1 \leq \alpha \leq \beta \right\}.
\end{eqnarray*}
Combining these observations with \eqref{(h1)} and \eqref{(h2)} then yields 
$$\widetilde{H}_{E}=\{\lambda([\alpha]_{2,2}\circ [\beta]_{2,3}\circ [\gamma]_{3,3}\circ E): \lambda, \alpha, \beta, \gamma \in \lin, \alpha \vee \gamma \leq \beta \leq \1,  \}.$$

We now set out to confirm the fact stated earlier, that for any idempotent $E\in \upper{4}{\lin}$ we have that $\widetilde{H}_E$ is a subsemigroup.
\begin{proposition}\label{prop:htilden=4} For any idempotent 
$E\in \upper{4}{\lin}$, the $\Ht$-class $\widetilde{H}_E$ of $E$  is a subsemigroup.
\end{proposition}

\begin{proof}
Let $E \in \upper{4}{\lin}$  such that $E^{2}=E$. From Table \ref{table:n=4RtandLt}, there are four different patterns for $\widetilde{H}_E$,  depending on the tightness patterns of $E$ in paths of length $2$. In the case where $\widetilde{H}_{E}= H_{E}$, clearly $\widetilde{H}_{E}$ is a monoid (it is a maximal subgroup). In each of the remaining three cases, the elements of $\widetilde{H}_{E}$ have the form $\lambda([\alpha]_{2,2}\circ [\beta]_{2,3}\circ [\gamma]_{3,3} \circ E)$, where certain restrictions are placed on the parameters $\alpha, \beta, \gamma$. Noting that
$$\lambda([\alpha]_{2,2}\circ [\beta]_{2,3}\circ [\gamma]_{3,3} \circ E) \cdot \lambda'([\alpha']_{2,2}\circ [\beta']_{2,3}\circ [\gamma']_{3,3} \circ E) = \lambda\lambda'([\alpha\alpha']_{2,2}\circ [x]_{2,3}\circ [\gamma\gamma']_{3,3} \circ E),$$
where $x=\alpha\beta' \vee \beta\gamma'$, it then suffices to show that in each case the restrictions placed on the parameters are preserved by this product.

\noindent\textbf{Case 4. If $\mathbf E$ is tight in only $\mathbf{1\rightarrow 3\rightarrow 4}$, $\mathbf{2\rightarrow 3\rightarrow 4}$:}\\
 In this case we have $\alpha=\beta=\1$ and $\gamma \leq \1$. Since for $\gamma,\gamma' \leq \1$, we have $\gamma\gamma'\leq \1$, it follows from the form of the product above that, $\widetilde{H}_{E} \cong \lin\times \lin_{\leq \1}$.

\noindent\textbf{Case 6. If $\mathbf E$ is tight in only $\mathbf{1\rightarrow 2\rightarrow 4}$, $\mathbf{1\rightarrow 3\rightarrow 4}$:}\\
In this case we have $\alpha=\beta=\gamma \leq \1$, and as before, since $\gamma,\gamma'\leq \1$ implies that $\gamma\gamma'\leq \1$, it follows that $\widetilde{H}_{E} \cong \lin\times \lin_{\leq \1}$.

\noindent\textbf{Case 7. If $\mathbf E$ is tight in all four paths:}\\
In this case we have $\alpha \vee \gamma \leq \beta \leq\1$. Since
$\alpha' \leq \alpha' \vee \gamma' \leq \beta'$ and $\gamma \leq \alpha \vee \gamma \leq \beta$ we have
$$\alpha\alpha' \vee \gamma\gamma' \leq \alpha\beta' \vee \beta\gamma' = x,$$
and hence $\widetilde{H}_{E}$ is closed under multiplication. Indeed, by the previous lemma we have $\widetilde{H}_{E} \cong \lin\times \mathcal{O}_2([\0,\1]) $. In general this semigroup, it is not (unlike in the previous cases) commutative. 
For example, taking $\lin=\mathbb{R}$, and tuples $(\alpha_1,\beta_1,\gamma_1)=(0,1,2)$ and 
$(\alpha_2,\beta_2,\gamma_2)=(1,3,1)$, we see that 
$(\alpha_1+\beta_2)\oplus(\beta_1+\gamma_2)=3$ but
$(\alpha_2+\beta_1)\oplus(\beta_2+\gamma_1)=4$, so that $\widetilde{H}_{E}$ is not commutative.
\end{proof}

\subsection{More complex behaviour persists as $\mathbf n$ increases.} \label{subsec:expansion}

Let 
$ UT_{n}^{1,n}(\lin)$ denote the set of all matrices in $\upper{n}{\lin}$ with $\1$ at $(1,1)$ and $(n,n)$ positions. Then $ UT_{n}^{(1,n)}(\lin)$ is a subsemigroup of $\upper{n}{\lin}$ and $\uni{n}{\lin} \subseteq UT_{n}^{(1,n)}(\lin)$. Thus $UT_{n}^{(1,n)}(\lin)$ is a Fountain semigroup.
In order to use facts concerning the behaviour of Green's $\sim$-relations in  $\upper{n}{\lin}$ to inform their behaviour in $\upper{n+1}{\lin}$, we first show that 
$ UT_{n}^{1,n}(\lin)$  may be embedded in $ UT_{n+1}^{1,n+1}(\lin)$ in a way that preserves the unary operations of $^{(+)}$ and $^{(*)}$, as well as the semigroup multiplication\footnote{Formally, the embedding is of algebras of signature type $(2,1,1)$.}. 
 
\begin{definition}\label{defn:embed} Let $n\in \mathbb{N}$. We define $\theta_n: \upper{n}{\lin}
\rightarrow\upper{n+1}{\lin}$ by the rule \[A\mapsto \left( \begin{array} {c|c}
A & A_{*,n} \\ \hline
\0_{1 \times n} & \1
\end{array} \right).\]
\end{definition}

\noindent Clearly $\theta_{n}$ is an injective map;  $\theta_{n}$ is particularly well behaved when restricted to 
$UT_{n}^{(1,n)}(\lin)$.

\begin{lemma} \label{lemma:ntheta}
For all $n \in \mathbb{N}$ the map 
\[\theta_n \mid_{UT_{n}^{(1,n)}(\lin)}:UT_{n}^{(1,n)}(\lin)\rightarrow
UT_{n+1}^{(1,n+1)}(\lin) \] is an $\Rt$- and $\Lt$-preserving homomorphism.
Consequently, $\theta_n$ maps idempotents to idempotents.
\end{lemma}
\begin{proof} Let $A,B \in UT_{n}^{(1,n)}$. If $1 \leq i \leq k \leq j\leq n$ then $\theta_{n}(A)_{i,k}\theta_{n}(B)_{k,j}= A_{i,k} B_{j,k}$. So $$(\theta_{n}(A)\theta_{n}(B))_{i,j}=\bigvee_{i\leq k \leq j} \theta_{n}(A)_{i,k}\theta_{n}(B)_{k,j} =\bigvee_{i\leq k \leq j} A_{i,k} B_{k,j}=(AB)_{i,j}=\theta_{n}(AB)_{i,j}.$$
It remains to check the case where $i \leq j=n+1$. Here we have
\[
\begin{array}{rcl}
(\theta_{n}(A)\theta_{n}(B))_{i,n+1} &=& \bigvee_{i\leq k \leq n+1} \theta_{n}(A)_{i,k}\theta_{n}(B)_{k,n+1}\\
&=& \left(\bigvee_{i\leq k \leq n} A_{i,k} B_{k,n} \right) \vee (\theta_n(A))_{i,n+1} 
(\theta_n(B))_{n+1,n+1}\\
&=&  \left(\bigvee_{i\leq k \leq n} A_{i,k} B_{k,n} \right) \vee A_{i,n} \1\\
&=& \left(\bigvee_{i\leq k \leq n} A_{i,k} B_{k,n} \right)\\
&=&(AB)_{i,n}
\end{array} \]
where the last step uses  the fact that 
$A_{i,n}=A_{i,n}\1 =A_{i,n}B_{n,n}$.  
Thus $\theta_{n}(AB)=(\theta_{n}(A)\theta_{n}(B))$.

To see that $\theta_n$ preserves  $\Rt$  and $\Lt$, we show that $\rid{\theta_{n}(A)}=\theta_{n}(\rid A)$ and $\lid{\theta_{n}(A)}=\theta_{n}(\lid A)$ for all $A \in UT_{n}^{(1,n)}(\lin)$. 

For
 $i\leq j<n+1$,
$$(\rid {\theta_{n}(A)})_{i,j} = \bigwedge_{1\leq k \leq i} \theta_{n}(A)_{k,j}(\theta_{n}(A_{k,i}))^{-1} = \bigwedge_{1\leq k \leq i} A_{k,j}(A_{k,i})^{-1} =\rid{A}_{i,j}. $$
On the other hand, for 
 $i<n+1$,
 $$(\rid {\theta_{n}(A)})_{i,n+1} = \bigwedge_{1\leq k \leq i} \theta_{n}(A)_{k,n+1}(\theta_{n}(A_{k,i}))^{-1} = \bigwedge_{1\leq k \leq i} A_{k,n}(A_{k,i})^{-1} =\rid{A}_{i,n}. $$
 Finally,
for $i=n+1, j=n+1$, we have $(\rid {\theta_{n}(A)})_{n+1,n+1}= \1=( {\theta_{n}(\rid A)})_{n+1,n+1}$. Hence $\rid{\theta_{n}(A)}=\theta_{n}(\rid A)$.

Turning our attention to $^{(+)}$, we have that if $j<n+1$ then
\[
\begin{array}{rcl}
(\lid {\theta_{n}(A)})_{i,j}&=&\bigwedge_{j\leq k \leq n+1}  \theta_{n}(A)_{i,k}(\theta_{n}(A_{j,k}))^{-1}  \\
&=&\left( \bigwedge_{j\leq k\leq n}\theta _{n}(A)_{i,k}(\theta_{n}(A_{j,k}))^{-1}\right) \wedge \theta _{n}(A)_{i,n+1}(\theta_{n}(A)_{j,n+1})^{-1}\\
&=&\left( \bigwedge_{j\leq k\leq n}A_{i,k}(A_{j,k})^{-1}\right) \wedge A_{i,n}(A_{j,n})^{-1}\\ &=&\bigwedge_{j\leq k \leq n} A_{i,k}(A_{j,k})^{-1}  =\lid{A}_{i,j}.
\end{array} \]
For $j=n+1$, recalling that  
$(\lid A)_{i,n}=A_{i,n}(A_{n,n})^{-1}=A_{i,n}$ as  $A_{n,n}=\1$, we have
\[(\lid {\theta_{n}(A)})_{i,n+1}=\theta_{n}(A)_{i,n+1}(\theta_{n}(A)_{n+1,n+1})^{-1}=A_{i,n}=(\lid A)_{i,n}.\] Hence $\lid{\theta_{n}(A)}=\theta_{n}(\lid A)$. 
\end{proof}

Our first use of Lemma~\ref{lemma:ntheta} is to show that the analogue of
Proposition~\ref{prop:commute} does not hold for $n\geq 4$.

\begin{proposition} \label{prop:RtLtnotcommut} Let  $n\geq 4$, and suppose that $\lin$ is non-trivial. Then

(i) there is an
$A\in  \upper{n}{\lin}$ such that $\rid A$ and $\lid A$ are not $\D$-related;

(ii) the relations
$\Rt$ and $\Lt$ do not commute in $\upper{n}{\lin}$. 
\end{proposition}
\begin{proof} $(i)$  We begin with $n=4$. Let $g \in \lin$ with $g >\1$ and let
\[A=
 \begin{pmatrix}
 \1  &   g   &  \1 &  g^2  \\
 \0  &  \1   &  g   &  g   \\
 \0  &  \0   &  \1  &  \1   \\
 \0  &  \0   &  \0  &  \1   \\
\end{pmatrix} \mbox{ so that } \lid A=
\begin{pmatrix}
 \1  & g^{-1}&  \1 &  g^2  \\
 \0  &  \1   &  g   &  g   \\
 \0  &  \0   &  \1  &  \1   \\
 \0  &  \0   &  \0  &  \1   \\
\end{pmatrix} \mbox{ and }\rid A=
\begin{pmatrix}
 \1  &   g  &  \1  &  g^2  \\
 \0  &  \1  &g^{-1}&  g   \\
 \0  &  \0  &  \1  &  \1   \\
 \0  &  \0  &  \0  &  \1   \\
\end{pmatrix}.\]  Notice that $\Def_{\lid A}(1\rightarrow 2\rightarrow 4)= g^2(g^{-1}g)^{-1}=g^2$ whereas $\Def_{\rid A}(1\rightarrow 2\rightarrow 4)=g^2(gg)^{-1}=\1$. By Theorem \ref{thm:uniDdeficiency}, $\lid A$ and $\rid A$ are not $\D$-related in $\upper{4}{\lin}$.

 By using Lemma \ref{lemma:ntheta}, for any $n\geq 5$, setting
 $\phi_{n}:=\theta_{n-1}\theta_{n-2}\cdots \theta_4$ we have that
 $\phi_n(\lid A)=(\phi_n(A))^{+}$ and  $\phi_n(\rid A)=(\phi_n(A))^{*}$; but the same argument using deficiencies as in the case $n=4$ gives that  $(\phi_n(A))^{+}$
 and $(\phi_n(A))^{*}$ are not $\D$-related in  $\upper{n}{\lin}$.

Notice that this shows that Proposition~\ref{prop:commute} (i) and (iii) no longer holds for $3<n$.
Also note that, $\phi_n(\lid{A}) \Rt \circ \Lt \phi_n(\rid{A})$ in $\upper{n}{\lin}$ for $n\geq 4$. If there exists a matrix $B\in \upper{n}{\lin}$ such that $\phi_n(\lid{A}) \Lt B  \Rt \phi_n(\rid{A})$ then by Theorem \ref{thm:upper*} $\rid {(\lnorm B)}=\rid B=\phi_n(\lid{A})$ and $\lid {(\rnorm B)}=\lid B=\phi_n(\rid{A})$. By formulae given in \ref{Ralpha*} and \ref{Lbeta*},
\[ \lnorm B= \left(
\begin{tabular}{llll|lll}
$\1$ & $g^{-1}$  & $\1$ & $g^{2}$ & $g^{2}$ &$\hdots$ & $g^{2}$ \\
$\0$ & $\1$ & $g \beta_{2,3}$  & $g\beta_{2,4}$ & $g\beta_{2,5}$ & $\hdots$ & $g\beta_{2,n}$ \\
$\0$ & $\0$ & $\1$ & $\beta_{3,4}$ & $\beta_{3,5}$ & $\hdots$ & $\beta_{3,n}$ \\
$\0$ & $\0$ & $\0$ & $\1$ & $\beta_{4,5}$ & $\hdots$ & $\beta_{4,n}$ \\\hline
$\vdots$ & $\vdots $ & $\vdots $ & $\vdots $ & $\vdots $ & $\ddots$ & $\vdots$ \\
$\0$ & $\0$ & $\0$ & $\0$ & $\hdots$ & $\0$ & $\1$
\end{tabular} \right)\]
and
\[ \rnorm B=
\left(
\begin{tabular}{llll|lll}
$\1$ & $g\alpha_{1,2}$  & $\alpha_{1,3}$ & $g^{2}\alpha_{1,4}$ &$\hdots$ & $g^{2}\alpha_{1,n-1}$ & $g^{2}$ \\
$\0$ & $\1$ & $g^{-1} \alpha_{2,3}$  & $g\alpha_{2,4}$ & $\hdots$ & $g\alpha_{2,n-1}$ & $g$ \\
$\0$ & $\0$ & $\1$ & $\alpha_{3,4}$ & $\hdots$ & $\alpha_{3,n-1}$& $\alpha_{3,n}$ \\
$\0$ & $\0$ & $\0$ & $\1$ & $\hdots$ & $\alpha_{4,n-1}$ & $\alpha_{4,n}$ \\\hline
$\vdots$ & $\vdots $ & $\vdots $ & $\vdots $ & $\vdots $ & $\ddots$ & $\vdots$ \\
$\0$ & $\0$ & $\0$ & $\0$ & $\hdots$ & $\0$ & $\1$
\end{tabular} \right)\]
where $\1 \leq \alpha_{i,j}, \beta_{i,j}$. By comparing at positions $(2,j)$, where $4 \leq j \leq n$, we have
$$  g= \phi_n(\rid{A})_{2,j} = \rid {(\lnorm B)}_{2,j}= \lnorm B_{1,j} (\lnorm B_{1,2})^{-1} \wedge \lnorm B_{2,j} =  g^2 (g^{-1})^{-1} \wedge  g\beta_{2,j}= g\beta_{2,j}.$$
Hence $\beta_{2,j}=\1$ for all $4 \leq j \leq n$.
Now $B=\diag B \lnorm B=\rnorm B \diag B$ and then by comparison at positions $(1,n)\, (2,n)$, we must have $B_{1,1}g^2=B_{1,n}=B_{n,n}g^2$, \, $gB_{2,2}=B_{2,n}=B_{n,n}g$ respectively. So $B_{1,1}=B_{n,n}=B_{2,2}$. Again comparison at positions $(1,2)$ yields $B_{1,1} g^{-1}=B_{1,2}=B_{2,2} g \alpha_{1,2}$. So that, $\alpha_{1,2}=g^{-2} < \1$, which is a contradiction!
Thus our supposition is wrong and no such $B$ exists. So $\Rt \circ \ \Lt \neq \Lt \circ \ \Rt$ in $\upper{n}{\lin}$ where $n\geq 4$.
\end{proof}

\begin{proposition}\label{prop:ht}
We know from Proposition \ref{prop:htilden=4} that $\widetilde{H}_E$ is a monoid for each idempotent $E \in \upper{4}{\lin}$. For $n\geq 5$, however, this is not the case, as the following example illustrates.
\end{proposition}

\begin{proof}
Let $\1 < g \in \lin$ and $E,A\in \upper{5}{\lin}$ be such that
$$ E=
\begin{pmatrix}
\1 & \1 & g^2 & g^2 & g^2 \\
\0 & \1 & g   & g^2 & g^2 \\
\0 & \0 & \1  & \1  & \1  \\
\0 & \0 & \0  & \1  & \1  \\
\0 & \0 & \0  & \0  & \1  \\
\end{pmatrix},\ \mbox{ and }\ A=
\begin{pmatrix}
\1 & \1    & g^2    & g^2 & g^2 \\
\0 & g^{-2}& g^{-1} & g & g^2 \\
\0 & \0 & g^{-3}  & \1  & \1  \\
\0 & \0 & \0  & g^{-3}  & \1  \\
\0 & \0 & \0  & \0  & \1  \\
\end{pmatrix}.$$
Then it is straightforward to verify that $\lid A=E=\rid A$ and hence $A \Ht\ E$. But
$$ A^{2}=
\begin{pmatrix}
\1 & \1    & g^2    & g^2 & g^2 \\
\0 & g^{-4}& g^{-3} & g^{-1} & g^2 \\
\0 & \0 & g^{-6}  & g^{-3}  & \1  \\
\0 & \0 & \0  & g^{-6}  & \1  \\
\0 & \0 & \0  & \0  & \1  \\
\end{pmatrix}\, \mbox{ and }\ (A^2)^{(+)}=
\begin{pmatrix}
\1 & \1 & g^2 & g^2 & g^2 \\
\0 & \1 & g^2 & g^2 & g^2 \\
\0 & \0 & \1  & \1  & \1  \\
\0 & \0 & \0  & \1  & \1  \\
\0 & \0 & \0  & \0  & \1  \\
\end{pmatrix}.$$

Clearly $(A^2)^{(+)} \neq \lid A$ and by uniqueness of $(A^2)^{(+)}$ it follows that $A^2 \notin \widetilde{H}_{E}$\ and hence $\widetilde{H}_{E}$\ is not a subsemigroup of $\upper{5}{\lin}$.

We can extend this example to any $n > 5$ by using the map $\phi_n = \theta_{n-1} \ldots \theta_5$. Since $A \in \widetilde{H}_{E}$ and $A \in UT_{5}^{(1,5)}(\lin)$, Lemma \ref{lemma:ntheta} yields that $\phi_n(A) \in \widetilde{H}_{\phi_n(E)}$. Now $A^2 \in UT_{5}^{(1,5)}(\lin)$, and so $\lid{\phi_n((A^2))}= \phi_n(\lid{(A^2)})$ and then by the injectivity of the maps $\theta_k$ it follows that $\phi_n(\lid{(A^2)})$ is not equal to $\phi_n(E)$. Hence $\phi_n(A^2) \notin \widetilde{H}_{\phi_n(E)}$ and $\widetilde{H}_{\phi_n(E)}$ is not a subsemigroup of $\upper{n}{\lin}$ for all $n > 5$.
\end{proof}

The above considerations prompt us to end with the following.

\begin{question} Determine the $\Dt$-classes of $UT_n(\linz^*)$ for $n\geq 3$. In particular, are they determined by tightness and looseness properties of idempotents?
\end{question}
\bibliography{arXiv_submission.bib}
\bibliographystyle{plain}
\end{document}